\PassOptionsToPackage{dvipsnames}{xcolor}
\documentclass[twoside]{article}
\usepackage{amsmath,amsthm,amsfonts}
\usepackage{todonotes}
\usepackage{mathtools}
\usepackage{latexsym}
\usepackage{tikz,tikz-cd}
\usepackage{comment}
\usepackage{tikzit}			

\tikzstyle{morphism}=[fill=white, draw=black, shape=rectangle]
\tikzstyle{generic morphism}=[fill=white, draw=black, shape=rectangle,dashed]
\tikzstyle{medium box}=[fill=white, draw=black, shape=rectangle, minimum width=0.7cm, minimum height=0.7cm]
\tikzstyle{large morphism}=[fill=white, draw=black, shape=rectangle, minimum width=1.7cm, minimum height=1cm]
\tikzstyle{bn}=[fill=black, draw=black, shape=circle, inner sep=1.5pt]
\tikzstyle{bw}=[fill=white, draw=black, shape=circle, inner sep=1.5pt]
\tikzstyle{bin}=[fill=white, draw=black, shape=circle, inner sep=0pt]
\tikzstyle{effect}=[fill=white, draw=black, regular polygon, regular polygon sides=3, minimum width=0.8cm, inner sep=0pt]
\tikzstyle{state}=[fill=white, draw=black, regular polygon, regular polygon sides=3, minimum width=0.8cm, shape border rotate=180, inner sep=0pt]
\tikzstyle{medium state}=[fill=white, draw=black, regular polygon, regular polygon sides=3, minimum width=1.3cm, inner sep=0pt, shape border rotate=180]
\tikzstyle{large state}=[fill=white, draw=black, regular polygon, regular polygon sides=3, minimum width=2.2cm, shape border rotate=180, inner sep=0pt]
\tikzstyle{wide state}=[fill=white, draw=black, shape=isosceles triangle, minimum width=0.8cm, shape border rotate=270, inner sep=1.4pt, minimum height=0.5cm, isosceles triangle apex angle=80]
\tikzstyle{wn}=[fill=white, draw=black, shape=circle, inner sep=1.5pt]
\tikzstyle{blue morphism}=[fill=white, draw={rgb,255: red,15; green,0; blue,150}, shape=rectangle, text={rgb,255: red,15; green,0; blue,150}, tikzit category=blue]
\tikzstyle{red morphism}=[fill=white, draw={rgb,255: red,150; green,0; blue,2}, shape=rectangle, text={rgb,255: red,150; green,0; blue,2}, tikzit category=red]
\tikzstyle{blue state}=[fill=white, draw={rgb,255: red,15; green,0; blue,150}, shape=circle, regular polygon, regular polygon sides=3, minimum width=0.8cm, shape border rotate=180, inner sep=0pt, text={rgb,255: red,15; green,0; blue,150}, tikzit category=blue]
\tikzstyle{blue node}=[fill={rgb,255: red,15; green,0; blue,150}, draw={rgb,255: red,15; green,0; blue,150}, shape=circle, tikzit category=blue, inner sep=1.5pt]
\tikzstyle{blue}=[text={rgb,255: red,15; green,0; blue,150}, tikzit draw={rgb,255: red,191; green,191; blue,191}, tikzit category=blue, tikzit fill=white, inner sep=0mm]
\tikzstyle{blue wide state}=[fill=white, draw={rgb,255: red,15; green,0; blue,150}, text={rgb,255: red,15; green,0; blue,150}, shape=isosceles triangle, minimum width=0.8cm, shape border rotate=270, inner sep=1.4pt, minimum height=0.5cm, isosceles triangle apex angle=80]
\tikzstyle{red node}=[fill={rgb,255: red,150; green,0; blue,2}, draw={rgb,255: red,150; green,0; blue,2}, shape=circle, inner sep=1.5pt]
\tikzstyle{Purple node}=[fill={rgb,255: red,120; green,0; blue,120}, draw={rgb,255: red,120; green,0; blue,120}, text={rgb,255: red,120; green,0; blue,120}, shape=circle, inner sep=1.5pt]
\tikzstyle{red}=[text={rgb,255: red,150; green,0; blue,2}, inner sep=0mm, tikzit fill=white, tikzit draw={rgb,255: red,191; green,191; blue,191}]
\tikzstyle{purple}=[text={rgb,255: red,150; green,0; blue,150}, inner sep=0mm, tikzit fill=white, tikzit draw={rgb,255: red,191; green,191; blue,191}]
\tikzstyle{white morphism}=[fill=white, draw=white, shape=rectangle, tikzit draw={rgb,255: red,139; green,139; blue,139}]
\tikzstyle{leak morphism}=[fill=white, draw={rgb,255: red,120; green,0; blue,85}, shape=rectangle, text={rgb,255: red,120; green,0; blue,85}, tikzit category=leak]
\tikzstyle{leak}=[text={rgb,255: red,120; green,0; blue,85}, inner sep=0mm, tikzit fill=white, tikzit draw={rgb,255: red,191; green,191; blue,191}, tikzit category=leak]
\tikzstyle{leak node}=[fill={rgb,255: red,120; green,0; blue,85}, draw={rgb,255: red,120; green,0; blue,85}, shape=circle, inner sep=1.5pt, tikzit category=leak]
\tikzstyle{horiz state}=[fill=white, draw=black, regular polygon, regular polygon sides=3, minimum width=1cm, shape border rotate=90, inner sep=0pt]

\tikzstyle{arrow}=[->]
\tikzstyle{dashed box}=[-, dashed]
\tikzstyle{blue arrow}=[-, draw={rgb,255: red,15; green,0; blue,150}, tikzit category=blue]
\tikzstyle{red arrow}=[-, draw={rgb,255: red,150; green,0; blue,2}, tikzit category=red]
\tikzstyle{purple arrow}=[->, draw={rgb,255: red,120; green,0; blue,120}, >=stealth, shorten <=2pt, shorten >=2pt]
\tikzstyle{protected purple arrow}=[->, draw={rgb,255: red,120; green,0; blue,120}, >=stealth, shorten <=2pt, shorten >=2pt, preaction={line width=1.8pt, white, draw}]
\tikzstyle{mapsto}=[{|->}]
\tikzstyle{double wire}=[-, double]
\tikzstyle{curly brace}=[-, draw=none, tikzit draw={rgb,255: red,128; green,0; blue,128}]
\tikzstyle{protected}=[-, preaction={line width=1.8pt,white,draw}]
\tikzstyle{leak arrow}=[-, tikzit draw={rgb,255: red,150; green,0; blue,120}]
\tikzstyle{protected leak arrow}=[-, tikzit draw={rgb,255: red,150; green,0; blue,120}]
\tikzstyle{hollow arrow}=[-, very thin, white, preaction={line width=0.7pt,draw={rgb,255: red,120; green,0; blue,85}}, tikzit category=leak, tikzit draw={rgb,255: red,150; green,0; blue,120}]
\tikzstyle{protected hollow arrow}=[-, very thin, white, preaction={line width=0.7pt,draw={rgb,255: red,120; green,0; blue,85},preaction={line width=2.1pt,white,draw}}, tikzit category=leak, tikzit draw={rgb,255: red,150; green,0; blue,120}]
\tikzstyle{over arrow}=[-, black, preaction={draw=white, double}]
\tikzstyle{curly brace}=[-, decorate, decoration={brace,amplitude=5pt}]
\usetikzlibrary{decorations.pathreplacing,decorations.pathmorphing}	
\usepackage{enumitem}
	\setlist[enumerate]{label=(\roman*)}  
	\setlist[enumerate,2]{label=(\alph*)}  
\usepackage[T1,T2A]{fontenc}			
\usepackage[noadjust]{cite}
\usepackage{tikz-cd}
\usepackage{amssymb,xparse}		
\usepackage{scalerel}
\usepackage{geometry}
\usepackage{bm}				
\usepackage{authblk}
\usepackage{thmtools}

\definecolor{myurlcolor}{rgb}{0,0,0.3}
\definecolor{mycitecolor}{rgb}{0,0.3,0}
\definecolor{myrefcolor}{rgb}{0.3,0,0}
\usepackage[pagebackref,draft=false]{hyperref}
\hypersetup{colorlinks,
	linkcolor=myrefcolor,
	citecolor=mycitecolor,
urlcolor=myurlcolor}

\usepackage[noabbrev,capitalize]{cleveref}

\let\originalleft\left
\let\originalright\right
\renewcommand{\left}{\mathopen{}\mathclose\bgroup\originalleft}
\renewcommand{\right}{\aftergroup\egroup\originalright}

\newtheorem{theorem}{Theorem}[section]
\newtheorem{proposition}[theorem]{Proposition}
\newtheorem{lemma}[theorem]{Lemma}
\newtheorem{corollary}[theorem]{Corollary}

\theoremstyle{definition}
\newtheorem{definition}[theorem]{Definition}
\newtheorem{notation}[theorem]{Notation}\Crefname{notation}{Notation}{Notations}
\newtheorem{example}[theorem]{Example}

\newtheorem{remark}[theorem]{Remark}

\numberwithin{equation}{section}

\newcommand{\newterm}[1]{\textbf{\textit{#1}}}	
\newcommand{\N}{\mathbb{N}}
\newcommand{\Z}{\mathbb{Z}}
\newcommand{\R}{\mathbb{R}}

\newcommand{\norm}[1]{\left\lVert #1 \right\rVert} 
\newcommand{\abs}[1]{\left\lvert #1 \right\rvert} 
\newcommand{\spec}[1]{\mathrm{sp}\left(#1\right)}
\renewcommand{\emptyset}{\varnothing} 

\newcommand{\op}{\mathrm{op}}
\newcommand{\cat}[1]{{\mathsf{#1}}} 
\newcommand{\invo}[1]{\overline{#1}}
\newcommand{\involution}{\mathfrak{I}}

\newcommand{\id}{\mathrm{id}} 		
\newcommand{\tensor}{\odot}
\newcommand{\bigtensor}{\bigodot}
\DeclareMathOperator*{\colim}{colim}
\DeclareMathOperator{\im}{\operatorname{im}}

\tikzset{pullback/.style={minimum size=1.2ex,path picture={	
			\draw[opacity=1,black,-,#1] (-0.5ex,-0.5ex) -- (0.5ex,-0.5ex) -- (0.5ex,0.5ex);%
}}}
\newcommand{\comp}{ 		
	\mathchoice{\,}{\,}{}{} 	
}

\DeclareMathOperator{\swap}{swap}
\DeclareMathOperator{\cop}{copy}
\DeclareMathOperator{\discard}{del}
\newcommand{\cC}{\mathsf{C}}		
\newcommand{\cD}{\mathsf{D}}		
\renewcommand{\det}{\mathrm{det}}	
\newcommand{\pairing}[2]{(#1,#2)} 
\newcommand{\samp}{\mathsf{samp}}	
\newcommand{\ev}{\mathsf{ev}}		

\newcommand{\finstoch}{\mathsf{FinStoch}}

\newcommand{\chaus}{\mathsf{CHaus}}

\newcommand{\cofree}{\mathsf{Cof}}
\newcommand{\classical}[1]{C\ell(#1)}
\newcommand{\gen}[1]{\mathsf{Gen}(#1)}
\newcommand{\iMod}[1]{#1\operatorname{-}\hspace{-0.3ex}\mathsf{iMod}}
\newcommand{\close}[1]{\operatorname{cl}({#1})}
\newcommand{\closed}[1]{\widehat{#1}}
\newcommand{\cpu}{\mathsf{pCPU}}
\newcommand{\calgone}{\mathsf{pC}^*\mathsf{ulin}^{\op}}
\newcommand{\calgmin}{\calgone_{\min}}
\newcommand{\calgmax}{\calgone_{\max}}

\newcommand{\as}[1]{
	\def\relstate{#1}%
	\ifx\relstate\empty
		\text{a.s.}%
	\else
		{#1\text{-a.s.}}%
	\fi
}

\newcommand{\ase}[1]{\simeq_{#1}}
\newcommand{\asel}[1]{\overset{L}{\simeq}_{#1}}
\newcommand{\aser}[1]{\overset{R}{\simeq}_{#1}}
\newcommand{\asets}[1]{\cong_{#1}} 

\newcommand{\sharpc}{\natural}

\newcommand{\suppincop}{\mathsf{si}}	
\newcommand{\suppinc}[1]{
	\def\relstate{#1}%
	\ifx\relstate\empty
		\mathsf{si}%
	\else
		\suppincop_{#1}%
	\fi
}
\newcommand{\Suppop}{\mathsf{Supp}}
\newcommand{\Supp}[1]{
	\def\relstate{#1}%
	\ifx\relstate\empty
		S%
	\else
		\Suppop_{#1}%
	\fi
}
\newcommand{\suppprojop}{\mathsf{sp}}
\newcommand{\suppproj}[1]{
	\def\relstate{#1}%
	\ifx\relstate\empty
		\mathsf{sp}%
	\else
		\suppprojop_{#1}%
	\fi
}

\newcommand{\llsim}{
	\mathrel{\ooalign{%
			\raise.45ex\hbox{$\scaleobj{0.8}{\ll}$}\cr%
	\hidewidth\raise-.4ex\hbox{$\vstretch{0.7}{\sim}$}\hidewidth\cr}}%
}
\newcommand{\ggsim}{%
	\mathrel{\ooalign{%
			\raise.45ex\hbox{$\scaleobj{0.8}{\gg}$}\cr%
	\hidewidth\raise-.4ex\hbox{$\vstretch{0.7}{\sim}$}\hidewidth\cr}}%
}

\newcommand{\pre}[1]{#1}
\newcommand{\opm}[2]{
	\def\relstate{#2}%
	\ifx\relstate\empty
		#1^{\op}%
	\else
		#1^{#2,\op}%
	\fi
}

\providecommand{\given}{}			
\newcommand{\SetSymbol}[1][]{%
	\nonscript\;\,#1\vert
	\allowbreak
	\nonscript\;\,
	\mathopen{}
}
\DeclarePairedDelimiterX{\Set}[1]{\{}{\}}{%
	\renewcommand{\given}{\SetSymbol[\delimsize]}
	#1
}
\makeatletter
\let\oldSet\Set
\def\Set{\@ifstar{\oldSet}{\oldSet*}}
\makeatother

\DeclarePairedDelimiterX{\Family}[1]{(}{)}{%
	\renewcommand{\given}{\SetSymbol[\delimsize]} 
	#1
}
\makeatletter
\let\oldFamily\Family
\def\Family{\@ifstar{\oldFamily}{\oldFamily*}}

\newsavebox{\numbox}%
\newsavebox{\slashbox}%
\newsavebox{\denbox}%
\newlength{\slashlength}%
\newlength{\faktorscale}%
\DeclareDocumentCommand{\newfaktor}{m O{0.35} m O{-0.35}}{
	\savebox{\numbox}{\ensuremath{#1}}
	\savebox{\slashbox}{\ensuremath{\diagup}}
	\savebox{\denbox}{\ensuremath{#3}}
	\setlength{\faktorscale}{0.5\ht\numbox+0.5\ht\denbox}%
	\setlength{\slashlength}{2pt+0.8\faktorscale+#2\faktorscale-#4\faktorscale}%
	\raisebox{#2\ht\slashbox}{\usebox{\numbox}}
	\mkern-2mu%
	\rotatebox{-30}{\rule[#4\ht\denbox]{0.4pt}{\slashlength}}
	\mkern9mu%
	\hspace{-0.44\slashlength}%
	\raisebox{#4\ht\denbox}{\usebox{\denbox}}
}
\DeclareDocumentCommand{\linefaktor}{m O{0.08} m O{-0.08}}{
	\savebox{\numbox}{\ensuremath{#1}}
	\savebox{\slashbox}{\ensuremath{\diagup}}
	\savebox{\denbox}{\ensuremath{#3}}
	\setlength{\faktorscale}{0.5\ht\numbox+0.5\ht\denbox}%
	\setlength{\slashlength}{0.2\faktorscale+0.8\baselineskip}%
	\raisebox{#2\ht\slashbox}{\usebox{\numbox}}
	\mkern-1mu%
	\raisebox{-0.8pt}{%
		\rotatebox{-30}{\rule[#4\ht\denbox]{0.4pt}{\slashlength}} 
	}%
	\mkern-1mu%
	\hspace{-0.25\slashlength}%
	\raisebox{#4\ht\denbox}{\usebox{\denbox}}
}

\title{\vspace{-1ex} Involutive Markov categories and the\\ quantum de Finetti theorem}

\author[1]{Tobias Fritz}
\author[2,*]{Antonio Lorenzin}
\affil[1]{{\small Department of Mathematics, University of Innsbruck, Austria, ORCID: \href{https://orcid.org/0000-0001-7081-2635}{0000-0001-7081-2635}}}
\affil[2]{{\small Independent researcher, ORCID: \href{https://orcid.org/0000-0002-2415-4261}{0000-0002-2415-4261}}}
\affil[*]{{\small Corresponding author,  \href{mailto:a.lorenzin.95@gmail.com}{a.lorenzin.95@gmail.com}}}
\date{}
\usepackage{fancyhdr} 
\makeatletter
\newcommand{\runtitle}{Involutive Markov categories and the quantum de Finetti theorem}	
\makeatother
\begin{document}

\maketitle
\par\vspace{-8ex}
\begin{abstract}
	\noindent Markov categories have recently emerged as a powerful high-level framework for probability theory and theoretical statistics.
	Here we study a quantum version of this concept, called involutive Markov categories. 
	These are equivalent to Parzygnat's quantum Markov categories, but we argue that they offer a simpler and more practical approach.
	Our main examples of involutive Markov categories have pre-C*-algebras, including infinite-dimensional ones, as objects, together with completely positive unital maps as morphisms in the \emph{picture} of interest.
	In this context, we prove a quantum de Finetti theorem for both the minimal and the maximal C*-tensor norms, and we develop a categorical description of such quantum de Finetti theorems which amounts to a universal property of state spaces.
\end{abstract}

\tableofcontents

\section{Introduction}

Categorical probability, and in particular the theory of Markov categories, has attracted much interest in recent years. 
This ``synthetic'' approach to probability has proven fruitful in several areas, including statistics~\cite{fritz2023representable}, graphical models~\cite{fritz2022dseparation} and ergodic theory~\cite{moss2022ergodic}. 
The first paper on a \emph{quantum} version of Markov categories was written by Parzygnat~\cite{quantum-markov}, whose main topic has been information flow axioms (see~\cite{fritz2022dilations} on this term) and Bayesian inversion. 
Although there are many examples of ordinary Markov categories~\cite{fritz2019synthetic}, the only quantum Markov category considered so far has finite-dimensional C*-algebras as objects.
Roughly speaking, a C*-algebra is an algebra over the complex numbers, equipped with an involution modelled after hermitian conjugation and a norm which makes it into a Banach space.\footnote{In contrast to most references on C*-algebras, we will always assume our C*-algebras to be unital.}
In quantum physics, C*-algebras come up in the form of algebras of observables.
In this context, physical transformations or processes are modelled as completely positive unital maps between C*-algebras.
Probabilistic Gelfand duality~\cite{furber_jacobs_gelfand} then explains the sense in which this formalism is a noncommutative generalization of probability: 
The category of commutative C*-algebras with positive unital maps is dually equivalent to the category of compact Hausdorff spaces with continuous Markov kernels.

The present paper develops a quantum version of Markov categories which facilitates a sensible treatment of infinite-dimensional C*-algebras. 
Among other things, this results in a new perspective on the quantum de Finetti theorem in terms of a universal property of state spaces.
To the best of our knowledge, the version of the quantum de Finetti that we prove strengthens the known results in the theory of operator algebras.
\paragraph*{Quantum versions of Markov categories.}
A first important step is the definition of \emph{involutive Markov categories}.
These categories are equivalent to Parzygnat's quantum Markov categories (\cref{prop:qcdicd}), but they avoid any distinction between ``even'' and ``odd'' morphisms---which correspond to linear and antilinear maps, respectively---thus making the theory a bit cleaner. 
Another helpful notion introduced in this paper is that of \emph{picture}, which encodes some requirements one would expect to have when dealing with quantum probability.
For instance, the classical objects of any picture---i.e., those conforming to classical probability---form a Markov category. 
To better understand this concept, let us consider the example of C*-algebras.
The associated involutive Markov category contains all unital maps, but only the completely positive ones have a well-behaved and physically relevant probabilistic interpretation. 
Therefore, we consider the picture given by the completely positive unital maps, and keep the involutive Markov category as the environment for diagrammatic calculus. 

As pointed out by Parzygnat~\cite[Remark~3.12]{quantum-markov}, a major challenge for quantum probability is the construction of a quantum Markov category of infinite-dimensional C*-algebras.
The issue is that the multiplication map on an infinite-dimensional C*-algebra is typically unbounded~\cite[Remark~3.12]{quantum-markov}.
Since the multiplication maps are the copy morphisms, which form an essential piece of structure in any flavour of Markov category, we are forced to allow unbounded maps as morphisms.
In order for these to compose reasonably, we extend our study to \emph{pre-C*-algebras}, which, unlike C*-algebras, are not required to be norm-complete.
In this way, we obtain a well-behaved involutive Markov category containing infinite-dimensional C*-algebras, and with completely positive unital maps as the picture of interest.

Another issue with considering infinite-dimensional C*-algebras is that there is no canonical tensor product for them, as there are two generally distinct canonical norms for the tensor product of two C*-algebras. 
These are known as the \emph{minimal} and \emph{maximal} C*-tensor norms. 
They give rise to two similar settings, where the underlying categories are the same, but the symmetric monoidal structures differ (in their respective pictures).
Although all our results apply to both cases, the distinction between the two norms also manifests itself at the categorical level: see \cref{prop:display_compatibility}.
\paragraph*{The quantum de Finetti theorem.}
In classical probability, the de Finetti theorem tells us that whenever the probability of an infinite sequence of outcomes is independent of their order, such outcomes are drawn independently from the same underlying distribution (which can itself be random).
Historically, this result has been a milestone in the debate over the subjective vs.~objective view of probability, but it is also crucial to the development of nonparametric Bayesian modeling. 
A brief discussion can be found in~\cite{fritz2021definetti}, where the authors give a new proof of the de Finetti theorem using only the technology of Markov categories. 
Furthermore, this synthetic approach generalizes the result by allowing everything to depend measurably on an additional parameter.

A quantum analog of the de Finetti theorem is due to St{\o}rmer~\cite{stormer69exchangeable}, and it is fundamental in quantum Bayesianism~\cite{fuchs2004unknownquantum}. 
A categorical perspective on the quantum de Finetti theorem has already been considered by Staton and Summers~\cite{staton2023quantum}, who extended St{\o}rmer's result with parameters using categorical limits.
These works, however, focus only on the minimal C*-tensor norm.
The case of the maximal C*-tensor norm was considered by Hulanicki and Phelps~\cite{hulanicki68tensor}, but their paper seems to be little known. 
Their argument, which applies to both the minimal and the maximal C*-tensor norm, is presented in \cref{sec:qdf_cpu}.
Moreover, this generalizes to the case of completely positive unital maps with an additional tensor factor (\cref{thm:qdf}), in particular allowing for correlation and (a priori) entanglement with another system.
Our result then in particular implies that there is no entanglement with the additional system, a fact whose finite-dimensional special case is well-known in quantum information theory, as it underlies a prominent separability criterion for quantum states~\cite[Theorem~1]{doherty04complete}.

In our categorical setting, this quantum de Finetti theorem (for either choice of tensor product) results in an elegant universal property similar to the one of Staton and Summers~\cite{staton2023quantum}.
This limit is a universal property of the state space of a pre-C*-algebra, and it gives rise to a quantum analog of the notion of \emph{representability} of a Markov category~\cite{fritz2023representable}.

\subsubsection*{Overview}
\cref{sec:involutivemarkov} focuses on involutive Markov categories. 
After introducing the definition in \cref{sec:def_involutivemarkov}, we present the main example, given by pre-C*-algebras and unital linear maps in \cref{sec:mainexample_invM}. 
The connection to existing literature, in particular to quantum Markov categories, is covered in \cref{sec:literature}. 
We then study the relevant notions of functors and natural transformations (\cref{sec:icdfun}), leading to a strictification theorem for such categories in \cref{sec:strictification}. 
A general example of involutive Markov categories is given in \cref{sec:cofree}, where we define involutive comonoids in an involutive symmetric monoidal category and show that they assemble to an involutive Markov category.
\cref{sec:classical_compatible,sec:as} develop basic properties of objects and morphisms. 
The former discusses new key notions, such as classicality of objects and compatibility of morphisms, while the latter presents four non-equivalent definitions of almost sure equality, all of which specialize to the standard one for classical Markov categories.

\cref{sec:pictures} begins by introducing pictures (\cref{sec:picture_def}) as the relevant probabilistic framework, motivated by the main example presented in \cref{sec:mainexample_cpu}. 
We then extend the characterization of almost sure equality developed by Parzygnat~\cite[Section~5]{quantum-markov} in \cref{sec:as_nullspaces}, and introduce a quantum version of Kolmogorov products~\cite{fritzrischel2019zeroone} in \cref{sec:kolmogorov}.

Our main results appear in \cref{sec:representability_qdf}, where we investigate representability in various forms. 
\cref{sec:representability} deals with the most straightforward notion of representability, adapted from the classical counterpart~\cite{fritz2023representable}. 
While this concept may become significant in the future, it is not central to our present investigation.
We then define classical representability in \cref{sec:classical_representability}, where representability is required only against classical objects (i.e.~those encoding classical probability).
This is related to de Finetti representability (\cref{sec:qdf}), a concept arising from the de Finetti theorem, hence the name.
\cref{sec:qdf_cpu} begins with a detailed discussion of a quantum de Finetti theorem for states, independent of the choice of tensor norm, as established by Hulanicki and Phelps~\cite{hulanicki68tensor}. 
This result is then applied to prove our main theorem (\cref{thm:qdf}), which states that our pictures of interest are de Finetti representable.

\cref{sec:infinite_tensor_trivial} supplies a proof omitted from~\cref{rem:01laws} concerning the Hewitt--Savage zero--one law.

\subsubsection*{Prerequisites}
In this article, we will make extensive use of \emph{string diagrams} for symmetric monoidal categories. 
An introduction to such diagrams can be found in~\cite{selinger11graphical}. 
We point out that this graphical approach suppresses the coherence isomorphisms of the monoidal structure. 
For this reason, it is important to understand if the categories we work with admit a strictification. This turns out to be the case in our setting---a precise statement is \cref{prop:strictification}.

Besides string diagrams, the article is intended to be readable by people with a basic background in symmetric monoidal categories, although some knowledge of existing work on Markov categories or quantum Markov categories~\cite{quantum-markov,fritz2019synthetic} may improve the understanding of the paper.

For the sake of completeness, all necessary basics about pre-C*-algebras are covered within the paper (\cref{sec:mainexample_invM,sec:mainexample_cpu}), but a previous acquaintance with these structures will help with following the proofs.

\subsubsection*{Acknowledgements}
We thank Tom{\'a}\v{s} Gonda, Paolo Perrone, Sam Staton, Ned Summers and Nico Wittrock for fruitful discussions and comments.
The authors acknowledge support from the Austrian Science Fund (FWF P 35992-N).
We are grateful to the anonymous reviewer for their insight and thorough reading, which led to a complete restructuring of the paper to enhance clarity.

\section{Involutive Markov categories}\label{sec:involutivemarkov}

In this section, we introduce the protagonists of our approach to categorical quantum probability, discuss some basic examples and properties, and prove a strictification theorem.

We begin by introducing \emph{involutive Markov categories}, and the more general \emph{ICD-categories}, and discuss our main example, the involutive Markov category of pre-C*-algebras. 
This choice is motivated by the fact that there is no quantum Markov category of all C*-algebras~\cite[Remark~3.12]{quantum-markov}, and therefore we need to adapt the framework to something slightly different. 
Basically, pre-C*-algebras have a better behaved tensor product, and this allows us to consider the multiplication as a morphism despite it generally being unbounded.
After the main example, we show that involutive Markov categories are equivalent to Parzygnat's notion of quantum Markov category.
The main difference is that our categories do not involve a distinction between ``even'' and ``odd'' morphisms, since they only contain the even ones to start with.
This makes them somewhat simpler to work with than Parzygnat's original concepts.
As the names indicate, our categories have an involution on morphisms instead.
We then introduce a notion of involutive comonoid in a hom-involutive symmetric monoidal category such that the category of all such comonoids becomes an involutive Markov category (\cref{prop:cof_icd}). 
In particular, this answers the final question in~\cite[Question~3.24]{quantum-markov}.
We also state and prove a strictification theorem for ICD-categories (\cref{prop:strictification}). 
The section concludes by discussing some important notions of compatibility of morphisms related to compatibility of measurements in quantum theory, and by investigating almost sure equalities in our context.

\subsection{Definition of involutive Markov categories}\label{sec:def_involutivemarkov}

Throughout, we work in a symmetric monoidal category\footnote{The unorthodox choice of $\odot$ to denote the  tensor is motivated by our main example (\cref{sec:mainexample_invM}). 
In operator algebra, $\odot$ is conventionally used for the \emph{algebraic} tensor product, while $\otimes$ is reserved to its completion.} $(\cat{C},\tensor,I)$ together with a strict symmetric monoidal endofunctor
\[
	\involution \colon \cat{C} \to \cat{C}
\]
called \newterm{involution}, such that $\involution^2 =\id_{\cat{C}}$ and $\involution{(A)}=A$ for every object $A$ of $\cat{C}$. When needed, such a $\cat{C}$ will be called \newterm{hom-involutive symmetric monoidal category}.\footnote{Alternatively, this is a symmetric monoidal category enriched in $\mathbb{Z}_2$-sets. We leave the details of the equivalence to the interested reader.} We will write $\involution{(\phi)}= \invo{\phi}$ for morphisms.
\begin{definition}\label{def:icd}
	A hom-involutive symmetric monoidal category $\cat{C}$ is an \newterm{ICD-category} if every object $A$ comes equipped with a \newterm{copy morphism} $\cop_A\colon A \to A \tensor A$ and a \newterm{delete morphism} $\discard_A : A \to I$,	diagrammatically written as
	\begin{equation*}
		\begin{tikzpicture}
	\begin{pgfonlayer}{nodelayer}
		\node [style=bn] (0) at (-4, 0) {};
		\node [style=none] (1) at (-5, 1.25) {};
		\node [style=none] (2) at (-4, -1) {};
		\node [style=none] (3) at (-3, 1.25) {};
		\node [style=none] (4) at (-5, 1.75) {$A$};
		\node [style=none] (6) at (-3, 1.75) {$A$};
		\node [style=none] (7) at (-4, -1.5) {$A$};
		\node [style=none] (8) at (-9.5, 0) {$\cop_A$};
		\node [style=bn] (9) at (7, 0.5) {};
		\node [style=none] (10) at (7, -1) {};
		\node [style=none] (11) at (7, -1.5) {$A$};
		\node [style=none] (12) at (3, 0) {$\discard_A$};
		\node [style=none] (13) at (-7, 0) {$=$};
		\node [style=none] (14) at (5, 0) {$=$};
	\end{pgfonlayer}
	\begin{pgfonlayer}{edgelayer}
		\draw [in=165, out=-90] (1.center) to (0);
		\draw [in=270, out=15] (0) to (3.center);
		\draw (0) to (2.center);
		\draw (9) to (10.center);
	\end{pgfonlayer}
\end{tikzpicture}
	\end{equation*}
	such that:
	\begin{enumerate}
		\item Every object becomes a comonoid: 
		\begin{equation}\label{eq:icd_def_comonoid}
			\begin{tikzpicture}
	\begin{pgfonlayer}{nodelayer}
		\node [style=bn] (0) at (2.75, -0.75) {};
		\node [style=none] (1) at (2, 0.25) {};
		\node [style=none] (2) at (2.75, -1.5) {};
		\node [style=none] (3) at (3.5, 0.25) {};
		\node [style=none] (11) at (2, 1.5) {};
		\node [style=bn] (12) at (3.5, 0.75) {};
		\node [style=none] (13) at (5.5, 1.5) {};
		\node [style=none] (14) at (5.5, -1.5) {};
		\node [style=bn] (15) at (7.5, 0.75) {};
		\node [style=bn] (16) at (8.25, -0.75) {};
		\node [style=none] (17) at (9, 1.5) {};
		\node [style=none] (18) at (9, 0.25) {};
		\node [style=none] (19) at (7.5, 0.25) {};
		\node [style=none] (20) at (8.25, -1.75) {};
		\node [style=none] (21) at (4.5, 0) {$=$};
		\node [style=none] (22) at (6.5, 0) {$=$};
		\node [style=bn] (23) at (-8, -0.75) {};
		\node [style=none] (24) at (-8, -1.5) {};
		\node [style=bn] (25) at (-8.75, 0.25) {};
		\node [style=none] (26) at (-9.5, 1.25) {};
		\node [style=none] (27) at (-8, 1.25) {};
		\node [style=none] (28) at (-7, 0.5) {};
		\node [style=none] (29) at (-9.5, 1.5) {};
		\node [style=none] (30) at (-8, 1.5) {};
		\node [style=none] (31) at (-7, 1.5) {};
		\node [style=none] (32) at (-4.5, 1.5) {};
		\node [style=none] (33) at (-3.5, 1.5) {};
		\node [style=none] (34) at (-2, 1.5) {};
		\node [style=none] (35) at (-4.5, 0.5) {};
		\node [style=none] (36) at (-3.5, 1.25) {};
		\node [style=none] (37) at (-2, 1.25) {};
		\node [style=none] (38) at (-3.5, -1.5) {};
		\node [style=bn] (39) at (-3.5, -0.75) {};
		\node [style=bn] (40) at (-2.75, 0.25) {};
		\node [style=none] (41) at (-5.75, 0) {$=$};
	\end{pgfonlayer}
	\begin{pgfonlayer}{edgelayer}
		\draw (0) to (2.center);
		\draw (11.center) to (1.center);
		\draw (12) to (3.center);
		\draw [in=165, out=-90] (1.center) to (0);
		\draw [in=270, out=15] (0) to (3.center);
		\draw (14.center) to (13.center);
		\draw (17.center) to (18.center);
		\draw [in=15, out=-90] (18.center) to (16);
		\draw (15) to (19.center);
		\draw [in=165, out=-90] (19.center) to (16);
		\draw (20.center) to (16);
		\draw (23) to (24.center);
		\draw [in=165, out=-90] (25) to (23);
		\draw [in=15, out=-90] (28.center) to (23);
		\draw [in=165, out=-90] (26.center) to (25);
		\draw [in=-90, out=15] (25) to (27.center);
		\draw (29.center) to (26.center);
		\draw (30.center) to (27.center);
		\draw (31.center) to (28.center);
		\draw (32.center) to (35.center);
		\draw (33.center) to (36.center);
		\draw (34.center) to (37.center);
		\draw [in=165, out=-90] (36.center) to (40);
		\draw [in=15, out=-90] (37.center) to (40);
		\draw [in=165, out=-90] (35.center) to (39);
		\draw [in=15, out=-90] (40) to (39);
		\draw (39) to (38.center);
	\end{pgfonlayer}
\end{tikzpicture}
		\end{equation}
		\item They behave nicely with the involution:
		\begin{equation}\label{eq:icd_def_involution}
			\begin{tikzpicture}
	\begin{pgfonlayer}{nodelayer}
		\node [style=none] (42) at (5, -1.5) {};
		\node [style=bn] (43) at (5, 0.75) {};
		\node [style=none] (49) at (1.75, 1) {};
		\node [style=none] (50) at (2.25, 1) {};
		\node [style=none] (44) at (2, -1.5) {};
		\node [style=bn] (45) at (2, 0.75) {};
		\node [style=none] (48) at (3.5, 0) {$=$};
		\node [style=none] (51) at (-3.5, 1.5) {};
		\node [style=none] (52) at (-2, 1.5) {};
		\node [style=bn] (53) at (-2.75, -0.75) {};
		\node [style=bn] (54) at (-6.25, -0.75) {};
		\node [style=none] (55) at (-6.5, -0.5) {};
		\node [style=none] (56) at (-6, -0.5) {};
		\node [style=none] (57) at (-7, 1.5) {};
		\node [style=none] (58) at (-5.5, 1.5) {};
		\node [style=none] (59) at (-2.75, -1.5) {};
		\node [style=none] (60) at (-6.25, -1.5) {};
		\node [style=none] (61) at (-4.5, 0) {$=$};
		\node [style=none] (62) at (-7, 0.25) {};
		\node [style=none] (63) at (-5.5, 0.25) {};
		\node [style=none] (64) at (-3.5, 0) {};
		\node [style=none] (65) at (-2, 0) {};
	\end{pgfonlayer}
	\begin{pgfonlayer}{edgelayer}
		\draw (42.center) to (43);
		\draw (44.center) to (45);
		\draw (49.center) to (50.center);
		\draw (57.center) to (62.center);
		\draw (58.center) to (63.center);
		\draw [in=165, out=-90] (62.center) to (54);
		\draw [in=15, out=-90] (63.center) to (54);
		\draw (54) to (60.center);
		\draw (53) to (59.center);
		\draw [in=165, out=-90] (64.center) to (53);
		\draw [in=270, out=15] (53) to (65.center);
		\draw [in=90, out=-90] (51.center) to (65.center);
		\draw [style=protected, in=-90, out=90, looseness=0.75] (64.center) to (52.center);
		\draw (55.center) to (56.center);
	\end{pgfonlayer}
\end{tikzpicture}
		\end{equation}
			where we use $\invo{\bullet}$ to denote $\invo{\cop_{A}}$ and $\invo{\discard_{A}}$, respectively.
		\item They respect the symmetric monoidal structure:
			\begin{equation}\label{eq:icd_def_monoidal}
				\begin{tikzpicture}
	\begin{pgfonlayer}{nodelayer}
		\node [style=none] (1) at (-0.5, -1.75) {};
		\node [style=bn] (2) at (-0.5, -0.75) {};
		\node [style=none] (3) at (0.25, 0) {};
		\node [style=none] (4) at (0.25, 0) {};
		\node [style=none] (5) at (-1.25, 0.25) {};
		\node [style=none] (6) at (2.25, -1.75) {};
		\node [style=bn] (7) at (2.25, -0.75) {};
		\node [style=none] (8) at (3, 0.25) {};
		\node [style=none] (9) at (3, 0.25) {};
		\node [style=none] (10) at (1.5, 0) {};
		\node [style=none] (11) at (-1.25, 1.5) {};
		\node [style=none] (12) at (0.25, 1.5) {};
		\node [style=none] (13) at (1.5, 1.5) {};
		\node [style=none] (14) at (3, 1.5) {};
		\node [style=none] (15) at (-2.75, 0) {$=$};
		\node [style=bn] (16) at (-6, -0.25) {};
		\node [style=none] (17) at (-6, -1.75) {};
		\node [style=none] (18) at (-6.75, 0.75) {};
		\node [style=none] (19) at (-5.25, 0.75) {};
		\node [style=none] (21) at (-4.75, -1.5) {$A\tensor B$};
		\node [style=none] (22) at (-6.75, 1.5) {};
		\node [style=none] (23) at (-5.25, 1.5) {};
		\node [style=none] (28) at (0, -1.5) {$A$};
		\node [style=none] (29) at (2.75, -1.5) {$B$};
		\node [style=bn] (30) at (15, 0.75) {};
		\node [style=none] (31) at (15, -1.75) {};
		\node [style=none] (32) at (15.5, -1.5) {$I$};
		\node [style=none] (33) at (17.25, 1.5) {};
		\node [style=none] (34) at (17.25, -1.75) {};
		\node [style=none] (35) at (17.75, -1.5) {$I$};
		\node [style=none] (40) at (20.5, -1.5) {$I$};
		\node [style=none] (41) at (16.25, 0) {$=$};
		\node [style=none] (42) at (18.25, 0) {$=$};
		\node [style=bn] (43) at (6, 0.75) {};
		\node [style=none] (44) at (6, -1.75) {};
		\node [style=none] (45) at (7.25, -1.5) {$A \tensor B$};
		\node [style=bn] (46) at (10.25, 0.75) {};
		\node [style=none] (47) at (10.25, -1.75) {};
		\node [style=none] (48) at (10.75, -1.5) {$A$};
		\node [style=bn] (49) at (11.75, 0.75) {};
		\node [style=none] (50) at (11.75, -1.75) {};
		\node [style=none] (51) at (12.25, -1.5) {$B$};
		\node [style=none] (52) at (8.75, 0) {$=$};
		\node [style=bn] (53) at (20, -0.25) {};
		\node [style=none] (54) at (20, -1.75) {};
		\node [style=none] (55) at (19.25, 0.75) {};
		\node [style=none] (56) at (20.75, 0.75) {};
		\node [style=none] (57) at (19.25, 1.5) {};
		\node [style=none] (58) at (20.75, 1.5) {};
	\end{pgfonlayer}
	\begin{pgfonlayer}{edgelayer}
		\draw [style=none] (1.center) to (2);
		\draw [style=none, in=-90, out=165] (2) to (5.center);
		\draw [style=none, in=-90, out=15] (2) to (4.center);
		\draw [style=none] (6.center) to (7);
		\draw [style=none, in=-90, out=165] (7) to (10.center);
		\draw [style=none, in=-90, out=15] (7) to (9.center);
		\draw [style=none] (11.center) to (5.center);
		\draw [style=none] (8.center) to (14.center);
		\draw [style=none, in=-90, out=90] (3.center) to (13.center);
		\draw (17.center) to (16);
		\draw [in=270, out=165] (16) to (18.center);
		\draw [in=270, out=15] (16) to (19.center);
		\draw (23.center) to (19.center);
		\draw (18.center) to (22.center);
		\draw [style=protected, in=90, out=-90] (12.center) to (10.center);
		\draw (30) to (31.center);
		\draw (43) to (44.center);
		\draw (34.center) to (33.center);
		\draw (46) to (47.center);
		\draw (49) to (50.center);
		\draw (54.center) to (53);
		\draw [in=270, out=165] (53) to (55.center);
		\draw [in=270, out=15] (53) to (56.center);
		\draw (58.center) to (56.center);
		\draw (55.center) to (57.center);
	\end{pgfonlayer}
\end{tikzpicture}
			\end{equation}
	\end{enumerate}
\end{definition}

As is commonly done in string diagram calculus, we leave out the associators and unitors in the diagrams.
For example, the final equation in~\eqref{eq:icd_def_monoidal} needs to be understood as holding modulo the coherence isomorphism $I \tensor I \stackrel{\cong}{\to} I$.

\begin{definition}\label{def:selfadj_total}
	A morphism $\phi \colon A \to B$ in an ICD-category $\cat{C}$ is
	\begin{enumerate}
		\item \newterm{total} if $\discard_B \comp \phi = \discard_A$, i.e.
		\begin{equation*}
			\begin{tikzpicture}
	\begin{pgfonlayer}{nodelayer}
		\node [style=none] (42) at (2, -1.5) {};
		\node [style=bn] (43) at (2, 1.5) {};
		\node [style=none] (44) at (-2, -1.5) {};
		\node [style=bn] (45) at (-2, 1.5) {};
		\node [style=none] (48) at (0, 0) {$=$};
		\node [style=morphism] (49) at (-2, 0) {$\phi$};
	\end{pgfonlayer}
	\begin{pgfonlayer}{edgelayer}
		\draw (42.center) to (43);
		\draw (44.center) to (45);
	\end{pgfonlayer}
\end{tikzpicture}
		\end{equation*}
		\item \newterm{self-adjoint} if $\phi= \invo{\phi}$.
	\end{enumerate}
	If all morphisms in $\cat{C}$ are total, $\cat{C}$ is called an \newterm{involutive Markov category}. 
\end{definition}

In particular, all associators, unitors and swap morphisms are self-adjoint as the involution functor $\involution$ is assumed strict symmetric monoidal.
Similarly, identities are self-adjoint as $\involution$ is identity-on-objects.

\begin{remark}
	An involutive Markov category is, equivalently, an ICD-category whose monoidal unit is terminal. 
	In fact, delete morphisms are natural precisely on total morphisms.
\end{remark}

We will restrict our focus to involutive Markov categories from \Cref{sec:classical_compatible} onward.

\subsection{Main example: pre-C*-algebras}\label{sec:mainexample_invM}
The study of quantum probability generally concerns C*-algebras equipped with completely positive unital maps. 
In particular, it seems necessary to restrict to C*-algebras as objects and bounded linear maps as morphisms (or further to some subcategory thereof, like von Neumann algebras and normal completely positive or completely bounded maps).
Unfortunately this is not possible in the context of involutive Markov categories, since the C*-algebra multiplication, which is the expected formal dual of our copy, is typically not bounded for infinite-dimensional C*-algebras (refer to~\cite[Remark~3.12]{quantum-markov} for a precise discussion).
If we drop the boundedness requirement, then we do not have any tensor product of morphisms, since the tensor product of C*-algebras is a \emph{completion} of their algebraic tensor product, and without boundedness we no longer have a natural way to extend a linear map to the completion.

In order to circumvent these issues, we expand our interest to \emph{pre-}C*-algebras.
Moreover, we emphasize that all algebras (and rings) are \underline{assumed unital} without further mention.
In \cref{sec:pictures}, we will discuss how to recover a more standard framework in which maps are completely positive and probabilistic aspects become more apparent.

Let us start by recalling that a \newterm{$*$-ring} $R$ is a \emph{unital} ring $R$ together with a unary operation $*$ that is:
\begin{itemize}
	\item Additive, $x^* + y^*= (x+y)^*$;
	\item Unit-preserving, $1^*=1$;
	\item Multiplication-reversing, $(xy)^* = y^*x^*$;
	\item Involutive, $x^{**}=x$.
\end{itemize}
A $*$-ring $A$ is a \newterm{complex $*$-algebra}, sometimes simply $*$-algebra, if it is a $\mathbb{C}$-algebra (i.e.~$\lambda (xy)=(\lambda x)y=x(\lambda y)$) that additionally satisfies $(\lambda x)^* = \overline{\lambda} x^*$ for all $x,y \in A$ and $\lambda \in \mathbb{C}$.
\begin{definition}
	\label{def:precalg}
	A \newterm{pre-C*-algebra} ${A}$ is a complex $*$-algebra equipped with a  \newterm{C*-norm}, i.e.\ a submultiplicative norm that satisfies the C*-identity
	\begin{equation}\label{eq:cstar_id}
		\norm{ x^* x }	= \norm{ x }^2 	
	\end{equation}
	for every $x \in {A}$. 

If the norm on ${A}$ is complete, we say that $A$ is a \newterm{C*-algebra}.
\end{definition}
By submultiplicativity and the C*-identity, the star operation is isometric, i.e.~$\norm{x^*}= \norm{x}$.

As we will see, most of the relevant standard theory from the C*-setting still works if our algebras are not complete, provided that the definitions are suitably adapted.
Of course, not all results on Banach spaces extend to normed vector spaces: 
For example, the open mapping theorem and the closed graph theorem are two well-known examples of statements that fail to hold without assuming completeness.
For our purposes, however, virtually everything works without completeness.
The only exception is a detail on boundedness, namely \cref{ex:homnotbound}.

\begin{lemma}\label{lem:closure}
	The completion $\closed{A}$ of a pre-C*-algebra ${A}$ is a C*-algebra. 
	Conversely, any $*$-subalgebra of a C*-algebra is a pre-C*-algebra with respect to the induced norm.
\end{lemma}
\begin{proof}
	Since the addition, the product and the involution respect limits, the completion becomes a C*-algebra.
	The second part is immediate.
\end{proof}

\begin{lemma}\label{lem:trick}
	Let $A$ and $B$ be pre-C*-algebras and $\opm{\phi}{}\colon {B} \rightsquigarrow {A}$ a bounded linear map.
	Then $\opm{\phi}{}$ has a unique bounded extension $\close{\opm{\phi}{}}\colon \closed{B} \rightsquigarrow \closed{A}$ between C*-algebras, and $\norm{ \close{\opm{\phi}{}}} = \norm{ \opm{\phi}{} }$. 
\end{lemma}
Here we see the first occurrences of squiggly arrows ($\rightsquigarrow$) and the superscript $\op$. 
Throughout the article, this notation will be used to indicate the operator algebra direction, as discussed in \cref{rem:op_nota}.
\begin{proof}
	For $x = \lim_n x_n$ with $x \in {B}$, we define $\close{\opm{\phi}{}}(x) \coloneqq \lim_n \opm{\phi}{}(x_n)$. Since $\opm{\phi}{}$ is bounded, this is well-defined. Uniqueness is immediate since for every limit a bounded map satisfies $\opm{\phi}{}(\lim_n x_n)= \lim_n \opm{\phi}{}(x_n)$.
	The equation $\norm{ \close{\opm{\phi}{}}} = \norm{ \opm{\phi}{}}$ holds by continuity of the norm.
\end{proof}
It is worth noting that for every pre-C*-algebra $A$, its completion $\closed{A}$ is the unique (up to a unique $*$-isomorphism) C*-algebra that contains $A$ as a dense $*$-subalgebra.
Here \newterm{dense} is meant with respect to the topology induced by the norm.

\begin{notation}\label{nota:completion}
	According to \cref{lem:closure,lem:trick}, we use $\closed{A}$ to denote the {completion} of the pre-C*-algebra $A$ and $\close{\opm{\phi}{}}\colon \closed{B} \rightsquigarrow \closed{A}$ to denote the unique bounded extension of a bounded map $\opm{\phi}{}\colon B \rightsquigarrow A$.

	When a pre-C*-algebra is in particular a C*-algebra, we will denote it with $\closed{A}$ for emphasis.
\end{notation}

In order to obtain an involutive Markov category with pre-C*-algebras as objects, we need to understand how one can equip the algebraic tensor product with a C*-norm. 
This can be done by considering standard choices of C*-tensor norms.

\begin{proposition}\label{prop:prec_tensor_norm}
	Let $A$ and $B$ be pre-C*-algebras. 
	Then ${A} \odot {B}$ is dense in $\closed{A} \otimes \closed{B}$, the C*-algebra obtained by completing $\closed{A} \odot \closed{B}$ with respect to any C*-norm, i.e.~a submultiplicative norm satisfying the C*-identity \eqref{eq:cstar_id}.
\end{proposition} 

To emphasize the distinction between the algebraic tensor product $\odot$ and a tensor product of C*-algebras $\otimes$, we use different symbols.

\begin{proof}
	Clearly $\closed{A} \odot \closed{B}$ is dense in $\closed{A} \otimes \closed{B}$ by definition of the latter, so it is enough to show that ${A} \odot {B}$ is dense in $\closed{A} \odot \closed{B}$.
For a simple tensor $x\odot y \in \closed{A} \odot \closed{B}$, we have that $x=\lim_n x_n$ and $y=\lim_m y_m$ are limits of sequences $(x_n)$ in ${A}$ and $(y_m)$ in ${B}$, respectively. 
Then a standard argument shows that $x \odot y = \lim_{n} ({x_n \odot y_n})$ by submultiplicativity of the norm:
\begin{equation*}
	\begin{split}
\norm{x\odot y - x_n \odot y_n} & \le \norm{x\odot y - x\odot y_n}+ \norm{x\odot y_n - x_n\odot y_n} \\
&\le \norm{x\odot 1}\norm{1\odot (y-y_n)} + \norm{(x-x_n)\odot 1} \norm{1\odot y_n}	\\
& = \norm{x}\norm{(y-y_n)} + \norm{(x-x_n)} \norm{y_n} \xrightarrow{n\to \infty} 0,
	\end{split}
\end{equation*}
where the equalities $\norm{x}=\norm{x\odot 1}$ and $\norm{y}=\norm{1\odot y}$ follow from the fact that C*-norms on C*-algebras are unique (see, e.g., \cite[p.~2]{khalkhali13noncommutative}).
	Since any element of $\closed{A} \odot \closed{B}$ is a sum of simple tensors, and limits commute with addition, we conclude that ${A} \odot {B}$ is dense in $\closed{A} \odot \closed{B}$, as claimed.
\end{proof}

There are different notions of norm for the tensor product of C*-algebras. 
The most important ones are the \newterm{minimal} and the \newterm{maximal} C*-tensor norms. 
We omit a precise definition of these norms and refer the reader to the literature (see e.g.~\cite[Section~11.3]{kadison1997operatoralgebrasII} or~\cite[Section~3.3]{brownozawa}). 
Despite the importance of these two norms in giving a well-defined symmetric monoidal structure, their relevance in our current treatment is marginal.
\begin{notation}
	Given two C*-algebras $\closed{A}$ and $\closed{B}$, we denote by $\closed{A} \otimes_{\min} \closed{B}$, resp.~$\closed{A} \otimes_{\max} \closed{B}$, the minimal tensor product, resp.~the maximal tensor product.
\end{notation}
\begin{definition}\label{def:tensorproducts}
	Given two pre-C*-algebras $A$ and $B$, the \newterm{minimal algebraic tensor product}, denoted by $A \odot_{\min} B$, is the algebraic tensor product $A \odot B$ equipped with the norm of the minimal tensor product $\closed{A} \otimes_{\min} \closed{B}$.

	Analogously, the \newterm{maximal algebraic tensor product}, denoted by $A \odot_{\max} B$, is the algebraic tensor product with the norm induced by the inclusion $A \odot B \subseteq \closed{A} \otimes_{\max} \closed{B}$.
\end{definition}

\begin{notation}\label{nota:alg_tensor}
	Whenever we need to refer to both the minimal and the maximal algebraic tensor products, we will use $\odot$ to avoid overloading notation. 
	Similarly, $\otimes$ will denote both the minimal and the maximal tensor products between C*-algebras.

	This may cause confusion if the reader is considering the case of $*$-algebras, but since we focus mainly on pre-C*-algebras throughout (with the sole exception of \cref{sec:cofree}), the algebraic tensor product only makes sense when equipped with a norm.
\end{notation}

\begin{definition}[Main examples]\label{def:mainexample}
	Let $\calgone$ be the category defined as follows: 
	\begin{itemize}
		\item Objects are {pre-C*-algebras};
		\item Morphisms $\phi \colon A \to B$ are formal opposites of \emph{unital linear} maps 
		\[
			\opm{\phi}{}\colon {B} \rightsquigarrow {A}
		\]
		with the obvious composition.
	\end{itemize} 
	We write $\calgmin$ (resp. $\calgmax$) for the symmetric monoidal category given by equipping $\calgone$ with the tensor product 
	\[
		 A\tensor B\coloneqq A \odot_{\min} B \quad (\text{resp. } A\tensor B\coloneqq A \odot_{\max} B).
	\] 
\end{definition}

The naming of $\calgone$ is to recall that we will consider \emph{u}nital \emph{lin}ear maps.

Let us now describe the motivating example for considering such a setting, instead of restricting to bounded maps (for instance).
Let $\mathcal{H}$ be a separable Hilbert space and consider the C*-algebra of bounded operators $\mathcal{B}(\mathcal{H})$. 
Then the multiplication $\opm{\mu}{} \colon \mathcal{B}(\mathcal{H}) \odot \mathcal{B}(\mathcal{H}) \rightsquigarrow \mathcal{B}(\mathcal{H}) $ is neither bounded nor a $*$-homomorphism (see \cite[Remark 3.12]{quantum-markov}), regardless of whether one chooses the minimal or the maximal tensor norm.

\begin{remark}\label{rem:op_nota}
	The notation $\opm{\phi}{}$ has a dual purpose. 
	First, it indicates a switch in the direction using the terminology of opposite categories. 
	Second, it reminds the reader that these maps follow the \emph{op}erator algebra direction. 
	In physics terminology, this means that we use the \emph{Heisenberg picture}.

	To further emphasize the operator algebra direction, we will adopt squiggly arrows ($\rightsquigarrow$), as above.
\end{remark}

\begin{remark}\label{rem:calgmin_calgmax_differ}
	The symmetric monoidal categories $\calgmin$ and $\calgmax$ are isomorphic (via a strict symmetric monoidal functor), since the norms do not currently play any role.
	However, this will change when we consider pictures (\cref{def:picture}).
	Indeed, the subcategories of interest will require morphisms to be bounded, and this is sensitive to the choice of the tensor C*-norm.
\end{remark}

\begin{proposition}\label{prop:pcalg_icd}
	$\calgmin$ and $\calgmax$ are involutive Markov categories with respect to the involution induced by the \newterm{star operation}\footnote{We prefer to reserve the word ``involution'' for our functor $\involution \colon \phi \mapsto \invo{\phi}$ and therefore refer to the operation $*$ on a pre-C*-algebra or similar structure as ``star''.} $x \mapsto x^*$, i.e.
	\begin{equation}
		\label{eq:invo_calg}
		\invo{\opm{\phi}{}} (x)\coloneqq \opm{\phi}{}(x^*)^*,
	\end{equation}
	and copy morphisms
	\[
		\opm{\cop}{}_A \: : \: {A} \odot {A} \rightsquigarrow {A}
	\]
	given by multiplication (recall that $\odot$ is a shorthand for $\odot_{\min}$ and $\odot_{\max}$).
\end{proposition}

The proof is omitted, as all axioms can be verified by direct check.
For instance, \eqref{eq:icd_def_involution} translates to $(xy)^* = y^* x^*$ and $1^*=1$.
The reader may also refer to \cite[Example 3.9]{quantum-markov}, which addresses the example of finite dimensional C*-algebras (the correspondence between quantum CD-categories and ICD-categories is established in \cref{prop:qcdicd} below).
Further examples of ICD-categories will be discussed in \cref{sec:cofree}. 
In particular, the description given there generalizes the one adopted here for $\calgmin$ and $\calgmax$.

\begin{remark}
	Continuing \cref{rem:calgmin_calgmax_differ}, the two involutive Markov categories $\calgmin$ and $\calgmax$ are in fact isomorphic as involutive Markov categories. 
	This is intuitively obvious from the definitions since the only difference is in the norm and the norm does not have any bearing on the morphisms.
	And indeed this can be straightforwardly formalized with the relevant definition of morphism of ICD-categories, namely that of strong ICD-functors (\cref{def:strongICDfun}).
\end{remark}

Let us now give some notions relevant to the theory of pre-C*-algebras. 
In particular, the second one motivates our use of the term ``self-adjoint'' in \cref{def:selfadj_total}.
For the time being, we refrain from discussing positivity, as it is not needed in this section.
\begin{definition}
	Let $A$ and $B$ be pre-C*-algebras.
	\begin{itemize}
		\item An element $x \in A$ is \newterm{self-adjoint} if $x=x^*$;\footnote{Recall also that each element is a linear combination of two self-adjoints: $x = \left(\frac{x+x^*}{2}\right) + \mathsf{i}\, \left(\frac{x-x^*}{2\mathsf{i}}\right)$, where $\mathsf{i}$ is the imaginary unit.}
		\item A linear map $\opm{\phi}{}\colon B \rightsquigarrow A$ is \newterm{self-adjoint} if and only if $\opm{\phi}{}(x)^* = \opm{\phi}{}(x^*)$ for all $x \in B$.
		\item A self-adjoint map is a \newterm{$*$-homomorphism} if it additionally is an algebra homomorphism (and in particular it preserves the unit).
	\end{itemize}
\end{definition}

\begin{remark}
	In both $\calgmin$ and $\calgmax$, a morphism $\phi \colon A \to B$ is self-adjoint if and only if it is self-adjoint as a map between pre-C*-algebras, i.e.~$\opm{\phi}{}(x^*) = \opm{\phi}{}(x)^*$ for all $x \in {B}$.
\end{remark}

\begin{example}\label{ex:observables}
	Let $B = \mathbb{C}^{\{0,1\}}$ be the two-dimensional commutative C*-algebra whose elements are pairs of complex numbers $(b_0, b_1)$, multiplication is defined componentwise, and the unit is $(1,1)$.
	For any pre-C*-algebra $A$, the morphisms $\phi : A \to B$ correspond to pairs of elements $(a_0, a_1) \in A \times A$ with $a_0 + a_1 = 1$---such elements are the images of $(1,0)$ and $(0,1)$ respectively.
	Since such a pair is uniquely determined by one component, it follows that the morphisms $A \to \mathbb{C}^{\{0,1\}}$ can be identified with the elements of ${A}$.
	The self-adjoint morphisms correspond to the self-adjoint elements.
\end{example}

Before proceeding, let us recall a fundamental result on C*-algebras. Although it will only appear in a few arguments throughout the paper (specifically, in \cref{ex:homnotbound,prop:class_representable,rem:non_simplices}), we believe its insight on C*-algebras is highly valuable.

\begin{proposition}[{Gelfand duality, e.g.~\cite[Theorem~1.1.1]{khalkhali13noncommutative}}]
	Let us consider
	\begin{itemize}
		\item $\cat{CHaus}$, the category of compact Hausdorff spaces with continuous maps, and
		\item $\mathsf{CC}^*\operatorname{-}\hspace{-0.3ex}\cat{alg}$, the category of commutative C*-algebras with (bounded) $*$-homomorphisms.
	\end{itemize}	
	Then there is a \emph{contravariant} equivalence 
	\[
		\cat{CHaus}\xrightarrow{\,\cong\,} \mathsf{CC}^*\operatorname{-}\hspace{-0.3ex}\cat{alg}
	\]
	sending each compact Hausdorff space $X$ to the commutative C*-algebra $C(X)$ of complex-valued continuous maps.
\end{proposition}


\subsection{Relation to existing literature}\label{sec:literature}

This subsection explains some important connections between ICD-categories and the existing theory of CD-categories~\cite{chojacobs2019strings,fritz2019synthetic} and quantum CD-categories in the sense of Parzygnat~\cite{quantum-markov}. 
For the sake of brevity, we refrain from a precise definition of quantum CD-categories, as they are tangential to our interest.
Readers unfamiliar with these concepts may find it sufficient to know that the basic structure resembles that of ICD-categories, but morphisms are divided in two classes: \emph{even} and \emph{odd} morphisms, which correspond to linear and \emph{antilinear} maps respectively. 
This approach obviates the need for an explicit involution functor by equipping every object $A$ with a star (odd) morphism $*_A$ satisfying suitable conditions.
In the example of pre-C*-algebras, these star morphisms are the antilinear maps $*_A \colon x \mapsto x^*$.

\begin{remark}
	CD-categories are precisely the ICD-categories with trivial involution.
\end{remark}

\begin{theorem}\label{prop:qcdicd}
	The following two kinds of structures are equivalent:
	\begin{enumerate}
		\item Quantum CD-categories;
		\item ICD-categories.
	\end{enumerate}
	More precisely, to any quantum CD-category $\cat{C}'$~\cite[Definition~3.4]{quantum-markov}, we can associate an ICD-category, denoted by $J(\cat{C}')$, by restricting to even morphisms and considering the involution
	\begin{equation}
		\label{eq:invo_qcd}
		\invo{\phi} \coloneqq *_B\comp \phi\comp *_A 
	\end{equation}
	for every morphism $\phi \colon A \to B$.\footnote{Note the similarity with the involution given in \cref{def:mainexample}.}
	Conversely, there is a construction $\cat{C} \mapsto Q(\cat{C})$ turning every ICD-category $\cat{C}$ into a quantum CD-category $Q(\cat{C})$ such that:
	\begin{enumerate}
		\item $JQ(\cat{C})=\cat{C}$ for every ICD-category $\cat{C}$;
		\item $QJ(\cat{C}')$ is canonically strictly monoidally isomorphic to $\cat{C}'$ for any quantum CD-category $\cat{C}'$. 
	\end{enumerate} 
	Furthermore, the constructions $J$ and $Q$ are mutually inverse.
\end{theorem}

\begin{proof}
	It is a direct check that the category $J(\cat{C}')$ associated to a quantum CD-category $\cat{C'}$ is an ICD-category.
	Indeed, the even morphisms are closed under composition and under tensoring with morphisms of the same parity, and the copy and delete maps of $\cat{C}'$ are even.
	The formula~\eqref{eq:invo_qcd} defines a strict monoidal involution on the even subcategory: the equation $*^2=\id$ gives $\involution^2=\id$, and (QCD3) gives the compatibility with tensor products and delete maps.
	The remaining compatibility of the copy maps with the involution is exactly the last equation in (QCD1).

	Let us now explain in detail how we get a quantum CD-category $Q(\cat{C})$ from an ICD-category $\cat{C}$.
	This $Q(\cat{C})$ is defined as follows:
	\begin{itemize}
		\item Its objects are the ones of $\cat{C}$.
		\item Its hom-sets are
			\[
				{Q(\cat{C})}(A,B)\coloneqq {\cat{C}}(A,B) \times \Z_2.
			\]
			We write $\phi$ instead of $(\phi,0)$ and $\phi^*$ instead of $(\phi,1)$. 
			The former are the even morphisms while the latter are the odd ones.
		\item Composition is given by the rule
			\[
				(\psi, s)\comp(\phi, t)
				\coloneqq
				(\psi\comp \involution^s(\phi),s+t),
			\]
			where $\involution^0$ denotes the identity functor and $\involution^1 = \involution$.
			The identity morphisms are $(\id,0)$.
			For associativity, we compute both bracketings of a composable triple of morphisms,
			\begin{align*}
				\big((\chi,r)\comp(\psi,s)\big)\comp(\phi,t)
				& =
				(\chi\comp\involution^r(\psi),r+s)\comp(\phi,t)
				=
				(\chi\comp\involution^r(\psi)\comp\involution^{r+s}(\phi),r+s+t).
				\\
				(\chi,r)\comp\big((\psi,s)\comp(\phi,t)\big)
				& =
				(\chi,r)\comp(\psi\comp\involution^s(\phi),s+t)
				=
				(\chi\comp\involution^r(\psi\comp\involution^s(\phi)),r+s+t).
			\end{align*}
			The right-hand expressions coincide by functoriality of $\involution^r$ and $\involution^2=\id_{\cat{C}}$.

		\item The graded symmetric monoidal structure is such that the tensor product of equal parity is the corresponding tensor product in $\cat{C}$,
			\[
				(\psi, s)\tensor(\psi', s)
				=
				(\psi\tensor\psi', s).
			\]
			The interchange law follows from this formula and the fact that $\involution$ is strict monoidal,
			\begin{align*}
				\big((\psi,s)\tensor(\psi',s)\big)\comp
				\big((\phi,t)\tensor(\phi',t)\big)
				&=
				(\psi\tensor\psi',s)\comp(\phi\tensor\phi',t)
				\\
				&=
				((\psi\tensor\psi')\comp\involution^s(\phi\tensor\phi'),s+t)
				\\
				&=
				(\psi\comp\involution^s(\phi)\tensor \psi'\comp\involution^s(\phi'),s+t)
				\\
				&=
				(\psi,s)\comp(\phi,t)\tensor
				(\psi',s)\comp(\phi',t).
			\end{align*}
			The associators, unitors and swap morphisms are those of $\cat{C}$ itself included in $Q(\cat{C})$ as even morphisms. Their coherence equations are the ones of $\cat{C}$, while naturality for morphisms of a common parity follows from the corresponding naturality equation in $\cat{C}$ and the self-adjointness of these coherence morphisms.
			For example, for the associator $\alpha$ and morphisms of common parity $s$, one side of naturality is
			\begin{align*}
				(\alpha,0)\comp\big((\phi,s)\tensor((\psi,s)\tensor(\omega,s))\big)
				&=
				(\alpha,0)\comp(\phi\tensor(\psi\tensor\omega),s)
				\\
				&=
				(\alpha\comp(\phi\tensor(\psi\tensor\omega)),s),
			\end{align*}
			while the other side is
			\begin{align*}
				\big(((\phi,s)\tensor(\psi,s))\tensor(\omega,s)\big)\comp(\alpha,0)
				&=
				((\phi\tensor\psi)\tensor\omega,s)\comp(\alpha,0)
				\\
				&=
				(((\phi\tensor\psi)\tensor\omega)\comp\involution^s(\alpha),s)
				\\
				&=
				(((\phi\tensor\psi)\tensor\omega)\comp\alpha,s).
			\end{align*}
			The last equality uses the self-adjointness of $\alpha$.
			The two right-hand sides are equal by naturality of $\alpha$ in $\cat{C}$.
			Analogous arguments show naturality of the unitors and the swap morphisms.
	\end{itemize}
	To equip $Q(\cat{C})$ with a quantum CD-category structure,
	it is useful to simplify the notation a bit and write $\phi=(\phi,0)$ and $\phi^*=(\phi,1)$ for the morphisms of $Q(\cat{C})$.
	Like this, composition is given by
	\[
		\psi\comp \phi\coloneqq\psi\comp \phi, \qquad \psi\comp \phi^* \coloneqq (\psi\comp \phi)^*, \qquad \psi^*\comp \phi\coloneqq (\psi\comp \invo{\phi})^*, \qquad \psi^*\comp \phi^* \coloneqq \psi\comp \invo{\phi}.
	\]
	In particular, we have $\phi^*=\phi\comp \id^*$, and the equations are then constructed precisely such that $\invo{\phi} = \id^* \comp \phi \comp \id^*$.
	Since $\id^* \comp \id^* = \id$, this also shows that an even morphism $\phi$ is self-adjoint in $\cat{C}$ in our sense if and only if $\phi \comp \id^* = \id^* \comp \phi$,
	i.e.~if $\phi$ is $*$-preserving.
	For the tensor product, we have
	\begin{equation}\label{eq:defqcdstar}
			\phi^* \tensor \psi^* \coloneqq (\phi\tensor \psi)^*.
	\end{equation}

	We now show that $Q(\cat{C})$ is a quantum CD-category with respect to $*_A \coloneqq \id_A^*$ on every object $A$. First of all, we note $\invo{\phi}=* \comp \phi\comp *$ for all morphisms $\phi$.
	We then prove that Parzygnat's (QCD3) axiom holds:
	\begin{itemize}
		\item $\id^*\comp \id^* = \id \comp \invo{\id} = \invo{\id} =\id$,
		\item $\id_{A \tensor B} ^* = (\id_A\tensor \id_B )^*=\id_A^* \tensor \id_B^* $.
		\item Moreover, $\discard\comp  *= *\comp\discard$ because $\discard=\invo{\discard}$.
	\end{itemize}
	The only thing we still need to prove is the last condition of Parzygnat's (QCD1), i.e.
	\[
		\cop\comp * = \swap\comp  *\comp \cop \!.
	\]
	This follows from
	\[
		\cop \comp * = *\comp *\comp \cop\comp * = * \comp \invo{\cop} = * \comp \swap\comp \cop= \swap\comp  *\comp \cop,
	\]
	where in the last equality we used the self-adjointness of the swap morphism.

	From this construction, we immediately conclude that $JQ(\cat{C}) = \cat{C}$ holds by definition.
	Conversely, starting from a quantum CD-category $\cat{C}'$, note that all odd morphisms $\psi$ are of the form $\phi \comp *$ where $\phi$ is an even morphism. Indeed, it suffices to take $\phi\coloneqq\psi \comp *$ since $*\comp * = \id$. More precisely, we have a bijection
	\[
		{\cat{C}'}(A,B)^{\text{even}}\cong {\cat{C}'}(A,B)^{\text{odd}}
	\]
	given by precomposition with $*$. This implies that the even morphisms completely describe the morphisms of $\cat{C}'$, in the sense that there is a canonical isomorphism $F$ of categories between $\cat{C}'$ and $QJ(\cat{C}')$ given by the identity on objects and on even morphisms, and on odd morphisms the composite bijection
	\[
		{\cat{C}'}(A,B)^{\text{odd}}\cong {J(\cat{C}')}(A,B)\times \lbrace 1 \rbrace = {QJ(\cat{C}')} (A,B)^{\text{odd}}
	\]
	for any two objects $A$ and $B$.
	Compatibility of $F$ with identities and composition is precisely the statement that composition in $QJ(\cat{C}')$ was defined by transporting composition along this even--odd bijection.
	Equivalently, the four parity cases displayed above are the four possible composition cases in $\cat{C}'$, rewritten only in terms of the even part and the star maps.
	Furthermore, the graded symmetric monoidal structure on $QJ(\cat{C}')$ is exactly the one induced by transporting the one on $\cat{C}'$ via $F$.
	For even morphisms this is immediate.
	For odd morphisms, let $\phi^*$ and $\psi^*$ be any odd morphisms in $\cat{C}'$. Then 
	\begin{equation*}
		\begin{tikzpicture}
	\begin{pgfonlayer}{nodelayer}
		\node [style=morphism] (0) at (-8.5, 0) {$\phi^*$};
		\node [style=morphism] (1) at (-6.5, 0) {$\psi^*$};
		\node [style=none] (2) at (-8.5,-2.5) {};
		\node [style=none] (3) at (-8.5,2.5) {};
		\node [style=none] (4) at (-6.5,-2.5) {};
		\node [style=none] (5) at (-6.5,2.5) {};
		\node [style=none] (6) at (-5, 0) {$=$};
		\node [style=morphism] (7) at (-3.5,-1) {$*$};
		\node [style=morphism] (8) at (-1.5,-1) {$*$};
		\node [style=morphism] (9) at (-3.5,1) {$\phi$};
		\node [style=morphism] (10) at (-1.5,1) {$\psi$};
		\node[style=none] (11) at (-3.5,-2.5) {};
		\node[style=none] (12) at (-3.5,2.5) {};
		\node[style=none] (13) at (-1.5,-2.5) {};
		\node[style=none] (14) at (-1.5,2.5) {};
		\node[style=none] (15) at (0,0) {$=$};
		\node[style=none,font=\scriptsize] (29) at (0,0.5) {$(\clubsuit)$};
		\node[style=morphism] (16) at (1.5,1) {$\phi$};
		\node [style=morphism] (17) at (3.5,1) {$\psi$};
		\node [style=morphism] (18) at (2.5,-1) {$\quad\ * \quad\ $};
		\node [style=none] (19) at (1.5,-2.5) {};
		\node [style=none] (20) at (1.5,2.5) {};
		\node [style=none] (21) at (3.5,-2.5) {};
		\node [style=none] (22) at (3.5,2.5) {};
		\node [style=none] (23) at (5,0) {$=$};
		\node [style=morphism] (24) at (7.5,0) {$(\phi \tensor \psi)^* $};
		\node [style=none] (25) at(6.5,-2.5) {};
		\node [style=none] (26) at (6.5,2.5) {};
		\node [style=none] (27) at (8.5,-2.5) {};
		\node [style=none] (28) at (8.5,2.5) {};
	\end{pgfonlayer}
	\begin{pgfonlayer}{edgelayer}
		\draw (2) to (3);
		\draw (4) to (5);
		\draw (11) to (12);
		\draw (13) to (14);
		\draw (19) to (20);
		\draw (21) to (22);
		\draw (25) to (26);
		\draw (27) to (28);
	\end{pgfonlayer}
\end{tikzpicture}
	\end{equation*}
	where $(\clubsuit)$ holds by (QCD3). Therefore, \eqref{eq:defqcdstar} still holds.
	The coherence morphisms are even, and hence are preserved by $F$ on the nose.
	We conclude that $F$ is a strict symmetric monoidal isomorphism $\cat{C}' \to QJ(\cat{C}')$.
\end{proof}

\begin{remark}
	\begin{enumerate}
		\item \Cref{prop:qcdicd} tells us that quantum CD-categories and ICD-categories are different ways to talk about the same concept.
			This raises the question of which of these two formulations is preferable, if any.
			We think that ICD-categories should be preferred for defining and working with concrete examples, since one only needs to consider even morphisms.
			On the other hand, quantum CD-categories facilitate more general diagrammatic proofs by allowing the star $*$ to appear in string diagrams. 
			For example, for every morphism $\phi\colon A \to B$ in an ICD-category, we can write
			\begin{equation*}
				\begin{tikzpicture}
	\begin{pgfonlayer}{nodelayer}
		\node [style=bin] (0) at (2, -1.3) {$*$};
		\node [style=bin] (0) at (2, 1.3) {$*$};
		\node [style=none] (4) at (2, 2.25) {};
		\node [style=none] (5) at (2, -2.25) {};
		\node [style=none] (6) at (2, 2.75) {$B$};
		\node [style=none] (7) at (2, -2.75) {$A$};
		\node [style=morphism] (8) at (-2, 0) {$\invo{\phi}$};
		\node [style=none] (13) at (0, 0) {$=$};
		\node [style=morphism] (9) at (2, 0) {$\phi$};
		\node [style=none] (14) at (-2, 2.25) {};
		\node [style=none] (15) at (-2, 2.75) {$B$};
		\node [style=none] (16) at (-2, -2.25) {};
		\node [style=none] (17) at (-2, -2.75) {$A$};
	\end{pgfonlayer}
	\begin{pgfonlayer}{edgelayer}
		\draw (0) to (4.center);
		\draw (0) to (5.center);
		\draw (14.center) to (16.center);
	\end{pgfonlayer}
\end{tikzpicture}
			\end{equation*}
			where the circled $*$ denotes the star of the associated quantum CD-category.
			The problem with this approach is that one needs to be careful not to tensor an even with an odd morphism, as this is not defined.
			Moreover, for the purposes of this paper, using the star $*$ in string diagrams is not needed.
		\item While every CD-category also gives rise to a quantum CD-category, it seems unlikely that its odd morphisms will be of any interest at all.
			This suggests that ICD-categories encode classical examples better than quantum CD-categories do, and therefore should be slightly preferred.
	
			For example, consider the Markov category $\finstoch$ with finite sets as objects and stochastic matrices as morphisms~\cite[Example~2.5]{fritz2019synthetic}.
		Then $Q(\finstoch)$ can also be described as the quantum Markov category whose odd morphisms are matrices with \emph{nonpositive} entries such that any column sums to $-1$ with $*_A \coloneqq -\id_A$.
	\end{enumerate}
\end{remark}

\begin{remark}[ICD-categories and involutive symmetric monoidal categories]\label{rem:ismc}
	Symmetric monoidal categories equipped with an involution have already appeared in the literature: see, for instance, \cite[Definition~4]{jacobsinvolutive}.
	In particular, also Parzygnat wondered whether quantum CD-categories could be described using symmetric monoidal categories with an involution~\cite[Question~3.24]{quantum-markov}.
	Although our ICD-categories are involutive symmetric monoidal categories (with additional structure), our conditions on the involution are stronger than what is generally required in that context: our requirement $A = \invo{A}$ for all objects $A$ is not satisfied for typical involutive symmetric monoidal categories.
	A good example is the category of complex vector spaces, where $\invo{A}$ for a vector space $A$ is the complex conjugate vector space, which has the same vectors as $A$ but conjugate scalar multiplication,
	\[
		\lambda \cdot_{\invo{A}} a \coloneqq \invo{\lambda} \cdot_{A} a.	
	\]
	Clearly this category does not respect the equality $A = \invo{A}$ required for hom-involutive symmetric monoidal categories.
	
	It is conceivable that using complex conjugation as involution on pre-C*-algebras could result in an interesting approach different from ours, and more similar to the one envisioned at~\cite[Question~3.24]{quantum-markov}.
	This amounts to equipping every object $A$ with an involution morphism $*_A \colon A \to \invo{A}$.
	Instead of going in this direction, we prefer to avoid fixing such a star operation for every object, since this results in a more complex setting and we currently see no real advantage in such additional complexity. 
	For our purposes, it is enough that the axioms fix how the involution acts on copy and delete morphisms, which is exactly what we have in \Cref{def:icd}, and allows us to talk about self-adjoint morphisms.
	Moreover, \cref{prop:qcdicd} ensures the possibility of using the star operation itself in string diagrams also for ICD-categories, suggesting that these categories provide a sufficiently powerful framework without the need to explore alternative possibilities at this stage. 
\end{remark}

\subsection{Deterministic morphisms and ICD-functors}\label{sec:icdfun}
Now that ICD-categories are defined, it is natural to ask what the corresponding notion of functor between such categories is.
For our framework, the introduction of such a concept is actually crucial, because string diagrams do not illustrate associators and unitors.
We are therefore allowed to use them as usual \emph{only if} we have a strictification theorem.

In the consideration of functors between ICD-categories, an important role is played by \emph{deterministic morphisms}. 
As we will see in this section, coherence morphisms of ICD-functors turn out to be deterministic (\cref{rem:cohF_det}), while ICD-natural transformations are defined as having deterministic components. 
For this reason, we first define what a deterministic morphism is.

\begin{definition}
	Let $\cat{C}$ be an ICD-category. A \newterm{deterministic morphism} is a self-adjoint total morphism (\cref{def:selfadj_total}) $\phi\colon A \to B$ such that 
	\begin{equation}\label{eq:deterministic}
		\begin{tikzpicture}
	\begin{pgfonlayer}{nodelayer}
		\node [style=none] (0) at (-3, -2) {};
		\node [style=bn] (1) at (-3, 0) {};
		\node [style=none] (2) at (-2, 1) {};
		\node [style=none] (3) at (-4, 1) {};
		\node [style=none] (4) at (-4, 1) {};
		\node [style=none] (5) at (-2, 1.25) {};
		\node [style=none] (6) at (-4, 1.25) {};
		\node [style=none] (12) at (-2, 1.5) {};
		\node [style=none] (13) at (-4, 1.5) {};
		\node [style=none] (16) at (0, 0) {$=$};
		\node [style=morphism] (17) at (-3, -1) {$\phi$};
		\node [style=none] (18) at (-2, 2) {$B$};
		\node [style=none] (19) at (-4, 2) {$B$};
		\node [style=none] (22) at (-3, -2.5) {$A$};
		\node [style=none] (27) at (3, -2) {};
		\node [style=bn] (28) at (3, -0.75) {};
		\node [style=none] (29) at (4, 0.25) {};
		\node [style=none] (30) at (2, 0.25) {};
		\node [style=none] (31) at (2, 0.25) {};
		\node [style=none] (32) at (4, 0.5) {};
		\node [style=morphism] (33) at (2, 0.5) {$\phi$};
		\node [style=none] (34) at (4, 1.5) {};
		\node [style=none] (35) at (2, 1.5) {};
		\node [style=morphism] (36) at (4, 0.5) {$\phi$};
		\node [style=none] (37) at (4, 2) {$B$};
		\node [style=none] (38) at (2, 2) {$B$};
		\node [style=none] (39) at (3, -2.5) {$A$};
	\end{pgfonlayer}
	\begin{pgfonlayer}{edgelayer}
		\draw [style=none] (0.center) to (1);
		\draw [style=none, bend right=45] (1) to (2.center);
		\draw [style=none, bend left=45] (1) to (3.center);
		\draw (12.center) to (2.center);
		\draw (13.center) to (6.center);
		\draw (4.center) to (6.center);
		\draw [style=none] (27.center) to (28);
		\draw [style=none, bend right=45] (28) to (29.center);
		\draw [style=none, bend left=45] (28) to (30.center);
		\draw (34.center) to (29.center);
		\draw (35.center) to (33);
		\draw (31.center) to (33);
	\end{pgfonlayer}
\end{tikzpicture}
	\end{equation}
\end{definition} 

It is easy to see that deterministic morphisms form a subcategory of $\cat{C}$, and we denote it by $\cat{C}_{\det}$.

\begin{example}\label{ex:homnotbound}
	In $\calgmin$ and $\calgmax$, a morphism $\phi$ is deterministic if and only if $\opm{\phi}{}$ is a $*$-homomorphism, i.e.~an algebra homomorphism that is also self-adjoint.

	Note that, contrary to the case of C*-algebras, $*$-homomorphisms between pre-C*-algebras can be unbounded. 
	For example, let us set $A = \mathbb{C}$ and $B= \mathbb{C}[x]$. 
	The latter is a pre-C*-algebra by identifying each polynomial $f \in \mathbb{C}[x]$ with the associated function $x \mapsto f(x)$ in the C*-algebra $C([0,1])$ of complex-valued continuous maps on the unit interval $[0,1]$.
	By the Stone--Weierstrass theorem, such an identification makes $\mathbb{C}[x]$ a dense $*$-subalgebra of $C([0,1])$.
	In particular, this equips $\mathbb{C}[x]$ with the C*-norm $\norm{f}\coloneqq \sup_{x\in [0,1]} \abs{f(x)}$.
	Let us now take $\opm{\phi}{}\colon \mathbb{C}[x] \rightsquigarrow \mathbb{C}$ given by evaluating polynomials at any fixed real number outside of $[0,1]$. 
	This is a $*$-homomorphism, but it cannot be bounded, because by Gelfand duality every bounded $*$-homomorphism $C([0,1]) \rightsquigarrow \mathbb{C}$ is given by evaluation at a point of $[0,1]$.

	A more abstract example of an unbounded $*$-homomorphism is given by the identity map
	\[
		\opm{\iota}{} \colon C \odot_{\min} D \rightsquigarrow C \odot_{\max} D
	\]
	for some choices of pre-C*-algebras $C$ and $D$.	Indeed, if $\opm{\iota}{}$ was bounded for all choices of ${C}$ and ${D}$, then we would get a $*$-homomorphism
	\[
		\close{\opm{\iota}{}}\colon \closed{C} \tensor_{\min} \closed{D} \rightsquigarrow \closed{C} \tensor_{\max} \closed{D}
	\]
	by \cref{lem:trick}, where $\closed{C}$ and $\closed{D}$ are the completions of $C$ and $D$.
	Moreover, $\norm{\opm{\iota}{}} = \norm{\close{\opm{\iota}{}}}=1$ by \cref{lem:trick} and the fact that bounded $*$-homomorphisms have norm 1 (cf.~\cref{prop:pu_bounded}; see \cite[Proposition~2.1]{paulsen2002} for a proof).
	In particular, we obtain that $\norm{ \sum_i c_i \tensor d_i }_{\max} \le \norm{ \sum_i c_i \tensor d_i }_{\min}$ for all elements of the algebraic tensor product.
	Since the reverse inequality always holds, this would mean that the two norms coincide, which is generally not the case~\cite[Theorem~6]{takesaki}. 
	This therefore contradicts our assumption of $\opm{\iota}{}$ being bounded for all $C$ and $D$.
\end{example}

\begin{example}\label{ex:det_coherence}
	Associators, unitors, and swap morphisms are deterministic, the proof being analogous to the one for classical Markov categories~\cite[Lemma~10.12]{fritz2019synthetic}.
	The same is true for $\discard$ simply because $\cop$ is total.
	In particular, $\cat{C}_{\det}$ is a semicartesian\footnote{Recall that this means that its monoidal unit is terminal.} symmetric monoidal subcategory of $\cat{C}$, since tensor products and composites of deterministic morphisms are easily seen to be deterministic.
	This generalizes the known statement for Markov categories~\cite[Remark~10.13]{fritz2019synthetic} and was already briefly pointed out by Parzygnat~\cite[Remark~3.19]{quantum-markov} for quantum Markov categories, although the question of determinism of the coherence isomorphisms was not addressed.

	An important difference between the classical and the quantum setting is that a copy morphism $\cop_A$ is deterministic if and only if $A$ is a classical object in the sense of \cref{def:classical}, which we will discuss in \cref{sec:classical_compatible}. 
\end{example}

We turn our attention to functors.
As for monoidal categories in general, it may be of interest to consider different versions like strong, lax or oplax functors between ICD-categories.
For our purposes, the strong version will be sufficient. 
Let us recall that a \newterm{strong symmetric monoidal functor} between two symmetric monoidal categories $\cC$ and $\cD$ is a functor $F \colon \cC \to \cD$ together with a natural isomorphism
\[
\phi_{A,B} \colon FA \tensor FB \to F(A \tensor B)
\]
between functors $\cC \times \cC \to \cD$ and an isomorphism $\phi_I \colon I \to F(I)$, satisfying compatibility with associators, unitors, and swap morphisms. The isomorphisms $\phi_{A,B}$ and $\phi_I$ are called \newterm{coherence isomorphisms}.
\begin{definition}
	Let $\cat{C}$ and $\cD$ be hom-involutive symmetric monoidal categories. A \newterm{hom-involutive strong symmetric monoidal functor} is a strong symmetric monoidal functor $F\colon \cat{C} \to \cD$ such that
	\begin{enumerate}
		\item $F$ strictly commutes with the involution: $F(\invo{\phi}) = \invo{F(\phi)}$ for every morphism $\phi$ in $\cat{C}$;
		\item The coherence isomorphisms $\phi_{A,B}$ and $\phi_I$ are self-adjoint.
	\end{enumerate}
\end{definition}

This concept specializes the general notion of involutive monoidal functor between involutive monoidal categories~\cite[Definition~5]{jacobsinvolutive}, with the difference that we require the additional coherence isomorphisms $F(\invo{A}) \stackrel{\cong}{\longrightarrow} \invo{F(A)}$ to be identities, which is natural in our setting.

\begin{definition}\label{def:strongICDfun}
	Let $\cat{C}$ and $\cD$ be ICD-categories. Then $F \colon \cat{C} \to \cD$ is a \newterm{strong ICD-functor} if
	\begin{enumerate}
		\item $F$ is a hom-involutive strong symmetric monoidal functor;
		\item The coherence isomorphisms $\phi_{A,B}$ and $\phi_I$ additionally satisfy the commutativity of  
	\begin{equation}\label{eq:strong_icd_fun}
	\begin{tikzcd}
		& FA \ar[rd,"F(\cop_A)"]\ar[ld,"\cop_{FA}" above left] & \\
		FA \tensor FA \ar[rr, "\phi_{A,A}"] && F(A \tensor A)
	\end{tikzcd}\qquad 	\text{and} \qquad 
	\begin{tikzcd}
		& FA \ar[rd,"F(\discard_A)"]\ar[ld,"\discard_{FA}" above left] & \\
		I \ar[rr, "\phi_{I}"] && F(I).
	\end{tikzcd}
\end{equation}
\end{enumerate}
\end{definition}

If $\cat{C}$ and $\cD$ are CD-categories, then this notion specializes to the existing notion of strong CD-functor (called \emph{strong gs-monoidal functor} in~\cite[Definition~2.5]{fritz2023lax}).

\begin{remark}\label{rem:cohF_det}
	The coherence morphisms of a strong ICD-functor are necessarily deterministic.
	This statement is an instance of~\cite[Theorem~4.7]{fong2020supply}.
\end{remark}
\begin{example}
	The isomorphism between $\calgmin$ and $\calgmax$ is actually a strong ICD-functor (or even strict, by which we mean that the coherence isomorphisms are identities).
\end{example}

For the relevant notion of natural transformation, see also~\cite[Definition~10.14]{fritz2019synthetic} in the Markov categories context.

\begin{definition}
	Given two strong ICD-functors $F,G \colon \cat{C} \to \cD$, an \newterm{ICD-natural transformation} $\eta\colon F \to G$ is a monoidal natural transformation whose components are deterministic. 
\end{definition}

It is then straightforward to see that ICD-categories, strong ICD-functors and ICD-natural transformations form a 2-category.
The invertible 2-cells in this 2-category, which we call \newterm{ICD-natural isomorphisms}, are precisely those ICD-natural transformations whose components are deterministic isomorphisms.
This explains why we require the components of ICD-natural transformations to be deterministic: 
An ICD-natural isomorphism should make the comonoid structures on any two objects between which it interpolates match up.

The equivalences in the 2-category of ICD-categories are the \newterm{ICD-equivalences}.
Concretely, an ICD-equivalence thus is a strong ICD-functor $F\colon \cat{C} \to \cD$ for which there exist a strong ICD-functor $G\colon \cD\to \cat{C}$ and ICD-natural isomorphisms $\eta\colon F \comp G \to \id_{\cD}$ and $\mu \colon G \comp F \to \id_{\cat{C}}$.
As for ordinary categories and for Markov categories~\cite[Proposition~10.16]{fritz2019synthetic}, ICD-equivalences can also be characterized more concretely.

\begin{proposition}\label{prop:icd_equiv}
	A strong ICD-functor $F \colon \cat{C} \to \cD$ is an ICD-equivalence if and only if it is fully faithful and for every $B \in \cD$ there exist $A\in \cat{C}$ and a deterministic isomorphism $B \to FA$.
\end{proposition}

\begin{proof}
	If $F$ is an ICD-equivalence, then it is in particular an equivalence of the underlying categories and hence fully faithful.
	For $B\in\cD$, take $A\coloneqq GB$ for an essential inverse $G$ as in the definition.
	Then the inverse component $\eta_B^{-1}\colon B\to FGB$ is a deterministic isomorphism.

	Conversely, suppose that $F$ is fully faithful and satisfies the stated deterministic essential surjectivity condition.
	The usual construction for equivalences of Markov categories~\cite[Proposition~10.16]{fritz2019synthetic} gives a strong symmetric monoidal inverse $G\colon\cD\to\cat{C}$ and a monoidal natural isomorphism
	\[
		\epsilon\colon F\comp G \to \id_{\cD}
	\]
	whose components are deterministic isomorphisms.
	The strong monoidal structure on $G$ is the one transported across $F$ and $\epsilon$:
	the coherence isomorphisms $\gamma_{A,B}\colon GA\tensor GB\to G(A\tensor B)$ and $\gamma_I\colon I\to G(I)$ are the unique morphisms whose images under $F$ satisfy
	\[
		\epsilon_{A\tensor B}\comp F(\gamma_{A,B})\comp \phi_{GA,GB} = \epsilon_A\tensor\epsilon_B,
		\qquad
		\epsilon_I\comp F(\gamma_I)\comp \phi_I = \id_I,
	\]
	where $\phi_{-,-}$ and $\phi_I$ are the coherence isomorphisms of $F$.
	Since all other morphisms involved are self-adjoint,
	these equations also imply that these coherences are self-adjoint, as required for $G$ to be hom-involutive strong monoidal.

	We show next that $G$ is hom-involutive.
	To this end, let $f\colon A\to B$ be a morphism of $\cD$.
	Naturality of $\epsilon$ on $\invo{f}$ and $f$ as well as self-adjointness of $\epsilon_A$ and $\epsilon_B$ give
	\[
		\epsilon_B\comp FG(\invo{f})
		= \invo{f}\comp \epsilon_A
		= \invo{f\comp \epsilon_A}
		= \invo{\epsilon_B\comp FG(f)}
		= \epsilon_B\comp \invo{FG(f)}.
	\]
	Since $\epsilon_B$ is an isomorphism, $FG(\invo{f})=\invo{FG(f)}$.
	Using that $F$ preserves the involution, this says
	\[
		F(G(\invo{f})) = F(\invo{G(f)}),
	\]
	and faithfulness of $F$ implies $G(\invo{f})=\invo{G(f)}$.
	Thus $G$ is hom-involutive.

	We now verify the two diagrams in~\eqref{eq:strong_icd_fun} for $G$.
	For the copy morphism of $A\in\cD$, we obtain
	\[
		\epsilon_{A\tensor A}\comp FG(\cop_A)
		= \cop_A\comp \epsilon_A
		= (\epsilon_A\tensor\epsilon_A)\comp \cop_{FGA}
		= \epsilon_{A\tensor A}\comp F(\gamma_{A,A})\comp F(\cop_{GA}).
	\]
	The first equality is naturality of $\epsilon$, the second is determinism of $\epsilon_A$, and the third is monoidality of $\epsilon$ together with the copy-preservation axiom for $F$.
	Since $\epsilon_{A\tensor A}$ is an isomorphism and $F$ is faithful, this is exactly the copy-preservation axiom for $G$.
	The delete axiom is proved in the same way, using totality of $\epsilon_A$ and the unit part of monoidality of $\epsilon$.
	Therefore $G$ is a strong ICD-functor.

	Finally, define $\mu_A\colon GFA\to A$ as the unique morphism with $F(\mu_A)=\epsilon_{FA}$.
	Full faithfulness of $F$ and monoidality of $\epsilon$ imply that $\mu$ is a monoidal natural isomorphism.
	Its components are deterministic because the components $\epsilon_{FA}$ are deterministic and $F$ reflects determinism.
	Hence $\epsilon$ and $\mu$ are ICD-natural isomorphisms, so $F$ is an ICD-equivalence.
\end{proof}

\subsection{Cofree ICD-categories}\label{sec:cofree}
Before addressing strictification, let us present some examples beyond $\calgmin$ and $\calgmax$.

Cofree ICD-categories are ICD-categories that possess a universal property (\cref{prop:cofree_universal}) dual to that of free groups, hence the name. 
As we will see shortly, the objects of this category are involutive comonoids.

Recall that a \newterm{comonoid} in a symmetric monoidal category $\cat{C}$ is a triple $(A,\mu_A, \epsilon_A)$ where $\mu_A \colon A \to A \tensor A$ and $\epsilon_A\colon A \to I$ are like $\cop$ and $\discard$ in the sense that they satisfy \eqref{eq:icd_def_comonoid}, i.e.~the comultiplication is associative and we have left and right unit laws. In the involutive setting, it is natural to impose an additional compatibility condition with the involution.

\begin{definition}\label{def:involutive_comonoid}
	A comonoid $(A,\mu_A, \epsilon_A)$ in a hom-involutive symmetric monoidal category $\cat{C}$ is \newterm{involutive} if \eqref{eq:icd_def_involution} holds, i.e.
	\[
		\invo{\mu_A}=\swap_{A,A}\comp \mu_A\qquad \text{and} \qquad \invo{\epsilon_A}=\epsilon_A.
	\]
\end{definition} 

\begin{example}
If $\cat{C}$ carries the trivial involution, then an involutive comonoid is just a commutative comonoid.
\end{example}

\begin{definition}
	For a hom-involutive symmetric monoidal category $\cat{C}$, its \newterm{cofree ICD-category} $\cofree(\cat{C})$ is the category whose objects are involutive comonoids $(A,\mu_A, \epsilon_A)$ and morphisms are the ones in $\cat{C}$. This category becomes symmetric monoidal by lifting the tensor of $\cat{C}$ to involutive comonoids as $\epsilon_{A \tensor B}\coloneqq \epsilon_A \tensor \epsilon_B$ and 
	\begin{equation*}
		\begin{tikzpicture}
	\begin{pgfonlayer}{nodelayer}
		\node [style=none] (0) at (-5, 2.5) {$A\tensor B$};
		\node [style=none] (1) at (3, -2) {};
		\node [style=morphism] (2) at (3, -0.5) {$\quad \mu_{A}\quad $};
		\node [style=none] (3) at (4, 0) {};
		\node [style=none] (4) at (4, -0.5) {};
		\node [style=none] (5) at (2, -0.5) {};
		\node [style=none] (6) at (7, -2) {};
		\node [style=morphism] (7) at (7, -0.5) {$\quad \mu_B \quad $};
		\node [style=none] (8) at (8, 0) {};
		\node [style=none] (9) at (8, -0.5) {};
		\node [style=none] (10) at (6, 0) {};
		\node [style=none] (11) at (2, 2) {};
		\node [style=none] (12) at (4, 2) {};
		\node [style=none] (13) at (6, 2) {};
		\node [style=none] (14) at (8, 2) {};
		\node [style=none] (15) at (0, 0) {$:=$};
		\node [style=morphism] (16) at (-3.5, -0.5) {$\quad \mu_{A \tensor B}\quad $};
		\node [style=none] (17) at (-3.5, -2) {};
		\node [style=none] (18) at (-5, -0.5) {};
		\node [style=none] (19) at (-2, -0.5) {};
		\node [style=none] (20) at (-2, 2.5) {$A\tensor B$};
		\node [style=none] (21) at (-3.5, -2.5) {$A\tensor B$};
		\node [style=none] (22) at (-5, 2) {};
		\node [style=none] (23) at (-2, 2) {};
		\node [style=none] (24) at (2, 2.5) {$A$};
		\node [style=none] (25) at (4, 2.5) {$B$};
		\node [style=none] (26) at (6, 2.5) {$A$};
		\node [style=none] (27) at (8, 2.5) {$B$};
		\node [style=none] (28) at (3, -2.5) {$A$};
		\node [style=none] (29) at (7, -2.5) {$B$};
	\end{pgfonlayer}
	\begin{pgfonlayer}{edgelayer}
		\draw [style=none] (1.center) to (2.center);
		\draw [style=none] (3.center) to (4.center);
		\draw [style=none] (6.center) to (7);
		\draw [style=none] (7) to (10.center);
		\draw [style=none] (7) to (9.center);
		\draw [style=none] (11.center) to (5.center);
		\draw [style=none] (8.center) to (14.center);
		\draw [style=none, in=-90, out=90] (3.center) to (13.center);
		\draw (17.center) to (16);
		\draw (16) to (18.center);
		\draw (16) to (19.center);
		\draw (23.center) to (19.center);
		\draw (18.center) to (22.center);
		\draw [style=protected, in=90, out=-90] (12.center) to (10.center);
	\end{pgfonlayer}
\end{tikzpicture}
	\end{equation*}
	In particular, the monoidal unit is $(I,\id_I, \id_I)$.
\end{definition}

In light of \cref{prop:qcdicd}, the following two propositions answer Parzygnat's~\cite[Question~3.24]{quantum-markov}, where it was asked whether involutive comonoids form a quantum Markov category and whether the objects in a quantum Markov category are involutive comonoids in a suitable sense.

\begin{proposition}\label{prop:cof_icd}
	For any hom-involutive symmetric monoidal category $\cat{C}$, the cofree ICD-category $\cofree(\cat{C})$ is an ICD-category, where:
	\begin{enumerate}
		\item The involution $\phi \mapsto \invo{\phi}$ is the one of $\cat{C}$.
		\item The copy and delete morphisms on each object are given by the comonoid structure.
	\end{enumerate} 
\end{proposition}
\begin{proof}
	This follows straightforwardly from the definition of involutive comonoid and the definition of their monoidal product.
\end{proof}

Of course, by restricting further to total morphisms (those which respect the counits), one obtains an involutive Markov category.

We now want to show that the cofree ICD-categories contain all ICD-categories in a similar way as to how any group is contained in the symmetric group over itself. For this, let us briefly mention that an \newterm{ICD-subcategory} $\cD$ of an ICD-category $\cat{C}$ is simply any symmetric monoidal subcategory $\cD \subseteq \cat{C}$ which is closed under the involution and contains all copy and delete morphisms.
It is clear that an ICD-subcategory is then an ICD-category in its own right.
This is analogous to the notion of Markov subcategory~\cite[Definition~A.1.2]{fritz2023supports}.

\begin{proposition}\label{prop:icd_subcat_cofree}
	Let $\cat{C}$ be an ICD-category. Then $\cat{C}$ is a full ICD-subcategory of $\cofree(\cat{C})$.
\end{proposition}
\begin{proof}
	We just consider $\cat{C}$ as a subcategory via $A \mapsto (A, \cop_A , \discard_A)$.
\end{proof}

The cofree ICD-category enjoys also a universal property, as the name suggests.
The relevant universal arrow is the forgetful functor $\operatorname{Forg} : \cofree(\cD) \to \cD$, for any fixed hom-involutive symmetric monoidal category $\cD$.\footnote{If $\cD$ is an ICD-category, $\operatorname{Forg}$ is generally not an ICD-functor.
Indeed, $\operatorname{Forg}$ is an ICD-functor if and only if every object of $\cD$ admits exactly \emph{one} involutive comonoid structure. To see that this fails in general, it suffices to assume the existence of a noncommutative involutive comonoid $(A, \mu_A, \epsilon_A)$ and take the involutive comonoid given by its involution, $(A, \invo{\mu_A}, \epsilon_A)$.}
 
\begin{proposition}[Universal property of the cofree ICD-category]\label{prop:cofree_universal}
	Let $\cat{C}$ be an ICD-category and let $\cD$ be a hom-involutive symmetric monoidal category.
	Then, any hom-involutive strong symmetric monoidal functor $F:\cat{C} \to \cD$ factors uniquely through $\cofree(\cD)$: there is a commutative diagram
	\begin{equation}
		\label{eq:cofree_universal}
		\begin{tikzcd}
			\cat{C} \ar[r,"\tilde{F}"]\ar[rd,swap,"F"] & \cofree(\cD)\ar[d,"\operatorname{Forg}"]\\
			& \cD
		\end{tikzcd}	
	\end{equation} 
	for a unique strong ICD-functor $\tilde{F}$.
\end{proposition}
\begin{proof}
	It is a standard fact that strong (and more generally oplax) monoidal functors send comonoids to comonoids.
	Applied to the canonical comonoid structure $(A, \cop_A, \discard_A)$ on every object $A$ in $\cat{C}$, this works by defining $\tilde{F}$ as the functor
	\[
		\tilde{F}(A) \coloneqq (FA, \psi_{A,A} \comp F(\cop_A), \psi_I \comp F(\discard_A)), \qquad \tilde{F}(\phi) \coloneqq F(\phi),
	\]
	where $\psi_{A, A} : F(A \tensor A) \to FA \tensor FA$ and $\psi_I : F(I) \to I$ are coherence morphism of $F$, with $\psi_{A,A} = \phi_{A,A}^{-1}$ and $\psi_I = \phi_I^{-1}$ relative to~\eqref{eq:strong_icd_fun}.
	By the same diagrams~\eqref{eq:strong_icd_fun}, it is clear that this is the only choice that makes~\eqref{eq:cofree_universal} commute and has a chance of making $\tilde{F}$ into an ICD-functor.
	Hence we already obtain the uniqueness part of the claim.

	For the existence, we show that $\tilde{F}$ actually lands in $\cofree(\cD)$, i.e.~that the comonoids $\tilde{F}(A)$ are involutive. 
	The first property in \cref{def:involutive_comonoid} follows from the commutativity of the diagram
	\[
		\begin{tikzcd}[column sep=4pc,row sep=3pc]
			F(A) \ar[r,"F(\cop_A)"] \ar[rd,"F(\invo{\cop_A}) \,=\, \invo{F(\cop_A)}" below left] & F(A \tensor A) \ar[d,"F(\swap_{A,A})"] \ar[r,"\psi_{A,A}"] & F(A) \tensor F(A) \ar[d,"\swap_{F(A),F(A)}"] \\
			& F(A \tensor A) \ar[r,"\psi_{A,A} \,=\, \invo{\psi_{A,A}}"] & F(A) \tensor F(A)
		\end{tikzcd}
	\]
	where the commutativity of the triangle is involutivity of the comonoid together with the fact that $F$ strictly commutes with the involution, while the commutativity of the square is the symmetry preservation of $F$. 
	Self-adjointness of the counit $\psi_I F(\discard_A)$ is immediate, since $\psi_I$ and $\discard_A$ and therefore also $F(\discard_A)$ are all self-adjoint.

	Finally, $\tilde{F}$ is clearly a strong symmetric monoidal functor with respect to the same self-adjoint coherence morphisms as $F$, and it commutes with the involution because $F$ does.
\end{proof}

We now investigate opposite categories of $*$-algebras as instances of cofree ICD-categories.

\begin{definition}[Involutive modules]
	Let $R$ be a commutative $*$-ring.
	\begin{enumerate}
		\item An $R$-module $M$ is \newterm{involutive} if it is equipped with an additive map $*:M \to M$ such that 
		\[ 
		m^{**} = m\qquad \text{and}\qquad (rm)^* = r^* m^*
		\]  
		for all $m \in M$ and $r \in R$.
		\item The \newterm{category of involutive modules} $\iMod{R}$ is the category whose objects are involutive modules and whose morphisms are $R$-linear maps, i.e. $f \colon M \to N$ such that 
		\[
		f(m+n)=f(m)+f(n) \qquad \text{and}\qquad f(rm)= r f(m)
		\]
		for all $m,n \in M$ and $r \in R$.
	\end{enumerate}
\end{definition}

In the case $R = \mathbb{C}$ with complex conjugation, $* : M \to M$ is thus required to be antilinear, and involutive modules are then also known as \emph{$*$-vector spaces}.

\begin{remark}\label{rem:inv_mod}
	The category $\iMod{R}$ comes equipped with a natural involution that fixes the objects and is given on morphisms by
	\[
		\invo{{\phi}}(x) \coloneqq {\phi}(x^*)^*.
	\]
	Further, $\iMod{R}$ is a hom-involutive symmetric monoidal category when we consider the symmetric monoidal structure given by the tensor product of $R$-modules with involution the unique antilinear extension of
	\[
		(m \otimes n)^* \coloneqq m^* \otimes n^*.
	\]
\end{remark}

Let us now consider an important example generalizing the idea underlying the main examples (\cref{def:mainexample}). 
To this end, we say that an \newterm{involutive monoid} is just the dual notion of an involutive comonoid. 
In other words, an involutive monoid is an object $A$, together with morphisms $\opm{\mu_A}{} \colon A \tensor A \to A$ and $\opm{\epsilon_A}{} \colon I \to A$, that satisfies \eqref{eq:icd_def_comonoid} and \eqref{eq:icd_def_involution} when read from top to bottom.
\begin{remark}\label{rem:star_alg}
	Unfolding the definitions shows that an {involutive monoid} in $\iMod{R}$ is the same thing as a \emph{$*$-algebra} over $R$.
	Indeed, a monoid simply gives an associative unital algebra $A$ over $R$. 
	In particular, bilinearity of the multiplication follows from $r (x\tensor y)= (rx)\tensor y = x \tensor (ry) \in A \tensor A$ for all $r \in R$ and $x,y \in A$. The requirement of being involutive shows that $A$ is also a $*$-algebra: for $x,y \in A$, $(xy)^*=y^*x^*$, and moreover $1^*=1$.
	Since the opposite category of $*$-algebras over $R$ with $R$-linear maps is therefore just $\cofree(\iMod{R}^{\op})$, we conclude that it is an ICD-category by \Cref{prop:cof_icd}.

	In this category, it is easy to understand the notions of morphisms we already encountered.
	\begin{enumerate}
		\item The total morphisms are the formal opposites of \emph{unital} $R$-linear maps.
		\item The self-adjoint morphisms are the formal opposites of $R$-linear maps which commute with the star operation, $\phi(x)^* = \phi(x^*)$.
		\item The deterministic morphisms are the formal opposites of \emph{$*$-homomorphisms}, i.e.~unital algebra homomorphisms commuting with the star operation.
	\end{enumerate}
\end{remark}

\begin{example}\label{ex:star_alg}
	If $R$ is just $\mathbb{Z}$ with the trivial involution, then the involutive monoids in $\iMod{\mathbb{Z}}$ are precisely the $*$-rings. If $R$ is the complex numbers $\mathbb{C}$ with conjugation, then the involutive monoids in $\iMod{(\mathbb{C},\overline{\,\cdot\,})}$ are exactly the complex $*$-algebras in the usual sense.
	It follows that complex $*$-algebras with linear maps form an ICD-category with respect to the formal opposites of the multiplication and the unit as the copy and delete morphisms, respectively.
	In particular, $\calgmin\cong \calgmax$ are both ICD-subcategories of $\cofree(\iMod{(\mathbb{C},\overline{\,\cdot\,})}^{\op})$ (by forgetting the norm).
\end{example}

\subsection{Strictification} \label{sec:strictification}

To formally justify the use of string diagrams, it is crucial to establish a strictification theorem.
A \newterm{strictification} of a given ICD-category $\cC$ consists of the data of a \newterm{strict ICD-category} $\cC'$, i.e.~an ICD-category whose underlying monoidal category is strict, together with an ICD-equivalence $\cC' \to \cC$ (recall \cref{prop:icd_equiv}).

\begin{theorem}[Strictification for ICD-categories]\label{prop:strictification}
	Every ICD-category is ICD-equivalent to a strict one.
\end{theorem}

The main difficulty in the proof is the fact that an isomorphism in an ICD-category is not necessarily deterministic~\cite[Remark~10.10]{fritz2019synthetic}.
We follow the same arguments as in the Markov category case~\cite[Theorem~10.17]{fritz2019synthetic}.

\begin{proof}
	Let $\cat{C}$ be an ICD-category.
	We first consider a strictification $\cat{C}'_{\det}$ of $\cat{C}_{\det}$ (recall \cref{ex:det_coherence}), and write $F_{\det} : \cat{C}'_{\det} \to \cat{C}_{\det}$ for the strong symmetric monoidal equivalence, with coherence isomorphisms $\varphi_{A,B}\colon F_{\det}A\tensor F_{\det}B\to F_{\det}(A\tensor B)$ and $\varphi_I\colon I\to F_{\det}I$.
	We then define a category $\cat{C}'$ whose objects are those of $\cat{C}'_{\det}$ and whose morphisms are given by
	\[
		\cat{C}'(A,B) \coloneqq \cat{C}(F_{\det}(A), F_{\det}(B)).	
	\]
	Composition is inherited from $\cat{C}$.
	The tensor product on objects is the strict one of $\cat{C}'_{\det}$, while on morphisms it is transported through the coherences:
	\[
		f\tensor g \coloneqq \varphi_{B,B'}\comp (f\tensor g)\comp \varphi^{-1}_{A,A'}
	\]
	for $f\colon A\to B$ and $g\colon A'\to B'$ in $\cat{C}'$, where the right-hand side is computed in $\cat{C}$.
	The coherence axioms for $\varphi$ imply that this makes $\cat{C}'$ into a strict symmetric monoidal category (cf.~\cite[Theorem~10.17]{fritz2019synthetic}).
	Also, the above definition of the hom-sets shows that $\cat{C}'$ directly inherits an involution from the one of $\cat{C}$.
	This involution is immediately strictly monoidal, and it is a symmetric monoidal functor because the coherence morphisms of $F_{\det}$ are deterministic and in particular self-adjoint.
	Therefore, $\cat{C}'$ becomes a hom-involutive symmetric monoidal category.

	Let us name $F\colon \cat{C}' \to \cat{C}$ the obvious functor which extends $F_{\det}$, i.e.~$F(A)=F_{\det}(A)$ and $F$ is the identity on the hom-sets just defined.
	This is a strong monoidal functor with coherence isomorphisms $\varphi_{A,B}$ and $\varphi_I$.
	We now transport the ICD-structure along these coherences.
	Since $F$ is fully faithful, for every object $A$ of $\cat{C}'$ there are unique morphisms $\cop_A\colon A\to A\tensor A$ and $\discard_A\colon A\to I$ in $\cat{C}'$ determined by
	\[
		F(\cop_A) \coloneqq \varphi_{A,A}\comp \cop_{F A}
		\qquad\text{and}\qquad
		F(\discard_A) \coloneqq \varphi_I\comp \discard_{F A}.
	\]
	These are the only possible definitions making the diagrams~\eqref{eq:strong_icd_fun} commute for $F$.

	It remains to verify the ICD-axioms for these morphisms.
	Since $F$ is faithful, it is enough to check them after applying $F$.
	After applying $F$, the comonoid axioms and the compatibility with the tensor product are exactly the corresponding axioms for $(FA,\cop_{FA},\discard_{FA})$ in $\cat{C}$, with the coherence isomorphisms $\varphi$ inserted; the required equalities are precisely the standard coherence equations for a strong symmetric monoidal functor.
	The compatibility with the involution follows similarly from $\invo{\varphi_{A,B}}=\varphi_{A,B}$ and $\invo{\varphi_I}=\varphi_I$, together with the involutivity equations for $\cop_{FA}$ and $\discard_{FA}$ in $\cat{C}$ and the symmetry preservation axiom for $\varphi$.
	Thus $\cat{C}'$ is an ICD-category, and by construction $F$ is a strong ICD-functor.
	It is fully faithful, and every object of $\cat{C}$ is deterministically isomorphic to one in the image of $F$ because the same is true for $F_{\det}$.
	Hence $F$ is an ICD-equivalence by \cref{prop:icd_equiv}.
\end{proof}

\subsection{Classicality and compatibility}\label{sec:classical_compatible}

In our setting, an interesting question is which objects in an involutive Markov category can be considered ``classical'' or ``commutative'' and which ones display genuinely ``quantum'' or ``noncommutative'' behavior.
With this in mind, we introduce \emph{classical objects}, which, as the name suggests, are those that behave as in classical probability.
Whenever an object $A$ is not classical, it is still relevant to understand when a pair of morphisms out of $A$ displays ``classical'' behavior.  
This gives rise to the notion of \emph{compatibility}.

\begin{notation}\label{nota:invM}
	Throughout the rest of the paper, we restrict our attention to involutive Markov categories.
	This is motivated by the main categories of interest to us ($\calgmin$ and $\calgmax$), where morphisms are formal opposites of \emph{unital} maps, such as quantum channels.
	In classical Markov categories, this restriction corresponds to the normalization of probability.
	Note that the total morphisms in any ICD-category form an involutive Markov category.

	This restriction to involutive Markov categories is also motivated by the recent successful research in categorical probability, including work on quantum Markov categories~\cite{quantum-markov}, where the restriction to total morphisms facilitates arguments that would be impossible or more cumbersome for (quantum) CD-categories in general.
\end{notation}

Let us start with classical objects. In ordinary Markov categories, the copy morphisms are required to be invariant under swapping the outputs. This motivates the following.
\begin{lemma}
	\label{lem:cop_deterministic}
	The following are equivalent for an object $A$:
	\begin{enumerate}
		\item $\invo{\cop_A} = \cop_A$.
		\item $\swap_{A,A} \comp \cop_A = \cop_A$.
		\item $\cop_A$ is deterministic.
	\end{enumerate}
\end{lemma}

\begin{proof}
	The equivalence of the first two conditions is clear by the requirement $\invo{\cop_A} = \swap_{A,A} \comp \cop_A$.
	The equivalence of the second and third conditions follows by a straightforward calculation as in~\cite[Remark~10.2]{fritz2019synthetic} for ordinary Markov categories.
\end{proof}

\begin{definition}
	\label{def:classical}
	Let $\cat{C}$ be an involutive Markov category. An object $A$ in $\cat{C}$ is \newterm{classical} if it satisfies the equivalent conditions of \cref{lem:cop_deterministic}.
\end{definition}

\begin{definition}
	\label{def:classical_subcat}
	The \newterm{classical subcategory} $\classical{\cat{C}}$ of an involutive Markov category $\cat{C}$ is the subcategory whose objects are classical and morphisms are the self-adjoint morphisms.
\end{definition}

It is easy to see that $\classical{\cat{C}}$ is a Markov category, and in particular that if $A$ and $B$ are classical, then so is $A \tensor B$.

\begin{example}
	In the involutive Markov category given by the total morphisms of a cofree ICD-category $\cofree(\cat{C})$, the classical subcategory consists of the commutative comonoids with self-adjoint morphisms.
	Similarly in $\calgmin$ and $\calgmax$, the classical subcategory consists of the commutative pre-C*-algebras (with self-adjoint morphisms).
	These statements are immediate from \Cref{lem:cop_deterministic} since the copy morphisms are given by multiplication.
\end{example}

\begin{notation}\label{nota:product_power}
	For two morphisms $\phi,\psi$ with the same domain $A$, and $n$ any positive integer, we will use the shorthand notations
	\begin{equation*}
		\begin{tikzpicture}
	\begin{pgfonlayer}{nodelayer}
		\node [style=none] (16) at (-6.5, 0) {$\coloneqq$};
		\node [style=none] (27) at (-3.5, -2) {};
		\node [style=bn] (28) at (-3.5, -0.75) {};
		\node [style=none] (29) at (-2.5, 0.25) {};
		\node [style=none] (30) at (-4.5, 0.25) {};
		\node [style=none] (31) at (-4.5, 0.25) {};
		\node [style=none] (32) at (-2.5, 0.5) {};
		\node [style=morphism] (33) at (-4.5, 0.5) {$\phi$};
		\node [style=none] (34) at (-2.5, 1.5) {};
		\node [style=none] (35) at (-4.5, 1.5) {};
		\node [style=morphism] (36) at (-2.5, 0.5) {$\psi$};
		\node [style=none] (37) at (-8.5, 0) {$\pairing{\phi}{\psi}$};
		\node [style=none] (38) at (4.5, 0) {$\coloneqq$};
		\node [style=none] (39) at (8, -2) {};
		\node [style=bn] (40) at (8, -0.75) {};
		\node [style=none] (41) at (9, 0.25) {};
		\node [style=none] (42) at (7, 0.25) {};
		\node [style=none] (43) at (7, 0.25) {};
		\node [style=none] (44) at (9, 0.5) {};
		\node [style=morphism] (45) at (7, 0.5) {$\phi^{(n-1)}$};
		\node [style=none] (46) at (9, 1.5) {};
		\node [style=none] (47) at (7, 1.5) {};
		\node [style=morphism] (48) at (9, 0.5) {$\phi$};
		\node [style=none] (49) at (2.5, 0) {$\phi^{(n)}$};
		\node [style=none] (50) at (0, 0) {\text{and}};
	\end{pgfonlayer}
	\begin{pgfonlayer}{edgelayer}
		\draw [style=none] (27.center) to (28);
		\draw [style=none, bend right=45] (28) to (29.center);
		\draw [style=none, bend left=45] (28) to (30.center);
		\draw (34.center) to (29.center);
		\draw (35.center) to (33);
		\draw (31.center) to (33);
		\draw [style=none] (39.center) to (40);
		\draw [style=none, bend right=45] (40) to (41.center);
		\draw [style=none, bend left=45] (40) to (42.center);
		\draw (46.center) to (41.center);
		\draw (47.center) to (45);
		\draw (43.center) to (45);
	\end{pgfonlayer}
\end{tikzpicture}
	\end{equation*}
	for easier writing, where the right-hand equation defines $\phi^{(n)}$ recursively starting with $\phi^{(0)} \coloneqq \discard_A$.
\end{notation}

In $\calgmin$ and $\calgmax$, the left-hand equation amounts to the formation of the \emph{product map}, which is defined as the unique linear extension of
\[
	\opm{\pairing{\phi}{\psi}}{}(x \odot y) \coloneqq \opm{\phi}{}(x) \, \opm{\psi}{}(y).
\]

\begin{definition}\label{def:comp}
	Let $\phi,\psi$ be morphisms with the same domain. Then:
	\begin{enumerate}
		\item $\phi$ and $\psi$ are \newterm{compatible} if $\pairing{\phi}{\psi} = \swap \comp \pairing{\psi}{\phi}$.
		\item $\phi$ is \newterm{autocompatible} if it is compatible with itself.
	\end{enumerate} 
\end{definition}

\begin{remark}\label{rem:compatibility_prop}
	\begin{enumerate}
		\item\label{it:class_comp} If the domain of $\phi$ and $\psi$ is classical, then $\phi$ and $\psi$ are immediately compatible because the copy morphism is invariant under swapping:
			\begin{equation*}
				\begin{tikzpicture}
	\begin{pgfonlayer}{nodelayer}
		\node [style=none] (27) at (-3, -1.75) {};
		\node [style=bn] (28) at (-3, -0.5) {};
		\node [style=none] (29) at (-2, 0.5) {};
		\node [style=none] (30) at (-4, 0.5) {};
		\node [style=none] (31) at (-4, 0.5) {};
		\node [style=none] (32) at (-2, 0.75) {};
		\node [style=morphism] (33) at (-4, 0.75) {$\phi$};
		\node [style=none] (34) at (-2, 1.75) {};
		\node [style=none] (35) at (-4, 1.75) {};
		\node [style=morphism] (36) at (-2, 0.75) {$\psi$};
		\node [style=none] (37) at (0, 0) {$=$};
		\node [style=none] (38) at (3, -2.5) {};
		\node [style=bn] (39) at (3, -1.25) {};
		\node [style=none] (40) at (4, -0.5) {};
		\node [style=none] (41) at (2, -0.5) {};
		\node [style=none] (42) at (2, -0.5) {};
		\node [style=none] (43) at (4, 1.5) {};
		\node [style=morphism] (44) at (4, 1.5) {$\psi$};
		\node [style=morphism] (45) at (2, 1.5) {$\phi$};
		\node [style=none] (46) at (4, 2.5) {};
		\node [style=none] (47) at (2, 1.5) {};
		\node [style=none] (48) at (2, 2.5) {};
		\node [style=none] (49) at (9, -2.25) {};
		\node [style=bn] (50) at (9, -1) {};
		\node [style=none] (51) at (10, 0) {};
		\node [style=none] (52) at (8, 0) {};
		\node [style=none] (53) at (8, 0.75) {};
		\node [style=none] (55) at (10, 2.25) {};
		\node [style=morphism] (58) at (8, 0.25) {$\psi$};
		\node [style=morphism] (60) at (10, 0.25) {$\phi$};
		\node [style=none] (61) at (6, 0) {$=$};
		\node [style=none] (62) at (10, 0.75) {};
		\node [style=none] (67) at (8, 2.25) {};
	\end{pgfonlayer}
	\begin{pgfonlayer}{edgelayer}
		\draw [style=none] (27.center) to (28);
		\draw [style=none, bend right=45] (28) to (29.center);
		\draw [style=none, bend left=45] (28) to (30.center);
		\draw (34.center) to (29.center);
		\draw (35.center) to (33);
		\draw (31.center) to (33);
		\draw [style=none] (38.center) to (39);
		\draw [style=none, in=-90, out=0] (39) to (40.center);
		\draw [style=none, in=-90, out=-180] (39) to (41.center);
		\draw [in=90, out=-90, looseness=0.75] (45) to (40.center);
		\draw (46.center) to (44);
		\draw [in=-90, out=90, looseness=0.75] (42.center) to (44);
		\draw (48.center) to (47.center);
		\draw [style=none] (49.center) to (50);
		\draw [style=none, bend right=45] (50) to (51.center);
		\draw [style=none, bend left=45] (50) to (52.center);
		\draw [style=protected, in=-90, out=90, looseness=0.75] (53.center) to (55.center);
		\draw [in=90, out=-90, looseness=0.75] (67.center) to (62.center);
		\draw [style=protected] (53.center) to (52.center);
		\draw [style=protected] (62.center) to (51.center);
	\end{pgfonlayer}
\end{tikzpicture}
			\end{equation*}
		\item Compatibility of morphisms is clearly a symmetric relation. Indeed, if we assume $\pairing{\phi}{\psi} = \swap\comp\pairing{\psi}{\phi}$, then we have
			\[
				\pairing{\psi}{\phi}= \swap \comp \swap \comp\pairing{\psi}{\phi} = \swap \comp \pairing{\phi}{\psi},
			\]
			and hence also $\psi$ and $\phi$ are compatible.
		\item \label{it:self_comp} Whenever $\phi$ and $\psi$ are self-adjoint, then compatibility is equivalent to $\pairing{\phi}{\psi}$ being self-adjoint.
	\end{enumerate}
\end{remark}

\begin{example}\label{ex:commutingranges}
	In $\calgmin$ and $\calgmax$, two morphisms $\phi \colon A \to B$ and $\psi : A \to C$ are compatible if and only if
	\[
		\opm{\phi}{} (x) \, \opm{\psi}{}(y) = \opm{\psi}{}(y) \,\opm{\phi}{}(x)
	\]
	for all $x \in \pre{B}$ and $y \in \pre{C}$, which means that $\opm{\phi}{}$ and $\opm{\psi}{}$ have commuting ranges.
	In particular, $\phi$ is autocompatible if and only if the range of $\opm{\phi}{}$ is commutative.
	
	We can now explain the term ``compatible'' by showing how the compatibility of quantum observables is an instance of this notion.
	If we take $B = C = \mathbb{C}^{\{0,1\}}$, then a self-adjoint morphism $\phi: A \to \mathbb{C}^{\{0,1\}}$ corresponds to a self-adjoint element $x \in \pre{A}$ by \Cref{ex:observables}, and similarly $\psi : A \to \mathbb{C}^{\{0,1\}}$ corresponds to a self-adjoint element $y \in \pre{A}$.
	Then $\phi$ and $\psi$ are compatible if and only if the elements $x$ and $y$ commute, and this is the usual notion of compatibility of observables in quantum theory.
\end{example}

Another important aspect of compatibility is that it allows us to state a version of the universal property of the maximal tensor product of C*-algebras (\cite[Exercise 3.5.1]{brownozawa}) in \emph{all} involutive Markov categories.

\begin{proposition}
	In any involutive Markov category, there is a bijective correspondence between the following two sets:
	\begin{equation}
		\label{eq:comp_det_bij}
		\Set{\hspace{-1ex}\begin{array}{c}
			\text{Compatible pairs of deterministic morphisms}\\
			A \to B\text{ and }A \to C
		\end{array}\hspace{-1ex}} 
		\quad \stackrel{\cong}{\longleftrightarrow} \quad
		\Set{\hspace{-1ex}\begin{array}{c}
			\text{Deterministic morphisms}\\
			A \to B \tensor C
		\end{array}\hspace{-1ex}}
	\end{equation}
	where a pair $(\phi\colon A \to B, \, \psi \colon A \to C)$ is sent to $\pairing{\phi}{\psi}$.
\end{proposition}

This result is reminiscent of the universal property of the maximal tensor product of C*-algebras (cf.~\cref{prop:display_compatibility}; see \cite[Proposition 3.3.7]{brownozawa} for an explicit statement restricted to $*$-ho\-mo\-mor\-phisms).
The crucial distinction is that, in the statement above, morphisms are not required to be bounded (see \cref{ex:homnotbound}), whereas in the universal property of the maximal tensor product all maps are completely positive, and hence bounded (\cref{prop:pu_bounded}).

\begin{proof}
	We first prove that the map is well-defined by showing that such $\pairing{\phi}{\psi}$ is indeed deterministic.
	Given a pair of compatible deterministic morphisms $\phi \colon A \to B$ and $\psi\colon A \to C$, by definition of compatibility $\pairing{\phi}{\psi}$ is immediately self-adjoint. 
	It is also deterministic, since the copy morphisms respect the symmetric monoidal structure~\eqref{eq:icd_def_monoidal}.

	Injectivity of the map~\eqref{eq:comp_det_bij} now immediately follows from marginalization, since $\phi$ and $\psi$ can be reconstructed from $\pairing{\phi}{\psi}$ by composing with deletion on either output.
	For surjectivity, consider a deterministic morphism $\pi \colon A \to B \tensor C$, and take its marginals
\begin{equation*}
	\begin{tikzpicture}
	\begin{pgfonlayer}{nodelayer}
		\node [style=morphism] (6) at (-8, 0) {$\phi$};
		\node [style=morphism] (74) at (-3.5, 0) {$\,\,\pi\,\,$};
		\node [style=none] (80) at (-4, 0) {};
		\node [style=none] (83) at (-3, 0) {};
		\node [style=none] (84) at (-4, 2) {};
		\node [style=bn] (85) at (-3, 1.5) {};
		\node [style=none] (88) at (-3.5, -2) {};
		\node [style=none] (89) at (-8, 2) {};
		\node [style=none] (90) at (-8, -2) {};
		\node [style=none] (91) at (-6, 0) {$\coloneqq$};
		\node [style=none] (92) at (-8, -2.5) {$A$};
		\node [style=none] (93) at (-8, 2.5) {$B$};
		\node [style=none] (94) at (-4, 2.5) {$B$};
		\node [style=none] (95) at (-3.5, -2.5) {$A$};
		\node [style=morphism] (96) at (3, 0) {$\psi$};
		\node [style=morphism] (97) at (7.5, 0) {$\,\,\pi\,\,$};
		\node [style=none] (98) at (7, 0) {};
		\node [style=none] (99) at (8, 0) {};
		\node [style=bn] (100) at (7, 1.5) {};
		\node [style=none] (101) at (8, 2) {};
		\node [style=none] (102) at (7.5, -2) {};
		\node [style=none] (103) at (3, 2) {};
		\node [style=none] (104) at (3, -2) {};
		\node [style=none] (105) at (5, 0) {$\coloneqq$};
		\node [style=none] (106) at (3, -2.5) {$A$};
		\node [style=none] (107) at (3, 2.5) {$C$};
		\node [style=none] (108) at (8, 2.5) {$C$};
		\node [style=none] (109) at (7.5, -2.5) {$A$};
		\node [style=none] (110) at (-0.25, 0) {$\text{and}$};
	\end{pgfonlayer}
	\begin{pgfonlayer}{edgelayer}
		\draw (84.center) to (80.center);
		\draw (85) to (83.center);
		\draw (74) to (88.center);
		\draw (89.center) to (90.center);
		\draw (100) to (98.center);
		\draw (101.center) to (99.center);
		\draw (97) to (102.center);
		\draw (103.center) to (104.center);
	\end{pgfonlayer}
\end{tikzpicture}
\end{equation*}
	As composites of deterministic morphisms, $\phi$ and $\psi$ are deterministic as well.
	Further we have
	\begin{equation*}
		\begin{tikzpicture}
	\begin{pgfonlayer}{nodelayer}
		\node [style=none] (0) at (8, -2) {};
		\node [style=bn] (1) at (8, -0.75) {};
		\node [style=none] (2) at (9, 0.25) {};
		\node [style=none] (3) at (7, 0.25) {};
		\node [style=none] (4) at (7, 0.25) {};
		\node [style=none] (5) at (9, 0.5) {};
		\node [style=morphism] (6) at (7, 0.5) {$\phi$};
		\node [style=none] (7) at (9, 2) {};
		\node [style=none] (8) at (7, 2) {};
		\node [style=morphism] (9) at (9, 0.5) {$\psi$};
		\node [style=none] (10) at (5, 0) {$=$};
		\node [style=none] (11) at (1.5, -2) {};
		\node [style=bn] (12) at (1.5, -0.75) {};
		\node [style=none] (13) at (2.5, 0.25) {};
		\node [style=none] (14) at (0.5, 0.25) {};
		\node [style=none] (15) at (0.5, 0.25) {};
		\node [style=none] (16) at (2.5, 0.5) {};
		\node [style=morphism] (17) at (0.5, 0.5) {$\,\,\pi\,\,$};
		\node [style=none] (18) at (2.5, 0.5) {};
		\node [style=morphism] (20) at (2.5, 0.5) {$\,\,\pi\,\,$};
		\node [style=none] (21) at (0, 0.5) {};
		\node [style=none] (22) at (1, 0.5) {};
		\node [style=none] (23) at (2, 0.5) {};
		\node [style=none] (24) at (3, 0.5) {};
		\node [style=none] (25) at (0, 2) {};
		\node [style=none] (26) at (3, 2) {};
		\node [style=bn] (27) at (2, 1.5) {};
		\node [style=bn] (28) at (1, 1.5) {};
		\node [style=none] (29) at (-6, -1) {};
		\node [style=bn] (30) at (-6, 0) {};
		\node [style=none] (31) at (-5, 1.25) {};
		\node [style=none] (32) at (-7, 1.25) {};
		\node [style=none] (33) at (-7, 1.25) {};
		\node [style=none] (47) at (-2, 0) {=};
		\node [style=morphism] (48) at (-5.5, -1) {$\,\,\pi\,\,$};
		\node [style=none] (49) at (-5, -1) {};
		\node [style=bn] (50) at (-5, 0) {};
		\node [style=none] (51) at (-4, 1.25) {};
		\node [style=none] (52) at (-6, 1.25) {};
		\node [style=none] (53) at (-6, 1.25) {};
		\node [style=none] (59) at (-7, 1.25) {};
		\node [style=none] (60) at (-6, 1.25) {};
		\node [style=none] (61) at (-5, 1.25) {};
		\node [style=none] (62) at (-4, 1.25) {};
		\node [style=none] (63) at (-7, 2) {};
		\node [style=none] (64) at (-4, 2) {};
		\node [style=bn] (65) at (-5, 1.5) {};
		\node [style=bn] (66) at (-6, 1.5) {};
		\node [style=none] (67) at (-5.5, -2) {};
		\node [style=none] (73) at (-9, 0) {=};
		\node [style=morphism] (74) at (-11.5, 0) {$\,\,\pi\,\,$};
		\node [style=none] (80) at (-12, 0) {};
		\node [style=none] (83) at (-11, 0) {};
		\node [style=none] (84) at (-12, 2) {};
		\node [style=none] (85) at (-11, 2) {};
		\node [style=none] (88) at (-11.5, -2) {};
	\end{pgfonlayer}
	\begin{pgfonlayer}{edgelayer}
		\draw [style=none] (0.center) to (1);
		\draw [style=none, bend right=45] (1) to (2.center);
		\draw [style=none, bend left=45] (1) to (3.center);
		\draw (7.center) to (2.center);
		\draw (8.center) to (6);
		\draw (4.center) to (6);
		\draw [style=none] (11.center) to (12);
		\draw [style=none, bend right=45] (12) to (13.center);
		\draw [style=none, bend left=45] (12) to (14.center);
		\draw (18.center) to (13.center);
		\draw (15.center) to (17);
		\draw (25.center) to (21.center);
		\draw (28) to (22.center);
		\draw (27) to (23.center);
		\draw (26.center) to (24.center);
		\draw [style=none] (29.center) to (30);
		\draw [style=none, in=270, out=30] (30) to (31.center);
		\draw [style=none, bend left=45] (30) to (32.center);
		\draw [style=none] (49.center) to (50);
		\draw [style=none, bend right=45] (50) to (51.center);
		\draw [style=protected, in=270, out=150] (50) to (52.center);
		\draw (63.center) to (59.center);
		\draw (66) to (60.center);
		\draw (65) to (61.center);
		\draw (64.center) to (62.center);
		\draw (48) to (67.center);
		\draw (84.center) to (80.center);
		\draw (85.center) to (83.center);
		\draw (74) to (88.center);
	\end{pgfonlayer}
\end{tikzpicture}
	\end{equation*}
	by determinism of $\pi$. 
	Finally since $\pi$ is self-adjoint, $\phi$ and $\psi$ are also compatible by \cref{rem:compatibility_prop}\ref{it:self_comp}.
\end{proof}

\begin{proposition}\label{prop:comp_product}
	Let $\phi,\psi$ and $\omega$ be pairwise compatible morphisms. Then $\phi$ is also compatible with $\pairing{\psi}{\omega}$.
\end{proposition}
\begin{proof}
	This follows by associativity of the copy morphisms via the following direct calculation: 
	
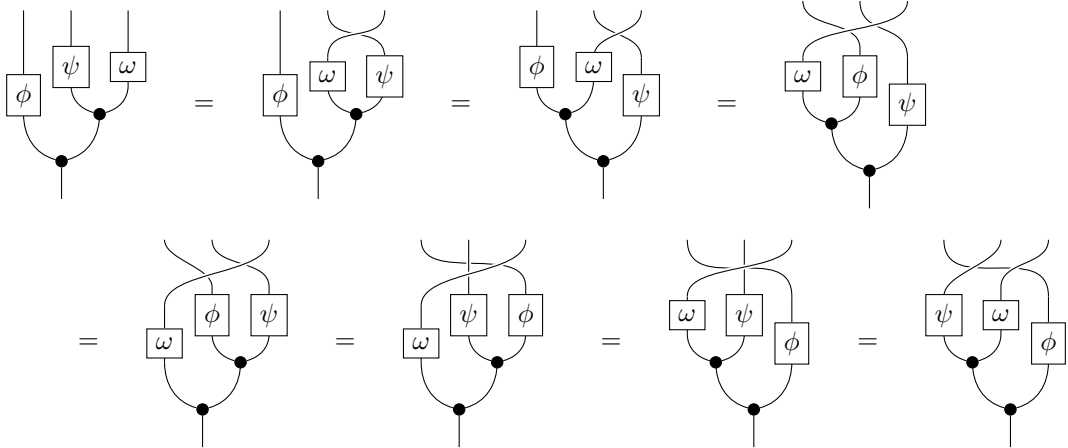
\begin{equation*}
		\begin{tikzpicture}
	\begin{pgfonlayer}{nodelayer}
		\node [style=none] (17) at (0, 0) {$=$};
		\node [style=none] (35) at (-4.75, -0.25) {};
		\node [style=none] (36) at (-2.75, -0.25) {};
		\node [style=none] (37) at (-4.75, 2.5) {};
		\node [style=bn] (39) at (-3.75, -1.5) {};
		\node [style=none] (43) at (-4.75, -0.25) {};
		\node [style=none] (51) at (-3.75, -2.5) {};
		\node [style=none] (133) at (-2.75, -0.25) {};
		\node [style=none] (134) at (-3.5, 0.5) {};
		\node [style=none] (135) at (-2, 0.5) {};
		\node [style=bn] (136) at (-2.75, -0.25) {};
		\node [style=none] (137) at (-3.5, 2.5) {};
		\node [style=none] (138) at (-2, 2.5) {};
		\node [style=morphism] (154) at (-4.75, 0.25) {$\phi$};
		\node [style=morphism] (155) at (-3.5, 1) {$\psi$};
		\node [style=none] (156) at (2, -0.25) {};
		\node [style=none] (157) at (4, -0.25) {};
		\node [style=none] (158) at (2, 2.5) {};
		\node [style=bn] (159) at (3, -1.5) {};
		\node [style=none] (160) at (2, -0.25) {};
		\node [style=none] (161) at (3, -2.5) {};
		\node [style=none] (162) at (4, -0.25) {};
		\node [style=none] (163) at (3.25, 0.5) {};
		\node [style=none] (164) at (4.75, 0.5) {};
		\node [style=bn] (165) at (4, -0.25) {};
		\node [style=none] (166) at (3.25, 1.25) {};
		\node [style=none] (167) at (4.75, 1.25) {};
		\node [style=morphism] (168) at (2, 0.25) {$\phi$};
		\node [style=morphism] (169) at (4.75, 0.75) {$\psi$};
		\node [style=none] (170) at (3.25, 2.5) {};
		\node [style=none] (171) at (4.75, 2.5) {};
		\node [style=none] (185) at (6.75, 0) {$=$};
		\node [style=morphism] (186) at (8.75, 1) {$\phi$};
		\node [style=morphism] (187) at (11.5, 0.25) {$\psi$};
		\node [style=none] (188) at (9.5, -0.25) {};
		\node [style=none] (189) at (11.5, -0.25) {};
		\node [style=none] (190) at (11.5, 1.25) {};
		\node [style=bn] (191) at (10.5, -1.5) {};
		\node [style=none] (192) at (10.5, -2.5) {};
		\node [style=none] (193) at (11.5, -0.25) {};
		\node [style=none] (194) at (8.75, 0.5) {};
		\node [style=none] (195) at (10.25, 0.75) {};
		\node [style=bn] (196) at (9.5, -0.25) {};
		\node [style=none] (197) at (8.75, 2.5) {};
		\node [style=none] (198) at (10.25, 1.25) {};
		\node [style=none] (199) at (10.25, 2.5) {};
		\node [style=none] (200) at (11.5, 2.5) {};
		\node [style=morphism] (201) at (10.25, 1) {$\omega$};
		\node [style=morphism] (202) at (3.25, 0.75) {$\omega$};
		\node [style=morphism] (203) at (-2, 1) {$\omega$};
		\node [style=morphism] (204) at (17.25, 0.75) {$\phi$};
		\node [style=morphism] (205) at (18.5, 0) {$\psi$};
		\node [style=none] (206) at (16.5, -0.5) {};
		\node [style=none] (207) at (18.5, -0.5) {};
		\node [style=none] (208) at (18.5, 1.25) {};
		\node [style=bn] (209) at (17.5, -1.75) {};
		\node [style=none] (210) at (17.5, -2.75) {};
		\node [style=none] (211) at (18.5, -0.5) {};
		\node [style=none] (212) at (15.75, 0.25) {};
		\node [style=none] (213) at (17.25, 0.25) {};
		\node [style=bn] (214) at (16.5, -0.5) {};
		\node [style=none] (215) at (15.75, 1.25) {};
		\node [style=none] (216) at (18.5, 2.75) {};
		\node [style=none] (217) at (17.25, 2.75) {};
		\node [style=none] (218) at (17.25, 1.5) {};
		\node [style=none] (219) at (15.75, 2.75) {};
		\node [style=none] (220) at (13.75, 0) {$=$};
		\node [style=morphism] (221) at (15.75, 0.75) {$\omega$};
		\node [style=none] (222) at (21.75, 0) {};
	\end{pgfonlayer}
	\begin{pgfonlayer}{edgelayer}
		\draw [in=-90, out=165] (39) to (35.center);
		\draw [in=-90, out=15] (39) to (36.center);
		\draw (43.center) to (37.center);
		\draw (51.center) to (39);
		\draw [in=-90, out=165] (136) to (134.center);
		\draw [in=-90, out=15] (136) to (135.center);
		\draw (137.center) to (134.center);
		\draw (138.center) to (135.center);
		\draw [in=-90, out=165] (159) to (156.center);
		\draw [in=-90, out=15] (159) to (157.center);
		\draw (160.center) to (158.center);
		\draw (161.center) to (159);
		\draw [in=-90, out=165] (165) to (163.center);
		\draw [in=-90, out=15] (165) to (164.center);
		\draw (166.center) to (163.center);
		\draw (167.center) to (164.center);
		\draw [style=protected, in=90, out=-90, looseness=1.25] (170.center) to (167.center);
		\draw [style=protected, in=90, out=-90, looseness=1.25] (171.center) to (166.center);
		\draw [in=-90, out=165] (191) to (188.center);
		\draw [in=-90, out=15] (191) to (189.center);
		\draw (192.center) to (191);
		\draw [in=-90, out=165] (196) to (194.center);
		\draw [in=-90, out=15] (196) to (195.center);
		\draw (193.center) to (190.center);
		\draw (197.center) to (194.center);
		\draw (198.center) to (195.center);
		\draw [style=protected, in=90, out=-90] (199.center) to (190.center);
		\draw [style=protected, in=90, out=-90] (200.center) to (198.center);
		\draw [in=-90, out=165] (209) to (206.center);
		\draw [in=-90, out=15] (209) to (207.center);
		\draw (210.center) to (209);
		\draw [in=-90, out=165] (214) to (212.center);
		\draw [in=-90, out=15] (214) to (213.center);
		\draw (211.center) to (208.center);
		\draw (215.center) to (212.center);
		\draw [in=90, out=-90, looseness=1.25] (217.center) to (208.center);
		\draw [in=90, out=-90] (219.center) to (218.center);
		\draw [style=protected, in=90, out=-90, looseness=0.75] (216.center) to (215.center);
		\draw (218.center) to (213.center);
	\end{pgfonlayer}
\end{tikzpicture}
	\end{equation*}
	\begin{equation*}
		\begin{tikzpicture}
	\begin{pgfonlayer}{nodelayer}
		\node [style=none] (219) at (-6.75, 0) {$=$};
		\node [style=none] (221) at (-4.75, -0.5) {};
		\node [style=none] (222) at (-2.75, -0.5) {};
		\node [style=none] (223) at (-4.75, 1) {};
		\node [style=bn] (224) at (-3.75, -1.75) {};
		\node [style=none] (225) at (-4.75, -0.5) {};
		\node [style=none] (226) at (-3.75, -2.75) {};
		\node [style=none] (227) at (-2.75, -0.5) {};
		\node [style=none] (228) at (-3.5, 0.25) {};
		\node [style=none] (229) at (-2, 0.25) {};
		\node [style=bn] (230) at (-2.75, -0.5) {};
		\node [style=none] (231) at (-3.5, 1.5) {};
		\node [style=none] (232) at (-2, 1.5) {};
		\node [style=morphism] (233) at (-3.5, 0.75) {$\phi$};
		\node [style=morphism] (234) at (-2, 0.75) {$\psi$};
		\node [style=none] (235) at (-3.5, 2.75) {};
		\node [style=none] (236) at (-4.75, 2.75) {};
		\node [style=none] (237) at (-2, 2.75) {};
		\node [style=morphism] (240) at (-4.75, 0) {$\omega$};
		\node [style=none] (241) at (0, 0) {$=$};
		\node [style=none] (242) at (2, -0.5) {};
		\node [style=none] (243) at (4, -0.5) {};
		\node [style=none] (244) at (2, 1) {};
		\node [style=bn] (245) at (3, -1.75) {};
		\node [style=none] (246) at (2, -0.5) {};
		\node [style=none] (247) at (3, -2.75) {};
		\node [style=none] (248) at (4, -0.5) {};
		\node [style=none] (249) at (3.25, 0.25) {};
		\node [style=none] (250) at (4.75, 0.25) {};
		\node [style=bn] (251) at (4, -0.5) {};
		\node [style=none] (252) at (3.25, 1.5) {};
		\node [style=none] (253) at (4.75, 1.5) {};
		\node [style=morphism] (254) at (4.75, 0.75) {$\phi$};
		\node [style=morphism] (255) at (3.25, 0.75) {$\psi$};
		\node [style=none] (256) at (2, 2.75) {};
		\node [style=none] (257) at (3.25, 2.75) {};
		\node [style=none] (258) at (4.75, 2.75) {};
		\node [style=morphism] (259) at (2, 0) {$\omega$};
		\node [style=none] (260) at (-9.75, 0) {};
		\node [style=morphism] (261) at (11.75, 0) {$\phi$};
		\node [style=morphism] (262) at (10.5, 0.75) {$\psi$};
		\node [style=none] (263) at (9.75, -0.5) {};
		\node [style=none] (264) at (11.75, -0.5) {};
		\node [style=none] (265) at (11.75, 1.25) {};
		\node [style=bn] (266) at (10.75, -1.75) {};
		\node [style=none] (267) at (10.75, -2.75) {};
		\node [style=none] (268) at (11.75, -0.5) {};
		\node [style=none] (269) at (9, 0.25) {};
		\node [style=none] (270) at (10.5, 0.25) {};
		\node [style=bn] (271) at (9.75, -0.5) {};
		\node [style=none] (272) at (9, 1.25) {};
		\node [style=none] (273) at (11.75, 2.75) {};
		\node [style=none] (274) at (9, 2.75) {};
		\node [style=none] (275) at (10.5, 1.5) {};
		\node [style=none] (276) at (10.5, 2.75) {};
		\node [style=none] (277) at (7, 0) {$=$};
		\node [style=morphism] (278) at (9, 0.75) {$\omega$};
		\node [style=morphism] (279) at (18.5, 0) {$\phi$};
		\node [style=morphism] (280) at (15.75, 0.75) {$\psi$};
		\node [style=none] (281) at (16.5, -0.5) {};
		\node [style=none] (282) at (18.5, -0.5) {};
		\node [style=none] (283) at (18.5, 1.25) {};
		\node [style=bn] (284) at (17.5, -1.75) {};
		\node [style=none] (285) at (17.5, -2.75) {};
		\node [style=none] (286) at (18.5, -0.5) {};
		\node [style=none] (287) at (15.75, 0.25) {};
		\node [style=none] (288) at (17.25, 0.25) {};
		\node [style=bn] (289) at (16.5, -0.5) {};
		\node [style=none] (290) at (15.75, 1.25) {};
		\node [style=none] (291) at (17.25, 2.75) {};
		\node [style=none] (292) at (15.75, 2.75) {};
		\node [style=none] (293) at (17.25, 1.5) {};
		\node [style=none] (294) at (18.5, 2.75) {};
		\node [style=none] (295) at (13.75, 0) {$=$};
		\node [style=morphism] (296) at (17.25, 0.75) {$\omega$};
	\end{pgfonlayer}
	\begin{pgfonlayer}{edgelayer}
		\draw [in=-90, out=165] (224) to (221.center);
		\draw [in=-90, out=15] (224) to (222.center);
		\draw (225.center) to (223.center);
		\draw (226.center) to (224);
		\draw [in=-90, out=165] (230) to (228.center);
		\draw [in=-90, out=15] (230) to (229.center);
		\draw (231.center) to (228.center);
		\draw (232.center) to (229.center);
		\draw [in=90, out=-90] (235.center) to (232.center);
		\draw [in=90, out=-90, looseness=0.75] (236.center) to (231.center);
		\draw [style=protected, in=-90, out=90, looseness=0.75] (223.center) to (237.center);
		\draw [in=-90, out=165] (245) to (242.center);
		\draw [in=-90, out=15] (245) to (243.center);
		\draw (246.center) to (244.center);
		\draw (247.center) to (245);
		\draw [in=-90, out=165] (251) to (249.center);
		\draw [in=-90, out=15] (251) to (250.center);
		\draw (252.center) to (249.center);
		\draw (253.center) to (250.center);
		\draw [in=90, out=-90] (256.center) to (253.center);
		\draw [style=protected, in=90, out=-90, looseness=0.75] (257.center) to (252.center);
		\draw [style=protected, in=-90, out=90, looseness=0.75] (244.center) to (258.center);
		\draw [in=-90, out=165] (266) to (263.center);
		\draw [in=-90, out=15] (266) to (264.center);
		\draw (267.center) to (266);
		\draw [in=-90, out=165] (271) to (269.center);
		\draw [in=-90, out=15] (271) to (270.center);
		\draw (268.center) to (265.center);
		\draw (272.center) to (269.center);
		\draw [in=90, out=-90, looseness=1.25] (274.center) to (265.center);
		\draw [in=90, out=-90] (276.center) to (275.center);
		\draw [style=protected, in=90, out=-90, looseness=0.75] (273.center) to (272.center);
		\draw (275.center) to (270.center);
		\draw [in=-90, out=165] (284) to (281.center);
		\draw [in=-90, out=15] (284) to (282.center);
		\draw (285.center) to (284);
		\draw [in=-90, out=165] (289) to (287.center);
		\draw [in=-90, out=15] (289) to (288.center);
		\draw (286.center) to (283.center);
		\draw (290.center) to (287.center);
		\draw [style=protected, in=90, out=-90, looseness=1.25] (292.center) to (283.center);
		\draw [style=protected, in=90, out=-90] (294.center) to (293.center);
		\draw [style=protected, in=90, out=-90, looseness=0.75] (291.center) to (290.center);
		\draw (293.center) to (288.center);
	\end{pgfonlayer}
\end{tikzpicture}
	\end{equation*}
\end{proof}

\begin{corollary}\label{cor:autocomp_powers}
	Let $\phi$ be an autocompatible morphism. Then $\phi^{(n)}$ is compatible with $\phi^{(m)}$ for any $n,m \in \N$. 
\end{corollary}
\begin{proof}
	This follows by a double induction on $n$ and $m$ using \cref{prop:comp_product}.
\end{proof}

Autocompatibility has a strong relation with classicality, as the following results suggest.

\begin{lemma}\label{lem:noninv_mono}
	Let $\phi\colon A \to B$ be an autocompatible morphism such that $\phi \tensor \phi$ is monic.
	Then $A$ is classical.
\end{lemma}

\begin{proof}
	This is a direct check:
	
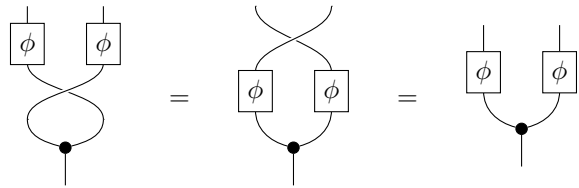
\begin{equation*}
		\begin{tikzpicture}
	\begin{pgfonlayer}{nodelayer}
		\node [style=none] (79) at (-8, 1.75) {};
		\node [style=none] (80) at (-10, 1.75) {};
		\node [style=none] (81) at (-10, 1.75) {};
		\node [style=none] (82) at (-8, 3.25) {};
		\node [style=none] (83) at (-10, 3.25) {};
		\node [style=none] (84) at (-9, 0) {};
		\node [style=bn] (85) at (-9, 1) {};
		\node [style=none] (86) at (-8, 1.75) {};
		\node [style=none] (87) at (-10, 1.75) {};
		\node [style=none] (88) at (-10, 1.75) {};
		\node [style=none] (89) at (-8, 1.75) {};
		\node [style=none] (90) at (-10, 1.75) {};
		\node [style=none] (91) at (-8, 1.75) {};
		\node [style=none] (92) at (-10, 1.75) {};
		\node [style=none] (94) at (-10, 4.75) {};
		\node [style=none] (95) at (-8, 4.75) {};
		\node [style=none] (96) at (-6, 2.25) {$=$};
		\node [style=none] (113) at (-2, 3) {};
		\node [style=none] (114) at (-4, 2.75) {};
		\node [style=none] (115) at (-4, 3) {};
		\node [style=none] (116) at (-2, 4.75) {};
		\node [style=none] (117) at (-4, 4.75) {};
		\node [style=none] (118) at (-3, 0) {};
		\node [style=bn] (119) at (-3, 1) {};
		\node [style=none] (120) at (-2, 2) {};
		\node [style=none] (121) at (-4, 2) {};
		\node [style=none] (122) at (-4, 2) {};
		\node [style=none] (123) at (-2, 2.25) {};
		\node [style=none] (124) at (-4, 2.25) {};
		\node [style=none] (125) at (-2, 3) {};
		\node [style=none] (126) at (-4, 3) {};
		\node [style=morphism] (128) at (-2, 2.5) {$\phi$};
		\node [style=none] (129) at (-4, 4.75) {};
		\node [style=none] (130) at (-2, 4.75) {};
		\node [style=morphism] (131) at (-4, 2.5) {$\phi$};
		\node [style=morphism] (148) at (-8, 3.75) {$\phi$};
		\node [style=morphism] (152) at (-10, 3.75) {$\phi$};
		\node [style=none] (154) at (0, 2.25) {$=$};
		\node [style=morphism] (169) at (4, 3) {$\phi$};
		\node [style=none] (175) at (3, 0.5) {};
		\node [style=bn] (176) at (3, 1.5) {};
		\node [style=none] (177) at (4, 2.5) {};
		\node [style=none] (178) at (2, 2.5) {};
		\node [style=none] (179) at (2, 2.5) {};
		\node [style=morphism] (181) at (2, 3) {$\phi$};
		\node [style=none] (182) at (4, 4.25) {};
		\node [style=none] (183) at (2, 4.25) {};
	\end{pgfonlayer}
	\begin{pgfonlayer}{edgelayer}
		\draw [in=90, out=-90, looseness=0.75] (83.center) to (79.center);
		\draw [style=protected, in=-90, out=90, looseness=0.75] (81.center) to (82.center);
		\draw [style=none] (84.center) to (85);
		\draw [style=none, in=-90, out=15] (85) to (86.center);
		\draw [style=none, in=270, out=165] (85) to (87.center);
		\draw (91.center) to (86.center);
		\draw (92.center) to (90.center);
		\draw (88.center) to (90.center);
		\draw (94.center) to (83.center);
		\draw (95.center) to (82.center);
		\draw [in=90, out=-90, looseness=0.75] (117.center) to (113.center);
		\draw [style=protected, in=-90, out=90, looseness=0.75] (115.center) to (116.center);
		\draw [style=none] (118.center) to (119);
		\draw [style=none, in=-90, out=15] (119) to (120.center);
		\draw [style=none, in=270, out=165] (119) to (121.center);
		\draw (125.center) to (120.center);
		\draw (126.center) to (124.center);
		\draw (122.center) to (124.center);
		\draw (129.center) to (117.center);
		\draw (130.center) to (116.center);
		\draw [style=none] (175.center) to (176);
		\draw [style=none, in=-90, out=15] (176) to (177.center);
		\draw [style=none, in=270, out=165] (176) to (178.center);
		\draw (182.center) to (177.center);
		\draw (183.center) to (181);
		\draw (179.center) to (181);
	\end{pgfonlayer}
\end{tikzpicture}
	\end{equation*}
	The monicness assumption can now be applied to obtain the desired conclusion.
\end{proof}

\begin{remark}
	\label{rem:nondet_iso}
	The mere existence of an autocompatible isomorphism $A \to B$ does not suffice to show that $B$ is classical. 
	For instance in $\calgmin$ or $\calgmax$, consider $B=M_n$, the matrix algebra of $n\times n$ matrices, and $A=\mathbb{C}^{n^2}$. 
	Then there exist linear unital isomorphisms $\opm{\phi}{} \colon B \rightsquigarrow A$, and any such $\phi$ is autocompatible since it has commutative range.
\end{remark}

\begin{lemma}\label{lem:noninv_det_epi}
	Let $\phi\colon A \to B$ be a deterministic autocompatible epimorphism. Then $B$ is classical.
\end{lemma}
\begin{proof}
	By autocompatibility, $\pairing{\phi}{\phi}$ is invariant under swap. Since $\phi$ is deterministic, we infer that
	
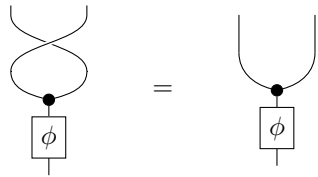
\begin{equation*}
		\begin{tikzpicture}
	\begin{pgfonlayer}{nodelayer}
		\node [style=none] (79) at (-2, 0.5) {};
		\node [style=none] (80) at (-4, 0.5) {};
		\node [style=none] (81) at (-4, 0.5) {};
		\node [style=none] (82) at (-2, 2) {};
		\node [style=none] (83) at (-4, 2) {};
		\node [style=none] (84) at (-3, -2.25) {};
		\node [style=bn] (85) at (-3, -0.25) {};
		\node [style=none] (86) at (-2, 0.5) {};
		\node [style=none] (87) at (-4, 0.5) {};
		\node [style=none] (88) at (-4, 0.5) {};
		\node [style=none] (89) at (-2, 0.5) {};
		\node [style=none] (90) at (-4, 0.5) {};
		\node [style=none] (91) at (-2, 0.5) {};
		\node [style=none] (92) at (-4, 0.5) {};
		\node [style=none] (94) at (-4, 2.25) {};
		\node [style=none] (95) at (-2, 2.25) {};
		\node [style=morphism] (148) at (3, -1) {$\phi$};
		\node [style=morphism] (152) at (-3, -1.25) {$\phi$};
		\node [style=none] (154) at (0, 0) {$=$};
		\node [style=none] (175) at (3, -2) {};
		\node [style=bn] (176) at (3, 0) {};
		\node [style=none] (177) at (4, 1) {};
		\node [style=none] (178) at (2, 1) {};
		\node [style=none] (179) at (2, 1) {};
		\node [style=none] (182) at (4, 2) {};
		\node [style=none] (183) at (2, 2) {};
	\end{pgfonlayer}
	\begin{pgfonlayer}{edgelayer}
		\draw [in=90, out=-90, looseness=0.75] (83.center) to (79.center);
		\draw [style=protected, in=-90, out=90, looseness=0.75] (81.center) to (82.center);
		\draw [style=none] (84.center) to (85);
		\draw [style=none, in=-90, out=15] (85) to (86.center);
		\draw [style=none, in=270, out=165] (85) to (87.center);
		\draw (91.center) to (86.center);
		\draw (92.center) to (90.center);
		\draw (88.center) to (90.center);
		\draw (94.center) to (83.center);
		\draw (95.center) to (82.center);
		\draw [style=none] (175.center) to (176);
		\draw [style=none, in=-90, out=15] (176) to (177.center);
		\draw [style=none, in=270, out=165] (176) to (178.center);
		\draw (182.center) to (177.center);
		\draw (183.center) to (179.center);
	\end{pgfonlayer}
\end{tikzpicture}
	\end{equation*}
	The epicness of $\phi$ implies that the same holds when $\phi$ is dropped; hence $B$ is classical.
\end{proof}

Combining \cref{lem:noninv_mono,lem:noninv_det_epi}, we also obtain the following.

\begin{corollary}
	The domain and codomain of a deterministic autocompatible isomorphism are classical.
\end{corollary}

\begin{definition}
	A morphism $\phi\colon A \to B$ in an involutive Markov category $\cat{C}$ is said to be \newterm{non-invasive} if it is compatible with the identity of $A$, i.e.~if
	\[
		\pairing{\phi}{\id} = \swap\comp \pairing{\id}{\phi}.
	\]
\end{definition}

\begin{remark}\label{rem:ni_prop}
	From \cref{rem:compatibility_prop} we deduce the following facts:
	\begin{enumerate}
		\item Any morphism with classical domain is non-invasive.
		\item Whenever $\phi$ is non-invasive and self-adjoint, then also $\pairing{\phi}{\id}$ is self-adjoint. 
		In particular, the non-invasiveness of the identity morphism $\id_A$ of $A$ is equivalent to the classicality of the object $A$.
	\end{enumerate}
	Moreover, a non-invasive morphism $\phi \colon A \to B$ is compatible with all morphisms out of $A$, as a simple string diagram calculation shows. In particular, every non-invasive morphism is also autocompatible.
\end{remark}

\begin{example}\label{ex:ni_center}
	A morphism $\phi\colon A \to B$ in $\calgmin$ or $\calgmax$ is non-invasive if and only if
	\[
		\opm{\phi}{}(x) \, y = y \, \opm{\phi}{}(x)
	\]
	for all $x \in \pre{B}$ and $y \in \pre{A}$. In other words, we can equivalently characterize non-invasive morphisms as those that map to the center of ${A}$.

	This notion of non-invasiveness is not to be confused with the concept of \emph{interaction-free measurement} in quantum theory~\cite{vaidman2001paradoxes}.
	While a general definition of what counts as ``interaction-free'' does not seem to exist, we can see as follows that our notion of non-invasiveness is quite different.
	If $A$ is a matrix algebra (of matrix size $>1$), then its centre is trivial, and therefore the only non-invasive morphisms $A \to B$ are those that factor across $I = \mathbb{C}$.
	This is a physically uninteresting operation, as it amounts to discarding the original system and preparing a new state.
	But the case of matrix algebras is the one that most quantum information literature focuses on, including the literature on interaction-free measurements, which can still be interesting and nontrivial despite trivial center.

	On the other hand, our notion of non-invasiveness also applies to morphisms that are not measurements.\footnote{Measurements can be defined as morphisms $A \to B$ for which $B$ is classical; for $B = \mathbb{C}^n$ in $\cpu_{\min/\max}$ (see \cref{def:cpu}), this reproduces the standard notion of POVM with finitely many outcomes.}
\end{example}

\begin{proposition}\label{prop:noninv_closure}
	Let $\phi$ and $\psi$ be any two morphisms in an involutive Markov category.
	\begin{enumerate}
		\item\label{it:tensor_ni} If $\phi$ and $\psi$ are non-invasive, then so is $\phi \tensor \psi$.
		\item\label{it:product_ni} If $\phi$ and $\psi$ have the same domain, then both are non-invasive if and only if $\pairing{\phi}{\psi}$ is non-invasive.
		\item \label{it:composition_ni} If $\phi$ and $\psi$ are composable and $\phi$ is non-invasive, then $\psi \comp \phi$ is non-invasive as well.
	\end{enumerate}
\end{proposition}
\begin{proof}
		Item~\ref{it:tensor_ni} is proved using naturality of the swap and the monoidal multiplicativity of copy from~\eqref{eq:icd_def_monoidal}.
		The ``if'' implication of Item~\ref{it:product_ni} follows by marginalization, while the ``only if'' is a consequence of \cref{prop:comp_product}.
	Concerning Item~\ref{it:composition_ni}, we have the following direct calculation:
	\begin{equation*}
		\begin{tikzpicture}
	\begin{pgfonlayer}{nodelayer}
		\node [style=none] (27) at (-3, 0.5) {};
		\node [style=bn] (28) at (-3, 1.75) {};
		\node [style=none] (29) at (-2, 3) {};
		\node [style=none] (30) at (-4, 3) {};
		\node [style=none] (31) at (-4, 3) {};
		\node [style=none] (32) at (-2, 3.5) {};
		\node [style=morphism] (33) at (-4, 3.25) {$\phi$};
		\node [style=none] (34) at (-2, 5.5) {};
		\node [style=none] (35) at (-4, 5.5) {};
		\node [style=morphism] (36) at (-4, 4.5) {$\psi$};
		\node [style=none] (37) at (0, 3) {$=$};
		\node [style=none] (40) at (4, 3) {};
		\node [style=none] (41) at (2, 2.75) {};
		\node [style=none] (42) at (2, 3) {};
		\node [style=none] (44) at (4, 4.25) {};
		\node [style=none] (45) at (2, 4.25) {};
		\node [style=none] (46) at (3, 0) {};
		\node [style=bn] (47) at (3, 1.25) {};
		\node [style=none] (48) at (4, 2.25) {};
		\node [style=none] (49) at (2, 2.25) {};
		\node [style=none] (50) at (2, 2.25) {};
		\node [style=none] (51) at (4, 2.5) {};
		\node [style=none] (52) at (2, 2.5) {};
		\node [style=none] (53) at (4, 3) {};
		\node [style=none] (54) at (2, 3) {};
		\node [style=morphism] (55) at (2, 4.75) {$\psi$};
		\node [style=morphism] (56) at (4, 2.5) {$\phi$};
		\node [style=none] (57) at (2, 6) {};
		\node [style=none] (58) at (4, 6) {};
		\node [style=none] (59) at (6, 3) {$=$};
		\node [style=none] (60) at (10, 4.5) {};
		\node [style=none] (61) at (8, 4) {};
		\node [style=none] (62) at (8, 4.5) {};
		\node [style=none] (64) at (10, 6) {};
		\node [style=none] (65) at (8, 6) {};
		\node [style=none] (66) at (9, 0) {};
		\node [style=bn] (67) at (9, 1.25) {};
		\node [style=none] (68) at (10, 2.5) {};
		\node [style=none] (69) at (8, 2.5) {};
		\node [style=none] (70) at (8, 2.5) {};
		\node [style=none] (71) at (10, 2.75) {};
		\node [style=none] (72) at (8, 2.75) {};
		\node [style=none] (73) at (10, 4.5) {};
		\node [style=none] (74) at (8, 4.5) {};
		\node [style=morphism] (75) at (10, 4) {$\psi$};
		\node [style=morphism] (76) at (10, 2.75) {$\phi$};
		\node [style=none] (77) at (8, 6) {};
		\node [style=none] (78) at (10, 6) {};
	\end{pgfonlayer}
	\begin{pgfonlayer}{edgelayer}
		\draw [style=none] (27.center) to (28);
		\draw [style=none, in=-90, out=15] (28) to (29.center);
		\draw [style=none, in=270, out=165] (28) to (30.center);
		\draw (34.center) to (29.center);
		\draw (35.center) to (33);
		\draw (31.center) to (33);
		\draw [in=90, out=-90, looseness=0.75] (45.center) to (40.center);
		\draw [style=protected, in=-90, out=90, looseness=0.75] (42.center) to (44.center);
		\draw [style=none] (46.center) to (47);
		\draw [style=none, in=-90, out=15] (47) to (48.center);
		\draw [style=none, in=270, out=165] (47) to (49.center);
		\draw (53.center) to (48.center);
		\draw (54.center) to (52.center);
		\draw (50.center) to (52.center);
		\draw (57.center) to (45.center);
		\draw (58.center) to (44.center);
		\draw [in=90, out=-90, looseness=0.75] (65.center) to (60.center);
		\draw [style=protected, in=-90, out=90, looseness=0.75] (62.center) to (64.center);
		\draw [style=none] (66.center) to (67);
		\draw [style=none, in=-90, out=15] (67) to (68.center);
		\draw [style=none, in=270, out=165] (67) to (69.center);
		\draw (73.center) to (68.center);
		\draw (74.center) to (72.center);
		\draw (70.center) to (72.center);
		\draw (77.center) to (65.center);
		\draw (78.center) to (64.center);
	\end{pgfonlayer}
\end{tikzpicture}\qedhere
	\end{equation*}
\end{proof}

\subsection{Almost sure equalities}\label{sec:as}

Almost sure equality has been introduced in categorical probability in a special case in~\cite[Definition~5.1]{chojacobs2019strings} and in~\cite[Definition~13.1]{fritz2019synthetic} in general.
We now develop this notion for involutive Markov categories
based on Parzygnat's~\cite[Definition~5.1]{quantum-markov}, where the first two of the following variants were considered.

\begin{definition}
	Let $\omega\colon A \to B$ and $\phi,\psi\colon B \to C$ be morphisms in an involutive Markov category.
	\begin{enumerate}
		\item $\phi$ and $\psi$ are \newterm{left (resp.~right) $\as{\omega}$~equal}, in symbols $\phi \asel{\omega} \psi$ (resp.~$\phi\aser{\omega} \psi$), if 
			\begin{equation}
				\begin{tikzpicture}
	\begin{pgfonlayer}{nodelayer}
		\node [style=none] (0) at (-8.5, -2) {};
		\node [style=bn] (1) at (-8.5, 0) {};
		\node [style=none] (2) at (-7.5, 1) {};
		\node [style=none] (3) at (-9.5, 1) {};
		\node [style=none] (4) at (-9.5, 1) {};
		\node [style=none] (5) at (-7.5, 1.25) {};
		\node [style=morphism] (6) at (-9.5, 1.25) {$\phi$};
		\node [style=none] (12) at (-7.5, 2.25) {};
		\node [style=none] (13) at (-9.5, 2.25) {};
		\node [style=none] (16) at (-5.5, 0) {$=$};
		\node [style=morphism] (17) at (-8.5, -1) {$\omega$};
		\node [style=none] (18) at (-7.5, 2.75) {$B$};
		\node [style=none] (19) at (-9.5, 2.75) {$C$};
		\node [style=none] (22) at (-8.5, -2.5) {$A$};
		\node [style=none] (27) at (-2.5, -2) {};
		\node [style=bn] (28) at (-2.5, 0) {};
		\node [style=none] (29) at (-1.5, 1) {};
		\node [style=none] (30) at (-3.5, 1) {};
		\node [style=none] (31) at (-3.5, 1) {};
		\node [style=none] (32) at (-1.5, 1.25) {};
		\node [style=morphism] (33) at (-3.5, 1.25) {$\psi$};
		\node [style=none] (34) at (-1.5, 2.25) {};
		\node [style=none] (35) at (-3.5, 2.25) {};
		\node [style=morphism] (36) at (-2.5, -1) {$\omega$};
		\node [style=none] (37) at (-1.5, 2.75) {$B$};
		\node [style=none] (38) at (-3.5, 2.75) {$C$};
		\node [style=none] (39) at (-2.5, -2.5) {$A$};
		\node [style=none] (40) at (4.5, -2) {};
		\node [style=bn] (41) at (4.5, 0) {};
		\node [style=none] (42) at (5.5, 1) {};
		\node [style=none] (43) at (3.5, 1) {};
		\node [style=none] (44) at (3.5, 1) {};
		\node [style=none] (45) at (5.5, 1.25) {};
		\node [style=none] (46) at (3.5, 1.25) {};
		\node [style=none] (47) at (5.5, 2.25) {};
		\node [style=none] (48) at (3.5, 2.25) {};
		\node [style=none] (49) at (7.5, 0) {$=$};
		\node [style=morphism] (50) at (4.5, -1) {$\omega$};
		\node [style=none] (51) at (3.5, 2.75) {$B$};
		\node [style=none] (52) at (5.5, 2.75) {$C$};
		\node [style=none] (53) at (4.5, -2.5) {$A$};
		\node [style=none] (54) at (10.5, -2) {};
		\node [style=bn] (55) at (10.5, 0) {};
		\node [style=none] (56) at (11.5, 1) {};
		\node [style=none] (57) at (9.5, 1) {};
		\node [style=none] (58) at (9.5, 1) {};
		\node [style=none] (59) at (11.5, 1.25) {};
		\node [style=none] (60) at (9.5, 1.25) {};
		\node [style=none] (61) at (11.5, 2.25) {};
		\node [style=none] (62) at (9.5, 2.25) {};
		\node [style=morphism] (63) at (10.5, -1) {$\omega$};
		\node [style=none] (64) at (11.5, 2.75) {$C$};
		\node [style=none] (65) at (9.5, 2.75) {$B$};
		\node [style=none] (66) at (10.5, -2.5) {$A$};
		\node [style=morphism] (67) at (5.5, 1.25) {$\phi$};
		\node [style=morphism] (68) at (11.5, 1.25) {$\psi$};
		\node [style=none] (69) at (2.5, 3) {};
		\node [style=none] (71) at (2.5, -2.75) {};
		\node [style=none] (72) at (12.5, 3) {};
		\node [style=none] (73) at (12.5, -2.75) {};
	\end{pgfonlayer}
	\begin{pgfonlayer}{edgelayer}
		\draw [style=none] (0.center) to (1);
		\draw [style=none, bend right=45] (1) to (2.center);
		\draw [style=none, bend left=45] (1) to (3.center);
		\draw (12.center) to (2.center);
		\draw (13.center) to (6);
		\draw (4.center) to (6);
		\draw [style=none] (27.center) to (28);
		\draw [style=none, bend right=45] (28) to (29.center);
		\draw [style=none, bend left=45] (28) to (30.center);
		\draw (34.center) to (29.center);
		\draw (35.center) to (33);
		\draw (31.center) to (33);
		\draw [style=none] (40.center) to (41);
		\draw [style=none, bend right=45] (41) to (42.center);
		\draw [style=none, bend left=45] (41) to (43.center);
		\draw (47.center) to (42.center);
		\draw (48.center) to (46.center);
		\draw (44.center) to (46.center);
		\draw [style=none] (54.center) to (55);
		\draw [style=none, bend right=45] (55) to (56.center);
		\draw [style=none, bend left=45] (55) to (57.center);
		\draw (61.center) to (56.center);
		\draw (62.center) to (60.center);
		\draw (58.center) to (60.center);
		\draw [bend right, looseness=0.50] (69.center) to (71.center);
		\draw [bend left, looseness=0.50] (72.center) to (73.center);
	\end{pgfonlayer}
\end{tikzpicture}
			\end{equation}
		\item $\phi$ and $\psi$ are \newterm{$\as{\omega}$~equal}, in symbols $\phi \ase{\omega} \psi$, if they are both left and right $\as{\omega}$ equal.
		\item $\phi$ and $\psi$ are \newterm{symmetrically $\as{\omega}$~equal}, in symbols $\phi \asets{\omega}\psi$, if
			\begin{equation}\label{eq:as_eq_twosided}
				\begin{tikzpicture}
	\begin{pgfonlayer}{nodelayer}
		\node [style=none] (0) at (-3.5, -2) {};
		\node [style=bn] (1) at (-3.5, 0) {};
		\node [style=none] (2) at (-2, 1) {};
		\node [style=none] (3) at (-3.5, 1) {};
		\node [style=none] (4) at (-3.5, 1) {};
		\node [style=none] (5) at (-2, 1.25) {};
		\node [style=morphism] (6) at (-3.5, 1.25) {$\phi$};
		\node [style=none] (12) at (-2, 2.25) {};
		\node [style=none] (13) at (-3.5, 2.25) {};
		\node [style=none] (16) at (0, 0) {$=$};
		\node [style=morphism] (17) at (-3.5, -1) {$\omega$};
		\node [style=none] (18) at (-2, 2.75) {$B$};
		\node [style=none] (19) at (-3.5, 2.75) {$C$};
		\node [style=none] (22) at (-3.5, -2.5) {$A$};
		\node [style=none] (74) at (-5, 1) {};
		\node [style=none] (75) at (-5, 1.25) {};
		\node [style=none] (76) at (-5, 2.25) {};
		\node [style=none] (78) at (-5, 2.75) {$B$};
		\node [style=none] (79) at (3.5, -2) {};
		\node [style=bn] (80) at (3.5, 0) {};
		\node [style=none] (81) at (5, 1) {};
		\node [style=none] (82) at (3.5, 1) {};
		\node [style=none] (83) at (3.5, 1) {};
		\node [style=none] (84) at (5, 1.25) {};
		\node [style=morphism] (85) at (3.5, 1.25) {$\psi$};
		\node [style=none] (86) at (5, 2.25) {};
		\node [style=none] (87) at (3.5, 2.25) {};
		\node [style=morphism] (88) at (3.5, -1) {$\omega$};
		\node [style=none] (89) at (5, 2.75) {$B$};
		\node [style=none] (90) at (3.5, 2.75) {$C$};
		\node [style=none] (91) at (3.5, -2.5) {$A$};
		\node [style=none] (92) at (2, 1) {};
		\node [style=none] (93) at (2, 1.25) {};
		\node [style=none] (94) at (2, 2.25) {};
		\node [style=none] (95) at (2, 2.75) {$B$};
	\end{pgfonlayer}
	\begin{pgfonlayer}{edgelayer}
		\draw [style=none] (0.center) to (1);
		\draw [style=none, in=270, out=0] (1) to (2.center);
		\draw [style=none] (1) to (3.center);
		\draw (12.center) to (2.center);
		\draw (13.center) to (6);
		\draw (4.center) to (6);
		\draw (76.center) to (74.center);
		\draw [in=-180, out=-90] (74.center) to (1);
		\draw [style=none] (79.center) to (80);
		\draw [style=none, in=270, out=0] (80) to (81.center);
		\draw [style=none] (80) to (82.center);
		\draw (86.center) to (81.center);
		\draw (87.center) to (85);
		\draw (83.center) to (85);
		\draw (94.center) to (92.center);
		\draw [in=-180, out=-90] (92.center) to (80);
	\end{pgfonlayer}
\end{tikzpicture}
			\end{equation}
	\end{enumerate}
\end{definition}

\begin{remark}
	We emphasize that $\phi \asets{\omega} \psi$ implies $\phi \ase{\omega}\psi$, simply by marginalizing the left and the right output of \eqref{eq:as_eq_twosided} separately.
	The converse is not true in general, as we will see in \cref{ex:ase_not_sym} as part of a more detailed study of almost sure equalities in $\calgmin$ and $\calgmax$.
\end{remark}

\begin{remark}
	For $\omega \colon A \to B$ with $B$ classical, all four notions of almost sure equality coincide, as a direct consequence of $\swap_{B,B}\comp \cop_B = \cop_B$.  
\end{remark}

We now recall the following tool that will be used in \cref{sec:as_nullspaces}, and which has appeared as~\cite[Corollary~5.6]{quantum-markov}.

\begin{lemma}\label{lem:asleftright}
	If $\omega$, $\phi$ and $\psi$ are self-adjoint, then
	\[
		\phi \asel{\omega} \psi \quad \Longleftrightarrow \quad {\phi} \aser{\omega} {\psi}.
	\]
In particular, whenever one of the two \as{} equalities holds, then $\phi \ase{\omega}\psi$ also holds.
\end{lemma}

Nevertheless, even in that case the symmetric \as{} equality may still be different (\cref{ex:ase_not_sym}).

\begin{lemma}\label{lem:asel_asets}
	If $\omega\colon A \to B$ is such that $\asel{\omega}$ coincides with $\aser{\omega}$, then all four notions of $\as{\omega}$ equality coincide.
\end{lemma}
\begin{proof}
	It suffices to note that, whenever $\phi \asel{\omega} \psi$, then $\pairing{\phi}{\id} \asel{\omega} \pairing{\psi}{\id}$, and this is simply by associativity of copy. Then $\pairing{\phi}{\id} \aser{\omega} \pairing{\psi}{\id}$ follows by assumption, and this is exactly $\phi \asets{\omega} \psi$.
\end{proof}

\section{Pictures}\label{sec:pictures}

Involutive Markov categories like $\calgmin$ and $\calgmax$ contain ``too many'' morphisms to be directly relevant to quantum probability.
This is already indicated by the isomorphisms $M_n \cong \mathbb{C}^{n^2}$ from \cref{rem:nondet_iso}, but illustrated more drastically by the well-known \emph{no--broadcasting theorem}~\cite{barnum2007nobroadcasting}.
In our setting, this states that the copy morphisms do not represent physically realizable processes.\footnote{See semicartesian categorical probability~\cite{fritz2022dilations,houghtonlarsen2021dilations} for a different categorical approach to quantum probability which considers only physically realizable processes from the start.}
So involutive Markov categories constitute an abstract framework where we can broadcast, but this operation does not correspond to any physically implementable operation.\footnote{See also~\cite{coecke2012picturing}, where the phrase \emph{logical broadcasting} has been coined for this idea.}
For this reason, an involutive Markov category acts as an \emph{environment} that enables string diagrammatic calculus, while physically realizable operations form a subcategory. 
This is analogous to how operator algebra provides an environment for quantum probability, while the physical observables are only the self-adjoint ones, and the physical operations are the completely positive ones.
With these considerations in mind, in this section we introduce the notion of \emph{picture}, which restricts from the environment to a subclass of morphisms, to be thought of as the physically realizable ones.

In contrast to involutive Markov categories, a picture is a good candidate for satisfying information flow axioms such as positivity and causality, as noted similarly by Parzygnat~\cite[Theorem~4.2 and Proposition~8.34]{quantum-markov}.
Moreover, it is a good framework to discuss representability, as we will investigate in \cref{sec:representability_qdf}.

\subsection{Definition of pictures}\label{sec:picture_def}

Before introducing the concept of picture, let us consider a toy example that exists in every involutive Markov category. 

\begin{example}[The self-adjoint subcategory]
	Let $\cat{C}$ be an involutive Markov category. We define $\operatorname{SA}(\cat{C})$ to be the subcategory of self-adjoint morphisms.
	This subcategory has the following properties:
	\begin{enumerate}
		\item\label{it:sa_tensor} It is a symmetric monoidal subcategory containing the delete morphisms (by definition).
		\item Every morphism in it is self-adjoint (trivially).
		\item Whenever an autocompatible self-adjoint morphism $\phi\colon A \to B$ is compatible with a self-adjoint morphism $\psi\colon A \to C$, then also $\pairing{\phi}{\psi}\colon A \to B \tensor C$ is self-adjoint, as a special case of \cref{rem:compatibility_prop}\ref{it:self_comp}.
	\end{enumerate}
\end{example}

These properties are the ones we now impose in the general definition of ``picture''. 

\begin{definition}[Picture]\label{def:picture}
	Given an involutive Markov category $\cat{C}$, a wide subcategory\footnote{This means that every object of $\cat{C}$ also belongs to $\cD$, and implements the idea that a picture is a choice of morphisms rather than a choice of objects.} $\cD \subseteq \cat{C}$ is a \newterm{picture} in $\cat{C}$ if:
	\begin{enumerate}
		\item\label{it:det_weak} It is a symmetric monoidal subcategory containing the delete morphisms. 
		\item Every morphism of $\cD$ is self-adjoint.
		\item\label{it:ni_weak} Whenever an autocompatible morphism $\phi\colon A \to B$ (of $\cD$) is compatible with a morphism $\psi\colon A \to C$ (of $\cD$), then also $\pairing{\phi}{\psi}\colon A \to B \tensor C$ belongs to $\cD$. 
	\end{enumerate} 
\end{definition}

When dealing with pictures, we reserve the term \newterm{morphism} for the morphisms in $\cD$, while morphisms in $\cat{C}$ will be called \newterm{generalized morphisms}. 
Generalized morphisms in string diagrams will be denoted by dashed boxes: 
\begin{equation*}
	\begin{tikzpicture}
	\begin{pgfonlayer}{nodelayer}
		\node [style=none] (85) at (-3.5, -1.25) {};
		\node [style=morphism] (88) at (-3.5, 0) {$\phi$};
		\node [style=none] (89) at (-3.5, 1.25) {};
		\node [style=none] (90) at (-2, 0) {$\in \cat{D},$};
		\node [style=none] (91) at (2, -1.25) {};
		\node [style=generic morphism] (92) at (2, 0) {$\psi$};
		\node [style=none] (93) at (2, 1.25) {};
		\node [style=none] (94) at (3.5, 0) {$\in \cat{C}.$};
	\end{pgfonlayer}
	\begin{pgfonlayer}{edgelayer}
		\draw (85.center) to (88);
		\draw (88) to (89.center);
		\draw (91.center) to (92);
		\draw (92) to (93.center);
	\end{pgfonlayer}
\end{tikzpicture}
\end{equation*}
The copy morphisms form an exception: although these typically do not belong to the picture, we will still use the same notation as before. 

Further, the term \newterm{state} will refer to the morphisms (not generalized!) whose domain is the monoidal unit. In string diagrams, this type of morphism is depicted as 
\begin{equation*}
	\begin{tikzpicture}
	\begin{pgfonlayer}{nodelayer}
		\node [style=state] (88) at (0, -0.5) {$\phi$};
		\node [style=none] (89) at (0, 1) {};
	\end{pgfonlayer}
	\begin{pgfonlayer}{edgelayer}
		\draw (88) to (89.center);
	\end{pgfonlayer}
\end{tikzpicture}

\end{equation*}
to emphasize that such morphisms ``have no input''.

For brevity, we usually leave $\cat{C}$ implicit and say that $\cD$ is a picture.
Whenever the containing category $\cat{C}$ needs to be referenced, we will denote it by $\gen{\cD}$.
This reminds us of the idea that $\gen{\cD}$ is the category of generalized morphisms of $\cD$.
Also, we will use the shorthand 
\[
\cD_{\det} \coloneqq \cD \cap \cat{C}_{\det}	
\]
to indicate the symmetric monoidal subcategory of deterministic morphisms.
We do not require $\cD_{\det}= \gen{\cD}_{\det}$ to hold, as this will be false in our setting (our pictures of interest are given by bounded maps, and $*$-homomorphisms are not always bounded by \cref{ex:homnotbound}).

\begin{remark}\label{rem:sam}
	As suggested to us by Sam Staton, a picture could alternatively be defined in terms of a strong symmetric monoidal faithful functor $U\colon \cD \to \operatorname{SA}(\cat{C})$ rather than in terms of $\cD$ being a subcategory.
	This version may be preferred for better analogy with forgetful functors.
	However, we have not adopted this definition as it would clutter our notation further, while being essentially equivalent: $U$ can always be taken to be such an inclusion functor, at least modulo replacing $\gen{\cD}$ by an equivalent involutive Markov category.
\end{remark}

Given how we mimicked the situation of the self-adjoint subcategory, one might wonder whether property \ref{it:ni_weak} in \cref{def:picture} should be strengthened by requiring that for all compatible morphisms $\phi$ and $\psi$ in $\cD$, also $\pairing{\phi}{\psi}$ belongs to $\cD$ (see \cref{rem:compatibility_prop}\ref{it:self_comp}). 
While this is possible, imposing such a condition would exclude the minimal C*-tensor norm as a viable monoidal structure (\cref{prop:display_compatibility}). 
For this reason, we prefer to treat this strenghtening as an additional property of a picture.
For our present purposes, this property is used only to distinguish the minimal and the maximal tensor product and will not be needed otherwise.

\begin{definition}\label{def:display_compatibility}
	A picture $\cD$ \newterm{displays compatibility} if, whenever $\phi$ and $\psi$ are compatible, then $\pairing{\phi}{\psi}$ is a morphism in $\cD$.
\end{definition}
As already discussed above, given any involutive Markov category $\cat{C}$, the self-adjoint subcategory $\operatorname{SA}(\cat{C})$ is a picture that displays compatibility: see \cref{rem:compatibility_prop}\ref{it:self_comp}.

\begin{remark}\label{rem:display_autocomp}
	Property~\ref{it:ni_weak} of \cref{def:picture} already allows some interesting ``displaying''. 
	For instance, if a morphism $\phi$ is non-invasive, then $\pairing{\phi}{\id}$ is also a morphism by using the mentioned property with $\psi=\id$.

	Another example is given by autocompatibility: if $\phi$ is an autocompatible morphism, an inductive argument combining \cref{cor:autocomp_powers} and Property~\ref{it:ni_weak} shows that $\phi^{(n)}$, for any positive integer $n$, also belongs to the picture. 
	In fact, this can also be reversed using~\cref{rem:compatibility_prop}\ref{it:self_comp}: 
	if $\phi^{(n)}$ is a morphism for some $n\ge 2$, then $\phi$ is autocompatible.
\end{remark}

\begin{remark}\label{rem:class_pic}
	The classical subcategory of a picture $\cD$, defined as 
	\begin{equation*}
		\classical{\cD} \coloneqq \cD \cap \classical{\gen{\cD}},
	\end{equation*}
	is a Markov category. 
	Indeed, by Property~\ref{it:ni_weak} of \cref{def:picture},  \cref{rem:compatibility_prop}\ref{it:class_comp} implies that the copy morphisms belong to $\classical{\cD}$.
\end{remark}

\subsection{Main example: completely positive maps}\label{sec:mainexample_cpu}

The standard study of quantum probability is concerned with states on C*-algebras, and more generally completely positive unital maps between them.
Our aim is to retrieve this setting, although not perfectly (see \cref{rem:prec_and_c_notequal}).
For this reason, we consider the following notion of positivity, defined with respect to the containing C*-algebra.

\begin{definition}\label{def:positive_el}
	For a pre-C*-algebra $A$ with completion $\closed{A}$, we say that $x \in {A}$ is \newterm{positive}, and we write $x \ge 0$, if $x=y^* y$ for some $y \in \closed{A}$.
\end{definition}
Positivity induces a partial order, with $x \ge y$ shorthand for $x-y \ge 0$. 

\begin{definition}\label{def:positive_map}
	Let ${A}$ and ${B}$ be pre-C*-algebras. Then a unital linear map $\opm{\phi}{}\colon {B} \rightsquigarrow {A}$ is 
	\begin{itemize}
		\item \newterm{positive} if $\opm{\phi}{}(x)\ge 0$ for all $x\ge 0$;
		\item \newterm{completely positive} if $\opm{\phi}{} \odot \id_{M_n}$ is positive\footnote{One may wonder if we should emphasize which C*-norm is used; this is luckily not the case because $M_n$ is a nuclear C*-algebra (see, for example,~\cite[Proposition~T.5.20]{weggeolsen}).} for all $n$ ($M_n$ is the C*-algebra of matrices $n \times n$ with complex entries).
	\end{itemize}
	A positive unital functional $\opm{\phi}{} \colon {B} \rightsquigarrow \mathbb{C}$ is called \newterm{state}.
\end{definition}

\begin{proposition}\label{prop:pu_bounded}
	The positive unital maps ${B} \rightsquigarrow {A}$ are precisely the linear unital maps of norm $\le 1$, and they are self-adjoint.
\end{proposition}

\begin{proof}
	By composing with the inclusion ${A} \hookrightarrow \closed{A}$, we can assume without loss of generality that ${A}$ is complete.
	Then the claim follows from standard results in the literature, which we recall for the reader's convenience.

	First, assume $\opm{\phi}{}\colon B\rightsquigarrow A$ is positive unital. 
	The proof of \cite[Corollary 2.9]{paulsen2002} goes through even if $B$ is not complete.
	This shows that $\norm{\opm{\phi}{}}$ coincides with $\norm{\opm{\phi}{}(1)} =\norm{1}$, so $\opm{\phi}{}$ is indeed linear unital of norm $\le 1$. 

	Conversely, if $\opm{\phi}{}\colon B\rightsquigarrow A$ is linear unital of norm $\le 1$, then $\opm{\phi}{}$ is positive by \cite[Proposition 2.11]{paulsen2002}.
	Finally, the fact that all positive maps are self-adjoint is covered in \cite[Exercise 2.1]{paulsen2002}.\footnote{This exercise is quite simple once enough theory is in place. For example, it follows from the fact that $x+\norm{x}\ge 0$ for all self-adjoint $x$ (see \cref{cor:bound_with_norm}), since every positive element is self-adjoint.}
\end{proof}

\begin{proposition}\label{prop:Phi_closPhi_positivity}
	A unital map $\opm{\phi}{}\colon {B} \rightsquigarrow {A}$ is positive or completely positive if and only if the unique extension $\close{\opm{\phi}{}} \colon \closed{B} \rightsquigarrow \closed{A}$ is. 
\end{proposition}
This proposition is crucial in our study, as it allows to restrict some proofs to the case where $A$ and $B$ are both C*-algebras. 
For example, we can infer that bounded $*$-homomorphisms between pre-C*-algebras are necessarily completely positive because this is true for C*-algebras~\cite[Example~1.5.2]{brownozawa}.
\begin{proof}
	The ``if'' direction is clear.
	For the ``only if'', let us consider a positive element $x=y^*y$ of $\closed{{B}}$. Since $y \in \closed{{B}}$, there is a convergent sequence $y_n \to y$ with $y_n \in {B}$ for all $n$. In particular, this means that $y_n^* y_n \to x$.
	By definition of $\close{\opm{\phi}{}}$, we have $\close{\opm{\phi}{}}(x) = \lim_n {\opm{\phi}{}} (y_n^* y_n)$.
	Since the set of positive elements is closed~\cite[Theorem~1.4.8]{lin2001amenable}, we conclude that $\close{\opm{\phi}{}} (x)$ is still positive. 

	In order to deal with complete positivity, let us first show that $\close{\opm{\phi}{} \odot \id_{M_n}}= \close{\opm{\phi}{}} \odot \id_{M_n}$. Indeed, $M_n$ is nuclear and the algebraic tensor products with $M_n$ are already complete~\cite[Proposition~T.5.20]{weggeolsen}. This means that $\close{\opm{\phi}{} \odot \id_{M_n}} \colon \closed{B} \odot M_n \rightsquigarrow \closed{A} \odot M_n$. In particular, both $\close{\opm{\phi}{} \odot \id_{M_n}}$ and $\close{\opm{\phi}{}} \odot \id_{M_n}$ extend $\opm{\phi}{} \odot \id_{M_n}$, so they must coincide, according to \cref{lem:trick}. 
	We now can use the first part of this proof also for complete positivity: indeed, $\opm{\phi}{} \odot \id_{M_n}$ is positive if and only if $\close{\opm{\phi}{} \odot \id_{M_n}}= \close{\opm{\phi}{}} \odot \id_{M_n}$ is positive. In particular, $\opm{\phi}{}$ is completely positive if and only if $\close{\opm{\phi}{}}$ is completely positive.
\end{proof}

\begin{proposition}\label{prop:cpu_tensor}
	Let $\opm{\phi}{}\colon {B} \rightsquigarrow {A}$ be a completely positive unital map.
	Then for every pre-C*-algebra ${C}$, the map
	\[
		\opm{\phi}{} \tensor \id_{{C}} \: \colon \: {B} \tensor {C} \rightsquigarrow {A} \tensor {C}
	\]
	is completely positive unital with respect to both the minimal and the maximal tensor norm.
\end{proposition}
\begin{proof}
	By \cref{prop:Phi_closPhi_positivity}, the extension $\close{\opm{\phi}{}}\colon \closed{B}\rightsquigarrow \closed{A}$ is completely positive unital.
  	Applying the standard functoriality of the minimal and maximal C*-tensor products for completely positive maps~\cite[Theorem~3.5.3]{brownozawa} then yields that
  	\[
  		\close{\opm{\phi}{}}\otimes \id_{\closed{C}}
		\: \colon \:
  		\closed{B}\otimes \closed{C}
  		\rightsquigarrow
  		\closed{A}\otimes \closed{C}
  	\]
  	is completely positive unital, where as before $\otimes$ denotes the relevant completed tensor product.

	Clearly this map restricts to $\opm{\phi}{} \tensor \id_C$ on ${B} \odot {C}$, and similarly at the matrix level.
	Therefore this restriction is completely positive unital as well.
\end{proof}

\begin{proposition}[{\cite[Theorems~3.9 and 3.11]{paulsen2002}}]
	\label{prop:com_dom_codom}
	A positive unital map with commutative domain or commutative codomain is completely positive. In particular, all states are completely positive.
\end{proposition}

\begin{definition}\label{def:cpu}
	In the category $\calgone$, we denote by $\cpu$ the subcategory whose morphisms are (formal opposites of) completely positive unital linear maps.
\end{definition}
\begin{notation}
Whenever $\cpu$ is considered as a subcategory of either $\calgmin$ or $\calgmax$, we will decorate it accordingly: 
$\cpu_{\min}$ and $\cpu_{\max}$, respectively.
For brevity, we also use $\cpu_{\min/\max}$ to refer to both at the same time.
\end{notation}

As anticipated in \cref{rem:calgmin_calgmax_differ}, $\cpu_{\min}$ and $\cpu_{\max}$ are different \emph{symmetric monoidal} categories, since completely positive maps are bounded, and therefore influenced by the choice of the norm.
This difference manifests itself also in formal categorical properties in \cref{prop:display_compatibility}.

\begin{remark}[Relation with the context of C*-algebras]\label{rem:prec_and_c_notequal}
	The additional $\mathsf{p}$ in front of $\mathsf{CPU}$ is to stress the distinction between our setting and the more standard one where C*-algebras are considered instead of pre-C*-algebras.
	Indeed, the forgetful functor from the category of C*-algebras with completely positive unital maps to $\cpu^{\op}$ is a right adjoint, but not an equivalence. 

	The adjunction is an immediate consequence of \cref{prop:Phi_closPhi_positivity}. 
	To prove that this forgetful functor is not an equivalence, note that any pre-C*-algebra isomorphic to a C*-algebra in $\cpu$ is necessarily a C*-algebra, because isomorphisms are in particular isometries (and therefore completeness is preserved).

	This situation may suggest refining the idea considered in \cref{rem:sam} to incorporate C*-algebras with completely positive unital maps as a picture, while retaining $\calgone_{\min/\max}$ as the category of generalized morphisms. 
	However, keeping track of the two different tensors (in general, $\closed{A} \odot \closed{B} \subsetneq \closed{A} \otimes \closed{B}$) would result in subtleties that we believe would make the formalism harder to digest. 
	For this reason, we content ourselves with working with $\cpu$ and leave this alternative approach for future exploration.
\end{remark}

\begin{proposition}\label{prop:cpu_picture}
	The subcategories $\cpu_{\min/\max}$ are pictures.
	\end{proposition}
	\begin{proof}
		All morphisms are self-adjoint by \cref{prop:pu_bounded}. 
		Moreover, the subcategories $\cpu_{\min/\max}$ are symmetric monoidal since the coherence morphisms are bounded~\cite[Exercises~3.1.2, 3.3.1 and 3.3.2]{brownozawa} (recall \cref{lem:trick}), and the tensor product is respected by \cref{prop:cpu_tensor}.
		Additionally, the unique unital map $\mathbb{C} \rightsquigarrow \pre{A}$ corresponding to deletion is completely positive.
	
		Concerning Property~\ref{it:ni_weak}, we recall \cref{ex:commutingranges}: 
		For any autocompatible morphism $\phi\colon A \to B$, the completely positive map $\opm{\phi}{}$ has commutative range. In particular, the range $\im(\opm{\phi}{})$ generates a commutative pre-C*-subalgebra $R \subseteq A$, and so $\phi$ admits a factorization $\phi = \phi' \comp \iota$, where $\opm{\iota}{} \colon R \rightsquigarrow A$ is the inclusion.
		Since any morphism ${\psi}\colon A \to C$ compatible with $\phi$ is also compatible with $\iota$, the map $\opm{\pairing{\iota}{\psi}}{}$ is therefore completely positive because commutative (pre-)C*-algebras are nuclear~\cite[Proposition~12.9]{paulsen2002},
		which means that the minimal and maximal tensor norms coincide.
		Therefore we can apply the universal property of the maximal tensor product~\cite[Exercise~3.5.1]{brownozawa} for both monoidal structures. 
		We conclude that $\pairing{\phi}{\psi}$ is the composition of two morphisms $\phi' \tensor \id$ and $\pairing{\iota}{\psi}$ which are both completely positive, and so $\pairing{\phi}{\psi}$ is completely positive as well.
	\end{proof}

\begin{remark}
	In~\cite{quantum-markov}, the focus is on subcategories that are not necessarily symmetric monoidal. 
	In fact, the information flow axioms considered there are obtained by the Kadison--Schwarz inequality (\cref{prop:ks_inequality} below), and all results can therefore be stated in the context of positive unital maps satisfying such an inequality.  
	However, there does not seem to be any concrete advantage to considering such generality, since the physically realizable operations are only the completely positive unital maps, and also the mathematical literature mostly focuses on these.
	Furthermore, considering only symmetric monoidal subcategories provides a more understandable playground, as the reader does not have to pay attention to when tensoring is allowed.
	An additional motivation is given by the proof of \cref{prop:cpu_picture}: 
	Property~\ref{it:ni_weak} holds provided that the morphisms considered are completely positive, since without this hypothesis the universal property of the maximal tensor product~\cite[Exercise~3.5.1]{brownozawa} could not be applied.
\end{remark}

\begin{example}\label{rem:positive_elements}
	Let $A$ be a pre-C*-algebra. 
	Following \cref{ex:observables}, the formal opposites of (completely) positive unital maps $\opm{\phi}{}\colon \mathbb{C}^{\lbrace 0,1\rbrace} \rightsquigarrow A$ correspond to the elements of the unit interval $[0, 1]$ in ${A}$.\footnote{As is standard for C*-algebras, the unit interval $[0,1] \subseteq {A}$ is the set of elements $x \in {A}$ such that $0\le x \le 1$.}
\end{example}

\begin{proposition}\label{prop:display_compatibility}
	$\cpu_{\max}$ displays compatibility, while $\cpu_{\min}$ does not.
\end{proposition}
\begin{proof}
	Since compatibility is equivalent to commuting ranges in $\calgmin$ and $\calgmax$ (\cref{ex:commutingranges}), this is the universal property of the maximal tensor product: \cite[Exercise~3.5.1]{brownozawa} explain why $\cpu_{\max}$ has this property,
	while~\cite[Exercise~3.6.3]{brownozawa} shows that this is not the case in $\cpu_{\min}$ by noting that the left and right regular representations of a non-amenable discrete group commute, but the relevant continuity with respect to the minimal tensor norm fails.
\end{proof}

We now conclude by discussing some other results that will be used in our study and are related to the notion of spectrum of an element. 
For a brief introduction to this concept, the reader may refer to~\cite[Section~1.3]{weggeolsen}.

\begin{definition}
	\begin{enumerate}
		\item For a C*-algebra $\closed{A}$ and $x \in \closed{A}$, the \newterm{spectrum} $\spec{x}$ is the set of values $\lambda \in \mathbb{C}$ such that $x-\lambda 1$ is not invertible. 
		\item For a pre-C*-algebra ${A}$ and $x \in {A}$, the \newterm{spectrum} $\spec{x}$ is the spectrum of $x$ in the completion $\closed{A}$.
	\end{enumerate}
\end{definition}

It is a crucial fact of the theory that every element has nonempty spectrum~\cite[Theorem~3.2.3]{kadison1997operatoralgebrasI}.

\begin{proposition}[{e.g.~\cite[Proposition~7.8]{conway2000}}]
	\label{prop:spectrum_states}
	Let $x \in A$ be a self-adjoint. Then its spectrum is a compact subset of $\mathbb{R}$ with
	\[
		\left[ \min \spec{x}, \max \spec{x} \right] = \Set{\opm{\phi}{}(x) \given \opm{\phi}{}\colon A \rightsquigarrow\mathbb{C} \text{ is a state}}.
	\]
\end{proposition}
Using Beurling's formula~\cite[p.~204]{khalkhali13noncommutative}, from which $\norm{x}=\max\Set{ \abs{\lambda} \given \lambda \in \spec{x}}$ for all self-adjoint $x$, we also get the following.
\begin{corollary}\label{cor:norm_state}
	The norm of a self-adjoint $x$ can be characterized in terms of states as
	\begin{equation*}
		\norm{x} = \max \, \Set{  \abs{\opm{\phi}{}(x)} \given \opm{\phi}{}\colon A \rightsquigarrow \mathbb{C} \text{ is a state}}.
	\end{equation*}
\end{corollary}
\begin{corollary}\label{cor:bound_with_norm}
	Every self-adjoint $x$ satisfies $- \norm{x} \le x \le \norm{x}$. 
	In particular, every element can be written as a linear combination of positive elements.
\end{corollary}

Since every $x \in A$ can be written uniquely as $x = y + iz$ for self-adjoint $y$ and $z$, we also obtain:

\begin{corollary}\label{cor:zero_states}
	Let $x \in A$. If $\opm{\phi}{}(x)=0$ for every state $\opm{\phi}{}$, then $x=0$.
\end{corollary}

\begin{corollary}\label{cor:self_contr_positive}
	Let $\opm{\phi}{}\colon {B}\rightsquigarrow {A}$ be a self-adjoint unital map which is contractive on self-adjoint elements. Then $\opm{\phi}{}$ is positive.
\end{corollary}
\begin{proof}[\proofname.\protect\footnotemark\protect\nopunct]
	\footnotetext{This proof is based on Eric Wofsey's answer at \href{https://math.stackexchange.com/questions/3304114}{math.stackexchange.com/questions/3304114}.}
	By the Hahn--Banach theorem, an element of ${A}$ is positive if and only if it is positive under all states.
	Since composing $\opm{\phi}{}$ with any state produces a map $B \rightsquigarrow \mathbb{C}$ with the same properties,
	it is therefore enough to prove the claim for $A=\mathbb{C}$.

	Let $x\in {B}$ be a positive element and consider $\lambda \in \mathbb{R}$ such that $\lambda < -\norm{x}$. 
	By \cref{cor:bound_with_norm} and positivity of $x$,
	\[ 
	\lambda \le x+\lambda \le \norm{x}+\lambda \le 0,
	\] 
	from which we obtain $\norm{x+\lambda} \le \abs{\lambda}$.
	By contractibility on self-adjoint elements, $\abs{\opm{\phi}{}(x+\lambda)} \le \norm{x+\lambda}$. 
	Self-adjointness of $\opm{\phi}{}$ ensures that $\opm{\phi}{}(x)$ is real.
	Since $\lambda < 0$, the inequality $\abs{\opm{\phi}{}(x)+\lambda}\le \abs{\lambda}$ implies
	\[
		(\opm{\phi}{}(x)+\lambda)^2\le \lambda^2,
	\]
	or equivalently $\opm{\phi}{}(x)(\opm{\phi}{}(x)+2\lambda)\le 0$.
	Thus $\opm{\phi}{}(x)$ lies between the two roots $0$ and $-2\lambda$, and in particular $\opm{\phi}{}(x)\ge 0$.
\end{proof}

\subsection{Almost sure equalities via nullspaces}\label{sec:as_nullspaces}

Almost sure equality is a central notion in probability, and we have already considered the various definitions in the involutive setting in \cref{sec:as}.
We devote this subsection to studying the meaning of these almost sure equalities in $\cpu_{\min/\max}$.

\begin{definition}
	Let $\opm{\omega}{}\colon \pre{B} \rightsquigarrow \pre{A}$ be a completely positive unital map between pre-C*-algebras. Then:
	\begin{enumerate}
		\item Its \newterm{left nullspace} (resp.~\newterm{right nullspace}) is the subset of $\pre{B}$ given by
			\[
				{}_{\opm{\omega}{}} N \coloneqq \Set{ x \in \pre{B} \given \opm{\omega}{} (xx^*)= 0} \qquad \left(\text{resp. }N_{\opm{\omega}{}} \coloneqq \Set{ x \in \pre{B} \given \opm{\omega}{}(x^*x)=0 } \right).
			\] 
		\item Its \newterm{symmetric nullspace} is the subset of $\pre{B}$ given by
			\[
				SN_{\opm{\omega}{}} \coloneqq \Set{x \in \pre{B} \given \opm{\omega}{}(y^*x^* xy)=0 \quad \forall\, y \in \pre{B}}.
			\]
	\end{enumerate}
\end{definition}
	
Let us recall the following important inequality.
\begin{proposition}\label{prop:ks_inequality}
	Any completely positive unital map $\opm{\phi}{}\colon {B}\rightsquigarrow {A}$ satisfies the \newterm{Kadison--Schwarz inequality}: for all $x \in {B}$, 
	\[
		\opm{\phi}{}(x^*x) \ge \opm{\phi}{}(x)^* \opm{\phi}{}(x).
	\]
\end{proposition}
\begin{proof}
	By \cref{prop:Phi_closPhi_positivity}, we can assume the map to be between C*-algebras. 
	Now the result follows from~\cite[Proposition~3.3]{paulsen2002}.
\end{proof}

\begin{lemma}
	\label{lem:ideals}
	The left (resp.~right) nullspace is the largest closed right (resp.~left) ideal contained in the kernel of $\opm{\omega}{}$, and
	\[
		x \in N_{\opm{\omega}{}} \qquad \Longleftrightarrow \qquad x^* \in {}_{\opm{\omega}{}} N.
	\]
\end{lemma}

\begin{proof}
	First, ${}_{\opm{\omega}{}}N$ (resp.~$N_{\opm{\omega}{}}$) is a right (resp.~left) ideal by the inequality $0 \le y y^* \le \norm{y}^2$ (\cref{cor:bound_with_norm}) and positivity of $\opm{\omega}{}$.
	The closedness is straightforward.
	To prove that both ideals are contained in the kernel of $\opm{\omega}{}$, we note that, by the Kadison--Schwarz inequality (\cref{prop:ks_inequality}),  	
	\[
		0=\opm{\omega}{}(x x^*) \ge \opm{\omega}{}(x) \, \opm{\omega}{}(x)^* \ge 0.	
	\]
	Hence these inequalities must be equalities, and $\opm{\omega}{}(x)=0$.
	
	For the maximality, let $J$ be a right ideal contained in the kernel of $\opm{\omega}{}$. 
	Then we indeed must have $J \subseteq {}_{\opm{\omega}{}}N$, because $x \in J$ gives $xx^* \in J$, and therefore $\opm{\omega}{}(xx^*)=0$.
	The final statement holds by definition.
\end{proof}

\begin{lemma}
	\label{lem:SN_ideal}
	The symmetric nullspace $SN_{\opm{\omega}{}}$ is the largest closed two-sided ideal contained in the kernel of $\opm{\omega}{}$, and it coincides with
	\[
		S {}_{\opm{\omega}{}} N \coloneqq \Set{x \in \pre{B} \given \opm{\omega}{}(yxx^*y^*)=0 \quad \forall \, y \in \pre{B}}.
	\]
	Moreover, it is a $*$-ideal, i.e.~$x \in SN_{\opm{\omega}{}}$ if and only if $x^* \in SN_{\opm{\omega}{}}$.
\end{lemma}

\begin{proof}
	$SN_{\opm{\omega}{}}$ is a two-sided ideal by the following general fact applied to $J = N_{\opm{\omega}{}}$: Whenever $J$ is a left ideal in a ring $R$, then the ideal quotient
	\[
		(J : R) \coloneqq \Set{x \in R \given xy \in J\quad \forall\, y \in R}	
	\]
	is clearly a two-sided ideal.
	It is closed by construction as the intersection of the closed sets
	\[
		\Set{ x \given \opm{\omega}{}(y^* x^* x y) = 0 }.
	\]
	Moreover, it is contained in the kernel of $\opm{\omega}{}$ because $SN_{\opm{\omega}{}}\subseteq N_{\opm{\omega}{}}$, and it is the largest two-sided ideal contained in the kernel since any other such ideal $J$ must satisfy $J \subseteq N_{\opm{\omega}{}}$ by \Cref{lem:ideals} and therefore $J \subseteq SN_{\opm{\omega}{}}$ by the closedness under multiplication from the right.
	This maximality property also implies that $SN_{\opm{\omega}{}}$ coincides with the given $S {}_{\opm{\omega}{}} N$.
	The final statement follows from $SN_{\opm{\omega}{}}=S {}_{\opm{\omega}{}} N$.
\end{proof}

\begin{lemma}\label{lem:multiplication_nullspace}
	If $\opm{\omega}{}\colon \pre{B} \rightsquigarrow \pre{A}$ is a completely positive unital map and $\opm{\omega}{}(x^*x)=0$ for some $x\in \pre{B}$, then also
	\[
		\opm{\omega}{} (x^* y) = 0 = \opm{\omega}{}(y^*x)
	\]
	for every $y \in \pre{B}$.
\end{lemma}
\begin{proof}
	By \cref{lem:ideals}, the right nullspace is a left ideal contained in the kernel of $\opm{\omega}{}$.
	This means that $x \in N_{\opm{\omega}{}}$ implies $y^*x \in N_{\opm{\omega}{}}$, and in particular $\opm{\omega}{}(y^*x)=0$.
	Self-adjointness of $\opm{\omega}{}$ now implies $\opm{\omega}{}(x^*y)=0$ as well.
\end{proof}

The nullspaces now give us the following characterization of the almost sure equalities, extending results of Parzygnat~\cite[Theorem~5.12]{quantum-markov}.

\begin{theorem}\label{thm:ase_nullspace}
	Let us consider a morphism $\omega\colon A \to B$ in $\cpu_{\min/\max}$ and two generalized morphisms $\phi,\psi\colon B\to C$. Then:
	\[
	\begin{array}{rcll}
		\phi \asel{\omega} \psi &\hspace{0.5ex}\Longleftrightarrow\hspace{0.5ex}& \opm{\phi}{}(y)-\opm{\psi}{}(y) \in {}_{\opm{\omega}{}} N &\quad \forall y \in \pre{C}, \\
		\phi \aser{\omega} \psi  &\hspace{0.5ex}\Longleftrightarrow\hspace{0.5ex}& \opm{\phi}{}(y)-\opm{\psi}{}(y) \in N_{\opm{\omega}{}} &\quad \forall y \in \pre{C}, \\[3pt]
		\phi \asets{\omega} \psi &\hspace{0.5ex}\Longleftrightarrow\hspace{0.5ex}& \opm{\phi}{}(y)-\opm{\psi}{}(y) \in SN_{\opm{\omega}{}} &\quad \forall y \in \pre{C}.
	\end{array}
	\]
\end{theorem}
\begin{proof}
	The first two characterizations are already covered in~\cite[Theorem~5.12]{quantum-markov}.\footnote{Although that reference is concerned with finite-dimensional C*-algebras only, the same proof works in general.}
	The third one uses a similar argument as follows.
	We first note that $\phi \asets{\omega} \psi$ can be spelled out explicitly as
	\[
	\opm{\omega}{} (z_1 \opm{\phi}{}(y) z_2) = \opm{\omega}{} (z_1 \opm{\psi}{}(y) z_2)	
	\]
	for all $z_1,z_2 \in \pre{B}$ and $y \in \pre{C}$, or equivalently $\opm{\omega}{}(z_1 (\opm{\phi}{}(y)-\opm{\psi}{}(y)) z_2)=0$. By choosing
	\[
		z_1 \coloneqq z_2^*(\opm{\phi}{}(y)-\opm{\psi}{}(y))^*,
	\]
	we conclude $\opm{\phi}{}(y)-\opm{\psi}{}(y)\in SN_{\opm{\omega}{}}$. 
	Conversely, if we set $x\coloneqq \opm{\phi}{}(y) - \opm{\psi}{}(y) \in SN_{\opm{\omega}{}}$, then $\opm{\omega}{}((xz_2)^*xz_2)=0$ for every $z_2\in \pre{B}$, and \cref{lem:multiplication_nullspace} implies that $\opm{\omega}{}(z_1 x z_2) = 0$ for any choice of $z_1\in \pre{B}$.
\end{proof}

\begin{remark}\label{rem:simple}
	Consider a \emph{simple} pre-C*-algebra $\pre{B}$, i.e.~such that the only closed two-sided ideals are $\lbrace 0 \rbrace$ and $\pre{B}$ itself. 
	Then for every completely positive unital map $\opm{\omega}{} \colon \pre{B} \rightsquigarrow \pre{A}$, the symmetric nullspace $SN_{\opm{\omega}{}}$ must be trivial, since it is a closed two-sided ideal not containing $1$.
	By \cref{thm:ase_nullspace}, for every generalized $\phi, \psi \colon B \to C$, it follows that
	\[
		\phi \asets{\omega} \psi \quad \Longleftrightarrow \quad \phi = \psi
	\]
	whenever $\pre{B}$ is simple. 
\end{remark}

\begin{example}[Almost sure equalities are not necessarily symmetric]\label{ex:ase_not_sym}
	In general, $\phi \ase{\omega} \psi$ does not imply $\phi \asets{\omega} \psi$. The following counterexample is a special case of~\cite[Remark~2.75]{parzygnat2023disintegration}.
	For the sake of simplicity, we just consider the case of the morphisms defined below, where all omitted matrix entries are zero.
	\begingroup
	\[
		\setlength{\tabcolsep}{2pt}
		\begin{array}{rclcrclcrcl}
			\opm{\omega}{} \colon M_4 & \rightsquigarrow & \mathbb{C}&\quad& \opm{\phi}{} \colon M_2 & \rightsquigarrow & M_4 &\quad & \opm{\psi}{} \colon M_2 & \rightsquigarrow &M_4\\
			x&\mapsto & \operatorname{tr}\left(\left(\begin{smallmatrix}
				1/2 & & & \\
				    & 1/2 & & \\
				    & & 0 & \\
				    & & & 0
			\end{smallmatrix}\right) \cdot x \right) &&
			x&\mapsto & \left(\begin{smallmatrix}
				x && \\
				  & \operatorname{tr}(x)/2 & \\
				  && \operatorname{tr}(x)/2
			\end{smallmatrix}\right) &&x&\mapsto & \left(\begin{smallmatrix}
				x & \\
				  & x
			\end{smallmatrix}\right)
		\end{array}
	\]
	\endgroup
	These are completely positive unital maps such that $\opm{\phi}{}(y)-\opm{\psi}{}(y)\in N_{\opm{\omega}{}}$, which gives $\phi \aser{\omega} \psi$.
	Since all three morphisms are self-adjoint, we also obtain $\phi \asel{\omega} \psi$ from \cref{lem:asleftright}, and therefore $\phi \ase{\omega} \psi$.
	Since $M_n$ is simple for every $n$, by \cref{rem:simple} we can infer $\phi \not\asets{\omega} \psi$ from $\opm{\phi}{} \ne \opm{\psi}{}$. 
\end{example}

\begin{lemma}\label{lem:synthetic_supp}
	Let us consider $\cpu_{\min / \max}$ and a morphism $\omega\colon A \to B$. Then the following are equivalent:
	\begin{enumerate}
		\item\label{it:asel_aser} $\asel{\omega}$ coincides with $\aser{\omega}$ on all generalized morphisms out of $B$.
		\item\label{it:nphi_phin} $N_{\opm{\omega}{}}={}_{\opm{\omega}{}}N$.
	\end{enumerate}
\end{lemma}
\begin{proof}
	The implication \ref{it:nphi_phin} $\Rightarrow$ \ref{it:asel_aser} is an immediate application of \cref{thm:ase_nullspace}, so we are left to show the converse.
	
	For the other direction, by \Cref{thm:ase_nullspace} it is enough to show that every $x \in \pre{B}$ can be written as $x = \opm{\phi}{}(y) - \opm{\psi}{}(y)$ for some $y \in \pre{C}$ for suitable generalized morphisms $\phi, \psi \colon B \to C$ and suitable $C$.
	Taking $C = \mathbb{C}^{\{0,1\}}$ and $y = e_1\coloneqq (0,1)$, this is clear from \Cref{ex:observables}, which lets us achieve $\opm{\phi}{}(e_1) = 0$ and $\opm{\psi}{}(e_1) = x$.
\end{proof}

\begin{remark}
	In the proof of \cref{lem:synthetic_supp}, $\phi$ and $\psi$ are generally not self-adjoint because of \cref{lem:asleftright}.
\end{remark}

\begin{remark}
	By looking at \cref{lem:synthetic_supp} and \cref{lem:asel_asets}, one may wonder whether $N_{\opm{\omega}{}} = {}_{\opm{\omega}{}}N$ holds if and only if $N_{\opm{\omega}{}} = SN_{\opm{\omega}{}}$.
	This is indeed the case, because $N_{\opm{\omega}{}} = {}_{\opm{\omega}{}}N$ implies that $N_{\opm{\omega}{}}$ is a closed two-sided ideal, and hence we have $N_{\opm{\omega}{}} \subseteq SN_{\opm{\omega}{}}$ by the maximality statement of \Cref{lem:SN_ideal}.
	On the other hand, $SN_{\opm{\omega}{}} \subseteq N_{\opm{\omega}{}}$ holds trivially.
\end{remark}

\subsection{Kolmogorov products}\label{sec:kolmogorov}

In classical Markov categories, the notion of \emph{Kolmogorov product} axiomatizes the idea of taking infinite tensor products of objects~\cite{fritzrischel2019zeroone}. 
This is relevant when talking about joint distributions of infinitely many random variables. 
We now generalize this notion to the involutive setting in terms of pictures.

\begin{definition}\label{def:kolmogorov}
	Let $\cD$ be a picture.
	For any set $J$, a \newterm{Kolmogorov product} of a family of objects $(A_j)_{j \in J}$ in $\cD$ is an object $A_J$ together with a natural bijection between
	\begin{enumerate}
		\item generalized morphisms $B\to A_J \tensor E$, and
		\item families of generalized morphisms $(\phi_F \colon B \to A_F \tensor E)$, where $A_F$ is shorthand for $\bigtensor_{j \in F} A_j$, indexed by finite subsets $F \subseteq J$, such that whenever $F' \subseteq F$, we have
		\begin{equation}\label{eq:kolmogorov}
			\begin{tikzpicture}
	\begin{pgfonlayer}{nodelayer}
		\node [style=none] (27) at (0, 0) {$=$};
		\node [style=none] (28) at (-2.75, -2.25) {$B$};
		\node [style=morphism] (36) at (2, 1.25) {$\pi_{F,F'}$};
		\node [style=none] (38) at (-2.75, -1.75) {};
		\node [style=generic morphism] (59) at (-2.75, 0) {$\quad \phi_{F'} \quad$};
		\node [style=none] (60) at (-3.5, 0) {};
		\node [style=none] (61) at (-2, 0) {};
		\node [style=none] (62) at (-3.5, 2.25) {};
		\node [style=none] (63) at (-2, 2.25) {};
		\node [style=none] (64) at (2.75, -2.25) {$B$};
		\node [style=none] (66) at (2.75, -1.75) {};
		\node [style=generic morphism] (67) at (2.75, -0.25) {$\quad \phi_{F} \quad$};
		\node [style=none] (68) at (2, -0.25) {};
		\node [style=none] (69) at (3.5, -0.25) {};
		\node [style=none] (70) at (2, 2.25) {};
		\node [style=none] (71) at (3.5, 2.25) {};
		\node [style=none] (72) at (-3.5, 2.75) {$A_{F'}$};
		\node [style=none] (73) at (2, 2.75) {$A_{F'}$};
		\node [style=none] (74) at (-2, 2.75) {$E$};
		\node [style=none] (75) at (3.5, 2.75) {$E$};
	\end{pgfonlayer}
	\begin{pgfonlayer}{edgelayer}
		\draw (62.center) to (60.center);
		\draw (63.center) to (61.center);
		\draw (59) to (38.center);
		\draw (70.center) to (68.center);
		\draw (71.center) to (69.center);
		\draw (67) to (66.center);
	\end{pgfonlayer}
\end{tikzpicture}
		\end{equation}
		where the projection $\pi_{F,F'}\colon A_F \to A_{F'}$ is the identity on $A_{F'}$ and the delete morphism in the other components.
	\end{enumerate}
	Moreover, this natural bijection is required to restrict to natural bijections on $\cD$ and $\gen{\cD}_{\det}$.
\end{definition}
It is worth emphasizing that, by definition, a Kolmogorov product is a (cofiltered) limit that is preserved under tensoring. 
In particular, such a product is unique up to a unique isomorphism in $\cD_{\det}\coloneqq \cD \cap \gen{\cD}_{\det}$ because the projections $\pi_{F,F'}$ belong to $\cD_{\det}$ (recall \cref{ex:det_coherence}).

The Kolmogorov product of a constant family $(A)_{j \in J}$ is called \newterm{Kolmogorov power}.

\begin{definition}
	A picture $\cD$ \newterm{has (countable) Kolmogorov products} if every (countable) family of objects has a Kolmogorov product. 
\end{definition}

The notion of Kolmogorov products has already appeared implicitly in the study of the quantum de Finetti theorem by Staton and Summers~\cite[Definition~2.23 and Theorem~2.24]{staton2023quantum}.

\begin{remark}\label{rem:display_autocomp_N}
	Let us consider a picture with countable Kolmogorov products and a morphism $\phi\colon A \to B$. Then we can construct a generalized morphism $\phi^{(\N)}\colon A \to B_{\N}$ obtained by the family $(\phi^{(n)})$, where $n$ is any nonnegative integer.

	Then $\phi^{(\N)}$ is a morphism if and only if $\phi$ is autocompatible. 
	This is immediate from \cref{rem:display_autocomp}.
\end{remark}

\begin{proposition}\label{prop:kolmogorovproducts}
The pictures $\cpu_{\min/\max}$ have Kolmogorov products.
\end{proposition}

\begin{proof}
	We use the same notation as in \cref{def:kolmogorov} above. 
	If some $A_j$ is the zero algebra, then it is easy to see that the zero algebra also serves as a Kolmogorov product $A_J$.
	We can therefore assume that all $A_j$ are non-zero.

	Then we need to consider the diagram formed by the finite tensor products $A_F \coloneqq \bigodot_{j \in F} A_j$ of pre-C*-algebras, for any finite subset $F \subseteq J$, and for any $F' \subseteq F$, the canonical morphisms $\pi_{F,F'} \colon A_{F} \to A_{F'}$ corresponding to the inclusion $*$-homomorphisms
	\[
		\opm{\pi{}}{}_{F,F'}\colon \pre{A}_{F'} \rightsquigarrow \pre{A}_F.
	\]
	Our assumption that all $A_j$ are non-zero implies that these $*$-homomorphisms are split isometries.
	In one direction, the $\opm{\pi}{}_{F,F'}$ are clearly contractive since they are obtained by tensoring with units in the relevant C*-tensor products.
	In the other direction, choose any state on $A_j$ for every $j\in F\setminus F'$.
	Tensoring these states with the identities on the tensor factors in $F'$ gives a unital completely positive retraction $A_F\rightsquigarrow A_{F'}$ of $\opm{\pi}{}_{F,F'}$.
	By \cref{prop:cpu_tensor,prop:pu_bounded}, this retraction is contractive for either the minimal or the maximal tensor norm.
	Hence $\opm{\pi}{}_{F,F'}$ is a split isometry.

	In $\calgone$, the filtered colimit of the diagram formed by the $\opm{\pi}{}_{F,F'}$ can be constructed as the colimit in the category of sets.
	This is
	\[
		\pre{A_J}\coloneqq \colim_{F} \pre{A_F} = \newfaktor{\left(\displaystyle\bigsqcup_F \pre{A_F}\right)\,}{\sim}
	\] 
	where, given $x_F \in \pre{A_F}$ and $x_{F'} \in \pre{A_{F'}}$, we have $x_F \sim x_{F'}$ if and only if there exists $G \supseteq F,F'$ such that $\opm{\pi{}}{}_{G,F}(x_F) = \opm{\pi{}}{}_{G,F'} (x_{F'})$. 
	The induced algebraic structure clearly makes $\pre{A_J}$ the colimit both in the category of $*$-algebras and in the category of vector spaces.
	Since all $\opm{\pi}{}_{F,F'}$ are isometries, the colimit also inherits a C*-norm from the $\pre{A_F}$.
	(As a matter of fact, the completion $\closed{A_J}$ is the colimit of the $\closed{A_F}$ in the category of C*-algebras, since this is exactly how filtered colimits of C*-algebras can be constructed~\cite[Section~6.1]{murphy1990}.\footnote{See also the more general~\cite[Corollary~7.22]{bunke2021calgccat}, or the historical reference~\cite[Chapter~2]{guichardet66produitstensoriels}.}) 

	So the desired bijection holds with $E = \mathbb{C}$ both at the level of generalized morphisms and at the level of generalized deterministic morphisms (which are the $*$-homomorphisms, \cref{ex:homnotbound}).
	To check that the colimit is preserved by $- \odot E$ for any pre-C*-algebra $E$, recall that $\colim_F (A_F\odot E) \cong \left(\colim_F A_F\right) \odot E$ in the category of vector spaces since $-\odot E$ is a left adjoint endofunctor (\cite[Application 2.6.2]{weibel1994homological}). 
	By construction, this isomorphism is compatible with the induced multiplication and involution on both sides. 
	In other words, it is a $*$-isomorphism, and in particular it induces a bijection between $*$-homomorphisms out of either side. 

	This $*$-isomorphism is also isometric.
	Indeed for any $x \in \colim_F (A_F\odot E)$, choose a finite subset $F \subseteq J$ such that $x$ has a representative $x_F \in A_F \odot E$.
	Then
	\begin{equation}\label{eq:norm_colim}
	\norm{x}_{\colim_F (A_F\odot E)} = \norm{x_F}_{A_F \odot E},	
	\end{equation}
	because the connecting maps $A_{F'}\odot E\rightsquigarrow A_F\odot E$ for $F'\subseteq F$ are isometries by the splitness of $\opm{\pi}{}_{F,F'}$ noted above, applied with $E$ as an extra tensor factor.
	Again by splitness, also the inclusion maps $A_F\odot E\rightsquigarrow A_J \odot E$ are isometries, so that the norm of $x$ in $(\colim_F A_F) \odot E$ coincides with~\eqref{eq:norm_colim}.

	To show that we also obtain bijections at the level of pictures, let us consider a completely positive unital map $\pre{A_J} \odot \pre{E} \rightsquigarrow \pre{B}$. 
	Then the composite
	\[
		\pre{A_F} \odot \pre{E} \rightsquigarrow \pre{A_J} \odot \pre{E} \rightsquigarrow \pre{B}
	\]
	is a completely positive unital map for any $F$ because it is a composition of two such morphisms. 
	Conversely, let us assume that all $\opm{\phi}{}_F \colon \pre{A_F} \odot \pre{E} \rightsquigarrow \pre{B}$ are completely positive unital maps (for $F \subseteq J$ finite).
	Then the induced $\opm{\phi}{} \colon \pre{A_J} \odot \pre{E} \rightsquigarrow \pre{B}$ has norm $\le 1$: given any $x \in \pre{A}_J \odot \pre{E}$ with representative $x_F \in \pre{A}_F \odot \pre{E}$, we have
	\[
		\norm{\opm{\phi}{}(x)} =\norm{\opm{\phi_F}{}(x_F)}= \norm{\opm{\phi_F}{} (x_F)} \le \norm{x_F} = \norm{x},
	\]
	By \cref{prop:pu_bounded}, $\opm{\phi}{}$ is positive. Since $B$ and $E$ are arbitrary, also $\opm{\phi}{} \odot \id_C$ is positive for any $C$ because $\opm{\phi}{}_F \odot \id_{C}$ is. Therefore, $\opm{\phi}{}$ is completely positive.
\end{proof}

\begin{notation}\label{nota:sigma}
	Let us consider a picture $\cD$ and an object $A$.
	Suppose that $J$ and $J'$ are sets for which the Kolmogorov powers $A_J$ and $A_{J'}$ exist, and let $\iota\colon J\hookrightarrow J'$ be an injection.
	This induces a morphism
	\[
		A_\iota \: : \: A_{J'}\to A_J
	\]
	which keeps the factors of $A_{J'}$ indexed by $\im(\iota)$, relabels the factor $\iota(j)$ as the $j$-th factor, and marginalizes all factors indexed by $J'\setminus\im(\iota)$.
	Formally, $A_\iota$ is the unique morphism such that for every finite $F\subseteq J$,
	\[
		\pi_F\comp A_\iota = A_{\iota|_F}\comp \pi_{\iota(F)}.
	\]
	Here $\pi_F\colon A_J\to A_F$ and similarly $\pi_{\iota(F)}$ are the finite marginalizations,
	while $A_{\iota|_F}\colon A_{\iota(F)}\to A_F$ is the isomorphism induced by the bijection $\iota|_F\colon F\to\iota(F)$.
\end{notation}

This construction is contravariant in the sense that $A_{\iota\comp\kappa}=A_\kappa\comp A_\iota$ for composable injections $\iota$ and $\kappa$.
We will mainly apply this construction to obtain $A_\sigma : A_J \to A_J$ for any finite permutation $\sigma\colon J\to J$ of an infinite set $J$.
For an example with non-bijective $\iota$, if $\iota\colon \lbrace 1,2\rbrace \to \lbrace 1,2,3\rbrace$ is given by $\iota(1) = 2$ and $\iota(2) = 1$, then $A_\iota\colon A_{\lbrace 1,2,3\rbrace} \to A_{\lbrace 1,2\rbrace}$ is
\begin{equation*}
	\begin{tikzpicture}
	\begin{pgfonlayer}{nodelayer}
		\node [style=none] (79) at (-8, 0) {};
		\node [style=none] (80) at (-10, 0) {};
		\node [style=none] (81) at (-10, 0) {};
		\node [style=none] (82) at (-8, 2.75) {};
		\node [style=none] (83) at (-10, 2.75) {};
		\node [style=none] (86) at (-8, 0) {};
		\node [style=none] (87) at (-10, 0) {};
		\node [style=none] (88) at (-10, 0) {};
		\node [style=none] (89) at (-8, 0) {};
		\node [style=none] (90) at (-10, 0) {};
		\node [style=none] (91) at (-8, 0) {};
		\node [style=none] (92) at (-10, 0) {};
		\node [style=none] (94) at (-10, 2.75) {};
		\node [style=none] (95) at (-8, 2.75) {};
		\node [style=none] (153) at (-10, 0) {};
		\node [style=none] (154) at (-8, 0) {};
		\node [style=none] (155) at (-6, 0) {};
		\node [style=none] (156) at (-10, 3.25) {$A_1$};
		\node [style=none] (157) at (-8, 3.25) {$A_2$};
		\node [style=none] (158) at (-10, -0.5) {$A_1$};
		\node [style=none] (159) at (-8, -0.5) {$A_2$};
		\node [style=none] (160) at (-6, 1.75) {};
		\node [style=none] (161) at (-6, -0.5) {$A_3$};
		\node [style=bn] (162) at (-6, 1.75) {};
	\end{pgfonlayer}
	\begin{pgfonlayer}{edgelayer}
		\draw [in=90, out=-90, looseness=0.75] (83.center) to (79.center);
		\draw [style=protected, in=-90, out=90, looseness=0.75] (81.center) to (82.center);
		\draw (91.center) to (86.center);
		\draw (92.center) to (90.center);
		\draw (88.center) to (90.center);
		\draw (94.center) to (83.center);
		\draw (95.center) to (82.center);
		\draw (154.center) to (91.center);
		\draw (92.center) to (153.center);
		\draw (162) to (155.center);
	\end{pgfonlayer}
\end{tikzpicture}
\end{equation*}

\begin{remark}[Zero-one laws]\label{rem:01laws}
	Some readers may wonder whether the zero--one laws of~\cite{fritzrischel2019zeroone} generalize to the involutive setting.
	Indeed the abstract Kolmogorov zero--one law~\cite[Theorem~5.3]{fritzrischel2019zeroone} holds with the same proof, as long as one takes care to preserve the order of the tensor factors throughout the argument, and where all morphisms in the statement and proof are allowed to be generalized morphisms.
	At the same time, this vast generality also indicates that Kolmogorov's zero--one law is not a particularly deep or interesting result, which is why we do not elaborate further.

	The abstract Hewitt--Savage zero--one law~\cite[Theorem~5.4]{fritzrischel2019zeroone} is a bit deeper, and Parzygnat already asked for a quantum generalization of it~\cite[Question~8.41]{quantum-markov}. 
	Unfortunately, this zero--one law is essentially vacuously true even in the whole involutive Markov categories $\calgmin$ and $\calgmax$.
	The reason is that for infinite $J$, \emph{every} generalized morphism $\phi \colon A_J \to B$ which satisfies $\phi \comp A_\sigma = \phi$ for all finite permutations $\sigma$ of $J$ already factors through the delete morphism.
	Indeed for such $\phi$, every $\opm{\phi}{}(y) \in \pre{A_J}$ for $y \in \pre{B}$ must be invariant under finite permutations.
	Now the construction of $\pre{A_J}$ implies that $\opm{\phi}{}(y) = [x]$ for some $x \in \pre{A_F}$ and some finite subset $F\subseteq J$. 
	By choosing any finite permutation $\sigma \colon J \to J$ such that $F \cap \sigma(F)= \emptyset$, 
	we conclude that $x = \lambda 1$, so that $\phi$ in fact factors through the delete morphism. 

	One may wonder if such an unfortunate turn of events happens because we are working with pre-C*-algebras instead of C*-algebras. 
	This is not the case as the issue still persists for tensor powers of C*-algebras (\cref{prop:no_01HSlaw}). 
\end{remark}

\section{Representability and the quantum de Finetti theorem}\label{sec:representability_qdf}
In the following, we introduce and investigate representability, as well as its role in quantum de Finetti theorems.
In the classical case, this was defined in~\cite{fritz2023representable} as a way of describing a Markov category as the Kleisli category of an associated monad on its deterministic subcategory.
For involutive Markov categories, we will argue that representability comes in two distinct flavors. 
The more straightforward one is given by extending the classical notion from~\cite[Definitions~3.7 and 3.10]{fritz2023representable} to the involutive world. 
In the pictures that are of interest to quantum probability, the distribution objects (the objects defined by the universal property of representability) are then obtained by the \emph{universal C*-algebras}~\cite{kirchberg} of an operator system, namely of the underlying operator system of the C*-algebra that we start with (see \cref{prop:cpu_representable}).
The second flavour of representability is to require it only against classical objects. 
We call this concept \emph{classical representability}.
It describes the (pre-C*-algebra of complex-valued continuous maps on the) state space of a pre-C*-algebra through the associated universal property (\cref{prop:class_representable}).
For this reason, we call \emph{state space objects} the objects obtained from classical representability.
A first interesting consequence of our synthetic approach is that the state space is Gelfand dual to the abelianization of the universal C*-algebra (\cref{cor:abel_state}).

Our final topic is the quantum de Finetti theorem.
This result, which we prove for both the maximal and the minimal algebraic tensor products of pre-C*-algebras following~\cite{hulanicki68tensor}, is central to our topic because of its inherent connection to classical representability, as it is categorically formulated as a factorization of exchangeable morphisms through the state space object (\cref{thm:QA_SA}). This formulation is rather strong in that it not only applies to exchangeable states, but to exchangeable morphisms with general input and with an extra output as an additional tensor factor on the codomain.

\subsection{Representability}\label{sec:representability}

We start with the definition of representability. The reader may refer to~\cite{fritz2023representable} for representability of classical Markov categories; here we sketch most of the ideas and generalize them to the involutive setting.

\begin{definition}
	Let $\cD$ be a picture. Given an object $A$, a \newterm{distribution object} is an object $PA$ together with a \newterm{sampling morphism} $\samp_A \colon PA \to A$ such that the induced map
	\[
		\samp_A \circ \_ \colon \cD_{\det} (B,PA) \to \cD (B,A)
	\]
	is bijective for every object $B \in \cD$.
	We say that $\cD$ is \newterm{representable} if every object $A$ admits a distribution object $PA$.
\end{definition}

\begin{notation}
	We will denote by $(\_)^{\sharp}$ the inverse of $\samp_A \circ \_$, and $\delta_A \coloneqq (\id_A)^{\sharp}\colon A \to PA$ is also called the \newterm{delta morphism}. We will drop the subscript $A$ if it is sufficiently clear from context.
\end{notation}

Our first observation extends~\cite[Lemma~3.9]{fritz2023representable} to the involutive setting.

\begin{lemma}\label{lem:rep_adjoint}
	Let $\cD$ be a representable picture. Then $A \mapsto PA$ extends to a functor $\cD \to \cD_{\det}$ which is right adjoint to the inclusion $\cD_{\det} \hookrightarrow \cD$ with:
	\begin{enumerate}
		\item Unit given by the delta morphisms;
		\item Counit given by the sampling morphisms.
	\end{enumerate}
\end{lemma}
\begin{proof}
	Representability is exactly the requirement of $\cD_{\det} \hookrightarrow \cD$ being left adjoint, using the definition via universal morphisms with $P$ as the right adjoint.
\end{proof}

\begin{notation}
	In view of \cref{lem:rep_adjoint}, we write $P\phi$ to denote the morphism $PA \to PB$ obtained from $\phi \colon A \to B$ via the functor $P$.
\end{notation}

\begin{proposition}
	Let $\cD$ be a representable picture. Then the right adjoint $P\colon \cD \to \cD_{\det}$ has a canonical symmetric lax monoidal structure given by
	\begin{equation}\label{eq:nabla}
		\begin{tikzcd}[column sep=15ex]
			\nabla_{A,B} \colon \quad PA \tensor PB \ar[r,"\delta"] & P(PA \tensor PB) \ar[r,"P(\samp \tensor \samp)"] & P(A \tensor B)
		\end{tikzcd}
	\end{equation}
	and the adjunction between $P$ and the inclusion $\operatorname{incl}\colon \cD_{\det} \to \cD$ is symmetric monoidal.
\end{proposition}
\begin{proof}
	This proof follows verbatim the argument of~\cite[Proposition~3.15]{fritz2023representable}.
\end{proof}

From this result, we get a corollary analogous to~\cite[Corollary~3.17]{fritz2023representable}.
\begin{corollary}\label{cor:kleisli}
	Let $\cD$ be a representable picture. 
	The adjunction given by $P\colon \cD \to \cD_{\det}$ and the inclusion $\operatorname{incl}\colon \cD_{\det} \hookrightarrow \cD$ induces a symmetric monoidal and affine\footnote{A monad on a category with terminal object $I$ is \emph{affine} if $P(I)\cong I$.} 
	monad $P \operatorname{incl}$ on $\cD_{\det}$.  
	In particular, $\cD$ and the Kleisli category $K\ell(P\operatorname{incl})$ are isomorphic as symmetric monoidal categories. 
\end{corollary}

We can now construct this monad for our two pictures involving completely positive maps.

\begin{proposition}\label{prop:cpu_representable}
	The pictures $\cpu_{\min/\max}$ are representable.
\end{proposition}

\begin{proof}
	In the context of C*-algebras and completely positive unital maps, the universal property of $P$ translates to that of \emph{universal C*-algebras}, where the C*-algebra is seen as an operator system~\cite[Section~3]{kirchberg}. 
	Let us sketch the construction of the universal C*-algebra $\closed{PA}$ of $\closed{A}$ described in~\cite[before Lemma~2.5]{kavruk13quotients}. 
	We consider first the free unital $*$-algebra 
	\[
		F\closed{A} \coloneqq \bigoplus_{n\ge 0} {\closed{A}^{\odot n}}
	\] 
	with multiplication given by tensoring and star given by reversing the order of every elementary tensor $x_1 \odot \cdots \odot x_n$, applying the star on $\closed{A}$ to each element, and extending antilinearly.
	For every completely positive unital map $\opm{\phi}{} \colon \closed{A} \rightsquigarrow \closed{B}$, we now obtain a $*$-homomorphism $\opm{\widetilde{\phi}}{} \colon F\closed{A} \rightsquigarrow \closed{B}$ by setting 
	\[
		\opm{\widetilde{\phi}}{} (x_1 \odot \cdots \odot x_n) \coloneqq \opm{\phi}{}(x_1) \cdots \opm{\phi}{}(x_n)
	\] 
	and extending linearly.
	This induces a seminorm on $F\closed{A}$ given by
	\[
		\norm{ x }_{\closed{PA}} \coloneqq \sup_{\phi} \, \norm{ \opm{\widetilde{\phi}}{} (x) },
	\]
	where $\phi$ ranges over completely positive unital maps $\closed{A} \rightsquigarrow \closed{B}$ for any $\closed{B}$.
	Finally, we take $\closed{PA}$ to be the completion of $F\closed{A}$ after quotienting by the nullspace,
	and this turns out to be a C*-algebra (loc.~cit.).
	The map $\opm{\phi}{\sharp} \colon \closed{PA} \rightsquigarrow \closed{B}$ is then the one induced by $\opm{\widetilde{\phi}}{}$.
	The universal property of the universal C*-algebra is now encoded in the completely positive unital inclusion $\closed{A} \rightsquigarrow \closed{PA}$, in the sense that it is the initial completely positive unital map from $\closed{A}$ to a C*-algebra.

	To extend this reasoning to pre-C*-algebras, we define $PA$ as the dense $*$-subalgebra of $\closed{P{A}}$ generated by the image of $\pre{A}$ via $A \rightsquigarrow \closed{PA}$. 
	By definition, $\pre{PA}$ can equivalently be seen as a quotient of the free unital $*$-algebra $F \pre{A}$ associated to $\pre{A}$ instead of $\closed{A}$.
	Therefore, given any completely positive unital map $\opm{\psi}{}\colon \pre{A} \rightsquigarrow \pre{B}$, the free construction yields a map $\opm{\widetilde{\psi}}{}\colon F\pre{A}\rightsquigarrow \pre{B}$, while its closure gives a completely positive unital map $\close{\opm{\psi}{}}\colon \closed{A} \rightsquigarrow \closed{B}$ (\cref{prop:Phi_closPhi_positivity}), and so the map $\pre{PA} \rightsquigarrow \pre{B}$ is simply the one obtained by restricting $\closed{P{A}}\rightsquigarrow \closed{B}$, since $\widetilde{\close{\opm{{\psi}}{}}}\colon F\closed{A} \rightsquigarrow \closed{B}$ restricts to $\opm{\widetilde{\psi}}{}\colon F\pre{A} \rightsquigarrow \pre{B}$.
\end{proof}

By construction, the sampling morphisms in $\cpu_{\min/\max}$ are simply given by the inclusions $x \mapsto [x]$, while the delta morphisms correspond to the multiplication maps
\[
	[x_1 \odot \dots \odot x_n] \longmapsto x_1 \cdots x_n.
\]
In particular, although the multiplication map is usually not bounded as a map $A \odot A \rightsquigarrow A$ with respect to either the minimal or maximal tensor norm, we now have an equivalent formulation given by a bounded $*$-homomorphism $\opm{\delta_A}{} \colon PA \rightsquigarrow A$.

The laxator $\nabla_{A,B}$ defined in \eqref{eq:nabla} corresponds to the following map on simple tensors:
\[
	[(x_1 \odot y_1) \odot \dots \odot (x_n \odot y_n)] \longmapsto [x_1 \odot \dots \odot x_n] \odot [y_1 \odot \dots \odot y_n],
	\]
where $x_j \in A$ and $y_j \in B$. 

In the classical setting, observational representability was introduced in~\cite[Appendix~A.4]{fritz2023supports} based on the notion of observational monad from~\cite{moss2022probability}.
The idea behind this concept is that iterated sampling can distinguish any two distinct probability measures.
Using \cref{nota:product_power}, we can consider the morphisms $\samp^{(n)}_A \colon PA \to A^{\otimes n}$ for any $n\ge 1$.

\begin{definition}
	A picture $\cD$ is \newterm{observationally representable} if it is representable and for every object $A$, the morphisms $(\samp_A^{(n)})_{n\ge 1}$ are jointly monic in $\gen{\cD}$: for any generalized morphisms $\phi, \psi \colon B \to PA$,
	\[
		\samp_A^{(n)} \phi = \samp_A^{(n)}\psi \quad  \forall\, n \qquad \implies \qquad \phi = \psi.
	\]
\end{definition}

\begin{remark}\label{rem:sampNmonic}
	If we assume that $\cD$ has countable Kolmogorov products, then the family $\samp_A^{(n)}$ glues together to a single generalized morphism $\samp_A^{(\mathbb{N})} \colon PA \to A_{\mathbb{N}}$. 
	The requirement of observational representability is then the same as saying that $\samp_A^{(\mathbb{N})}$ is monic in $\gen{\cD}$, since
	\[
		\samp_A^{(\mathbb{N})} \phi = \samp_A^{(\mathbb{N})} \psi \qquad \iff \qquad \samp_A^{(n)} \phi = \samp_A^{(n)} \psi \quad \forall \, n.
	\]
	Indeed the proof of this equivalence is an immediate application of the universal property of Kolmogorov products.
\end{remark}

\begin{proposition}\label{prop:cpu_obs_representable}
	The pictures $\cpu_{\min/\max}$ are observationally representable.
\end{proposition}
\begin{proof}
	In view of \cref{rem:sampNmonic} and \cref{prop:kolmogorovproducts}, we want to show that $\samp_A^{(\mathbb{N})}$ is monic with respect to generalized morphisms for every $A$.
	This translates into proving that the map $\opm{\samp_A}{(\mathbb{N})} \colon \pre{A_{\mathbb{N}}} \rightsquigarrow \pre{PA}$ is epic with respect to linear maps out of $\pre{PA}$.

	We recall that $\pre{PA}$ is a quotient of the free $*$-algebra $F\pre{A}$, as explained in the proof of \cref{prop:cpu_representable}.
	Therefore, the map $\opm{\samp_A}{(\mathbb{N})}$ sends by construction an element $x_1 \odot \dots \odot x_n \in \pre{A^{\odot n}} \subseteq \pre{A_{\mathbb{N}}}$ to the equivalence class of $x_1 \odot \dots \odot x_n$ in $\pre{PA}$. We conclude that $\opm{\samp_A}{(\mathbb{N})}$ is surjective because every element in $\pre{PA}$ is an equivalence class of finite sums of elements of the form $x_1 \odot \dots \odot x_n$ for some $n$.
\end{proof}

\begin{remark}\label{rem:it_samp_not_selfadjoint}
	In $\cpu_{\min/\max}$, the iterated sampling morphism $\samp^{(\N)}_A \colon PA \to A_{\N}$ is not self-adjoint in general.
	
	First, note that for self-adjoint $x, y \in {A}$, the element $x \odot y$ is self-adjoint in $A_\N$ while $[x \odot y]^* = [y \odot x]$ in $PA$. 
	Therefore, $\samp^{(2)}_A$ is self-adjoint if and only if $PA$ is commutative, and arguing likewise for more than two factors shows the statement for $\samp^{(\N)}_A$ as well.

	If $A = 0$, $\mathbb{C}$ or $\mathbb{C}^2$, then $PA$ is commutative: the first two are trivial because $P0=0$ and $P\mathbb{C}=\mathbb{C}$, while the third is discussed in~\cite[Section~5]{kirchberg}. 
	In the same section, Kirchberg and Wassermann prove that $\closed{P(\mathbb{C}^3)}$ and $\closed{P(M_2)}$ are not exact.
	This means in particular that they are not commutative,\footnote{The exactness of a commutative C*-algebra is a standard fact. See e.g.~\cite{brownozawa} for more details (every commutative C*-algebra is nuclear, and every nuclear C*-algebra is exact).} and therefore neither are $P(\mathbb{C}^3)$ and $P(M_2)$.
	Hence in these cases, the generalized morphisms $\samp^{(2)}$ and $\samp^{(\N)}$ are not self-adjoint.
\end{remark} 

\begin{proposition}\label{prop:obs}
	Let $\cD$ be a representable picture. Then the following are equivalent:
	\begin{enumerate}
		\item\label{it:obs_eq} $\cD$ is observationally representable;
		\item \label{it:obs_monic} For all objects $A$ and $B$, the morphisms $(\samp^{(n)}_A \tensor \id_B)_{n \in \N}$ are jointly monic in $\gen{\cD}$.
		\item\label{it:obs_as_eq} For all objects $A$, the morphisms $(\samp^{(n)}_A)_{n \in \N}$ are jointly monic in $\gen{\cD}$ modulo left \as{} equality: 
			given any suitably composable generalized morphisms $\phi,\psi$ and $\omega$, 
			\[
				\samp_A^{(n)} \comp \phi \asel{\omega} \samp_A^{(n)}\comp \psi  \quad  \forall\, n \qquad \implies \qquad \phi \asel{\omega} \psi.
			\]
		\item\label{it:obs_sym_as_eq} For all objects $A$, the morphisms $(\samp^{(n)}_A)_{n}$ are jointly monic in $\gen{\cD}$ modulo symmetric \as{} equality:
			given any suitably composable generalized morphisms $\phi,\psi,\omega$, 
			\[
				\samp_A^{(n)}\comp \phi \asets{\omega} \samp_A^{(n)}\comp \psi  \quad  \forall\, n \qquad \implies \qquad \phi \asets{\omega} \psi.
			\]
		\end{enumerate}
\end{proposition}

By symmetry, we could equivalently formulate condition~\ref{it:obs_as_eq} with respect to right \as{} equality.

\begin{proof}
	The implication \ref{it:obs_eq}$\implies$\ref{it:obs_monic} is shown exactly as in the classical case~\cite[Proposition~A.4.2]{fritz2023supports}.
	Indeed, a direct computation shows that the coherence morphisms $\nabla_{A,B}$ are split monomorphisms exactly as expressed in~\cite[Lemma~A.4.1]{fritz2023supports}:
\begin{equation*}
	\begin{tikzpicture}
	\begin{pgfonlayer}{nodelayer}
		\node [style=none] (92) at (9, 2) {};
		\node [style=none] (91) at (13.5, 2) {};
		\node [style=bn] (26) at (-5.25, 0) {};
		\node [style=none] (28) at (-7, 1.5) {};
		\node [style=none] (29) at (-7, 3.75) {};
		\node [style=none] (30) at (-3.5, 3.75) {};
		\node [style=none] (31) at (-3.5, 1.5) {};
		\node [style=none] (32) at (-5.25, -2) {};
		\node [style=morphism] (33) at (-5.25, -2) {$ \quad \nabla_{A,B}\quad$};
		\node [style=morphism] (34) at (-7, 2) {$P(\pi_2)$};
		\node [style=morphism] (35) at (-3.5, 2) {$P(\pi_1)$};
		\node [style=none] (36) at (-6.25, -3.5) {};
		\node [style=none] (37) at (-4.25, -3.5) {};
		\node [style=none] (38) at (-4.25, -2) {};
		\node [style=none] (39) at (-6.25, -2) {};
		\node [style=none] (40) at (-6.25, -4) {$PA$};
		\node [style=none] (41) at (-4.25, -4) {$PB$};
		\node [style=none] (42) at (-7, 4.25) {$PA$};
		\node [style=none] (43) at (-3.5, 4.25) {$PB$};
		\node [style=bn] (44) at (1.75, -2) {};
		\node [style=none] (46) at (0.5, 0.25) {};
		\node [style=none] (47) at (1, 3.75) {};
		\node [style=none] (48) at (4.5, 3.75) {};
		\node [style=none] (49) at (4, 0.25) {};
		\node [style=morphism] (52) at (1, 2) {$P(\pi_2)$};
		\node [style=morphism] (53) at (4.5, 2) {$P(\pi_1)$};
		\node [style=none] (54) at (1.75, -3.5) {};
		\node [style=none] (55) at (3.75, -3.5) {};
		\node [style=none] (56) at (3.75, -2) {};
		\node [style=none] (57) at (1.75, -2) {};
		\node [style=none] (58) at (1.75, -4) {$PA$};
		\node [style=none] (59) at (3.75, -4) {$PB$};
		\node [style=none] (60) at (1, 4.25) {$PA$};
		\node [style=none] (61) at (4.5, 4.25) {$PB$};
		\node [style=none] (62) at (-1.5, 0) {$=$};
		\node [style=morphism] (63) at (1, 0.25) {$\,\,\,\nabla_{A,B}\,\,\,$};
		\node [style=morphism] (64) at (4.5, 0.25) {$\,\,\,\nabla_{A,B}\,\,\,$};
		\node [style=bn] (65) at (3.75, -2) {};
		\node [style=none] (66) at (1.5, 0.25) {};
		\node [style=none] (67) at (5, 0.25) {};
		\node [style=bn] (68) at (10.25, -2) {};
		\node [style=none] (69) at (9, 0.25) {};
		\node [style=none] (70) at (9.5, 3.75) {};
		\node [style=none] (71) at (13, 3.75) {};
		\node [style=none] (72) at (12.5, 0.25) {};
		\node [style=none] (75) at (10.25, -3.5) {};
		\node [style=none] (76) at (12.25, -3.5) {};
		\node [style=none] (77) at (12.25, -2) {};
		\node [style=none] (78) at (10.25, -2) {};
		\node [style=none] (79) at (10.25, -4) {$PA$};
		\node [style=none] (80) at (12.25, -4) {$PB$};
		\node [style=none] (81) at (9.5, 4.25) {$PA$};
		\node [style=none] (82) at (13, 4.25) {$PB$};
		\node [style=none] (83) at (7, 0) {$=$};
		\node [style=morphism] (84) at (9.5, 2) {$\, \,\,\nabla_{A,I}\,\,\,$};
		\node [style=morphism] (85) at (13, 2) {$\, \,\,\nabla_{I,B}\,\,\,$};
		\node [style=bn] (86) at (12.25, -2) {};
		\node [style=none] (87) at (10, 0.25) {};
		\node [style=none] (88) at (13.5, 0.25) {};
		\node [style=bn] (89) at (10, 0.25) {};
		\node [style=bn] (90) at (12.5, 0.25) {};
		\node [style=none] (100) at (17.5, -3.5) {};
		\node [style=none] (103) at (17.5, 3.75) {};
		\node [style=none] (104) at (17.5, -4) {$PA$};
		\node [style=none] (106) at (17.5, 4.25) {$PA$};
		\node [style=none] (108) at (15.5, 0) {$=$};
		\node [style=none] (114) at (19.5, -3.5) {};
		\node [style=none] (115) at (19.5, 3.75) {};
		\node [style=none] (116) at (19.5, -4) {$PB$};
		\node [style=none] (117) at (19.5, 4.25) {$PB$};
	\end{pgfonlayer}
	\begin{pgfonlayer}{edgelayer}
		\draw (32.center) to (26);
		\draw [in=-90, out=15] (26) to (31.center);
		\draw (31.center) to (30.center);
		\draw (29.center) to (28.center);
		\draw [in=165, out=-90] (28.center) to (26);
		\draw (36.center) to (39.center);
		\draw (37.center) to (38.center);
		\draw [in=165, out=-90] (46.center) to (44);
		\draw (54.center) to (57.center);
		\draw (55.center) to (56.center);
		\draw [in=-90, out=15] (65) to (67.center);
		\draw [in=165, out=-90] (66.center) to (65);
		\draw (47.center) to (63);
		\draw (48.center) to (64);
		\draw [in=165, out=-90] (69.center) to (68);
		\draw (75.center) to (78.center);
		\draw (76.center) to (77.center);
		\draw [in=-90, out=15] (86) to (88.center);
		\draw [in=165, out=-90] (87.center) to (86);
		\draw (70.center) to (84);
		\draw (71.center) to (85);
		\draw (92.center) to (69.center);
		\draw (91.center) to (88.center);
		\draw (100.center) to (103.center);
		\draw (114.center) to (115.center);
		\draw [style=protected, in=-90, out=15] (44) to (49.center);
		\draw [style=protected, in=-90, out=15] (68) to (72.center);
	\end{pgfonlayer}
\end{tikzpicture}
\end{equation*}
	The other implications are simple checks.
\end{proof}

\subsection{Classical representability}\label{sec:classical_representability}

In this subsection, we consider a variation on the theme of representability which will turn out to provide a universal property for the \emph{state space} of a pre-C*-algebra.
This new representability is only required against classical objects, hence the name. 
The associated distribution objects, called state space objects, are realized by the C*-algebras of complex-valued continuous functions on the state spaces.
This fresh perspective on the state space is particularly fruitful.
For example, the universal property implies that the C*-algebra of continuous functions on the state space is the abelianization of the universal C*-algebra, i.e.~of the distribution object (\cref{cor:abel_state}).

\begin{definition}
	Let $\cD$ be a picture.
	\begin{enumerate}
		\item Given an object $A$, a \newterm{state space object} is a \underline{classical} object $SA$, together with an \newterm{evaluation morphism} $\ev_A \colon SA \to A$,
			such that the induced map
			\begin{equation}\label{eq:evaluation}
				\ev_A \circ \_ \colon \cD_{\det} (B, SA)\to \cD (B,A)
			\end{equation}
			is bijective for every \underline{classical} object $B \in \cD$.
		\item We say that $\cD$ is \newterm{classically representable} if every $A$ admits a state space object $SA$.
	\end{enumerate}
\end{definition}

\begin{notation}
	We denote by $(\_)^{\sharpc}$ the inverse of \eqref{eq:evaluation}. 
	In particular, $\phi = \ev \comp \phi^{\sharpc}$ for every morphism $\phi$ with classical domain, and every classical object $B$ will have a right inverse of $\ev_B$ given by $\id^{\sharpc}_B$.

	Also, following \cref{def:classical_subcat} (and \cref{rem:class_pic}), we use the shorthands
	\[
		\classical{\cD}\coloneqq \classical{\gen{\cD}} \cap \cD \qquad \text{and}\qquad \classical{\cD}_{\det} \coloneqq \classical{\cD} \cap \cD_{\det}.
	\]
\end{notation}

Classical representability has an analogue of~\cref{lem:rep_adjoint}.

\begin{lemma}\label{lem:class_rep_adjoint}
	Let $\cD$ be a classically representable picture. Then $A \mapsto SA$ extends to a functor $\cD \to \classical{\cD}_{\det}$ which is right adjoint to the inclusion $\classical{\cD}_{\det} \hookrightarrow \cD$ with:
	\begin{enumerate}
		\item Unit given by the $\id^\sharpc$;
		\item Counit given by the evaluation morphisms $\ev$.
	\end{enumerate}
\end{lemma}

\begin{remark}\label{rem:simplex_algebra}
	The adjunction 
	\[
		\begin{tikzcd}[column sep=large]
			\classical{\cD}_{\det} \ar[r,phantom,"\bot"] \ar[r,bend left,start anchor={north east},end anchor={north west},"\operatorname{incl}"] & \cD \ar[l,bend left,start anchor={south west},end anchor={south east}, "S" pos=0.53]
		\end{tikzcd}
	\]
	induces a monad on $\classical{\cD}_{\det}$, which we also denote by $S$ by abuse of notation.
	In particular, for every $A$, the classical object $SA$ becomes an Eilenberg-Moore $S$-algebra with structure morphism $S(\ev_A)$. This result is immediate from the monad construction.  
	It is interesting to note that this $S$-algebra is typically not free, as we will see in \cref{rem:non_simplices} below for $\cpu_{\min/\max}$.
\end{remark}
\begin{remark}
	Let $\cD$ be a classically representable picture having countable Kolmogorov products. 
	For any object $A$, the state space object $S(A_{\N})$ is the limit of the diagram $\langle S(A_F),S(\pi_{F,F'}) \rangle$ indexed by finite $F \subseteq \N$ ordered under inclusion.
	Indeed, $S$ is a right adjoint by \cref{lem:class_rep_adjoint}, and therefore preserves limits.
	This is a version of~\cite[Theorem~4.1]{staton2023quantum} formulated in terms of pictures.
\end{remark}

\begin{proposition}\label{prop:class_representable}
	The pictures $\cpu_{\min/\max}$ are classically representable.
\end{proposition}
\begin{proof}
	For any pre-C*-algebra, a state is by definition a completely positive unital map $\pre{A} \rightsquigarrow \mathbb{C}$, or equivalently $\closed{A} \rightsquigarrow \mathbb{C}$.
	We write $\Sigma A$ for this set of states and equip it with the weak-* topology as usual, which is the weakest topology making the evaluation maps $\opm{\ev_x}{} \colon \opm{\phi}{} \mapsto \opm{\phi}{}(x)$ continuous for every $x \in A$ (whether $x \in A$ or $x \in \closed{A}$ does not matter because of the uniform limit theorem).
	It is well-known that $\Sigma{A}$ is a compact Hausdorff space.
	We now take
	\[
		\closed{SA} \coloneqq C(\Sigma{A}) = \Set{ f \colon \Sigma{A} \to \mathbb{C} \given f\text{ is continuous}}.
	\]
	This is a commutative C*-algebra, which by construction contains the evaluation $\opm{\ev_x}{}$ for every $x \in A$.
	
	We then define $\pre{SA}$ to be the $*$-subalgebra of $\closed{SA}$ generated by $\opm{\ev_x}{}$ for $x \in \pre{A}$. (As our notation suggests, $\pre{SA}$ is dense in $\closed{SA}$ because the evaluation maps separate the points of $\Sigma{A}$ by the Hahn--Banach theorem, and therefore the density follows by the Stone--Weierstrass theorem).
	Next, we write
	\[
		\opm{\ev_A}{} \colon \pre{A} \rightsquigarrow \pre{S A}
	\]
	for the map sending each $x$ to its evaluation $\opm{\ev_x}{}$. This map is clearly positive unital, hence completely positive by \cref{prop:com_dom_codom}. 
	It is also injective: 
	Indeed, whenever $\opm{\ev_x}{} = \opm{\ev_y}{}$, then $\opm{\omega}{}(x-y)=0$ for all states $\opm{\omega}{}$, so $x=y$ by \cref{cor:zero_states}.

	In order to show that ${SA}$ is the state space object, let us consider a commutative pre-C*-algebra $\pre{B}$.
	Since the C*-algebra $\closed{B}$ is commutative, by Gelfand duality it is $*$-isomorphic to $C(X)$, the C*-algebra of complex-valued continuous map on a compact Hausdorff space $X$. 
	In order to prove the desired bijection, we now consider for any morphism $\phi \colon \closed{B} \to \closed{A}$ the map	
	\[
		\begin{array}{rcl}
			X& \to & \Sigma A\\
			p& \mapsto & (x \mapsto \opm{\phi}{}(x)(p)),
		\end{array}	
	\]
	which is continuous because post-composing with the evaluation maps $\opm{\ev_x}{}$ gives continuous maps.
	By Gelfand duality, we therefore obtain a $*$-homomorphism $\opm{\phi}{\sharpc} \colon \pre{SA} \rightsquigarrow \pre{B}$ sending $\opm{\ev_x}{} \in \pre{SA}$ to $\opm{\phi}{}(x)$ by definition. 
	This satisfies $\ev_A \comp \phi^{\sharpc} = \phi$ by construction.
	Its uniqueness is immediate since this equation fixes $\opm{\phi}{\sharpc}$ on the generators $\ev_x$ and $\opm{\phi}{\sharpc}$ is required to be a $*$-homomorphism.
\end{proof}
\begin{remark}
	\label{rem:non_simplices}
	Let us work out what the induced $S$-algebra structure on state space objects $SA$ for the pictures $\cpu_{\min/\max}$ amounts to.
	As we will show, this is exactly the structure which equips the state spaces $\Sigma A$ with the structure of compact convex sets (cf.~\cite{furber_jacobs_gelfand}, and the explicit formulations~\cite[Theorems~2.18 and 2.19]{staton2023quantum}).
	This is especially important in the foundations of quantum theory, where the convex structure is often considered fundamental due to its simple operational interpretation: a convex combination of states corresponds to a random ``mixing'' of the states that appear in it.\footnote{To see that this was already recognized by von Neumann, note the importance he attributes to discussion of mixtures in his eminent treatise~\cite{vonneumann2018foundations}.}
	So although our formal categorical framework does not have any built-in notion of a convex structure, this comes out automatically from classical representability.

	To begin, note that under Gelfand duality, there is a functor $\classical{\cpu_{\min/\max}}_{\det} \to \chaus$ given by applying Gelfand duality to the completion.
	In particular, if we restrict the objects to C*-algebras for simplicity, then this functor is an equivalence.
	On this subcategory, our monad $S$ corresponds to the \emph{Radon monad} $R : \chaus \to \chaus$, which assigns to every compact Hausdorff space $X$ the space of Radon probability measures $RX$~\cite{furber_jacobs_gelfand}; this is straightforward to see from the proof of \cref{prop:class_representable} by virtue of the natural isomorphism $\Sigma C(X) \cong RX$ given by the Riesz representation theorem.
	By a result of \'Swirszcz~\cite{swirszcz1974convexity}, the $R$-algebras are exactly the compact convex sets in locally convex spaces, where the structure morphism assigns to every Radon probability measure on such a set its barycenter.
	In particular, applying the structure morphism to finitely supported probability measures amounts to the formation of convex combinations. So in brief, the $R$-algebras are the compact convex sets.
	Its free algebras are the \emph{probability simplices}, which are the closed convex hulls of the Dirac deltas on a compact Hausdorff space $X$.\footnote{The precise notion of simplex here is that of \emph{Bauer simplex}. See~\cite[Corollary~II.4.2]{alfsen1971convex} for the theorem which identifies Bauer simplices with the spaces of Radon probability measures on compact Hausdorff spaces.}

	Now that we have clarified the relation between the Radon monad $R$ and our $S$, we can identify the $S$-algebra structure on $S A$ for every object $A$ in $\cpu_{\min/\max}$.
	An inspection of the proof of \cref{prop:class_representable} shows that the structure morphism $S(\ev_A)$ is given by
	\[
		\begin{array}{rcl}
			\Sigma C(\Sigma A) & \rightarrow & \Sigma A \\
			\mu & \mapsto & (a \mapsto \mu(\opm{\ev_a}{})).
		\end{array}
	\]
	Under the isomorphism $\Sigma C(\Sigma A) \cong R(\Sigma A)$, this map becomes
	\[
		\mu \mapsto \left( a \mapsto \int_{\Sigma A} \opm{\ev_a}{}(s) \, \mu(\mathrm{d} s) \right),
	\]
	which is exactly the barycenter formation on $\Sigma A$ considered as a compact convex set in the dual space $A^*$ equipped with the weak-* topology.

	Finally, note that the state space $\Sigma A$ is generally not a simplex, and therefore not a free $R$-algebra.
	For example, $\Sigma M_2$ is famously a three-dimensional Euclidean ball, the \emph{Bloch ball}~\cite{jaeger2007quantum}.
	More generally, the state space is a simplex if and only if $A$ is commutative~\cite[Corollary~II.4.2]{alfsen1971convex}.
	Therefore any non-commutativity manifests itself in the non-uniqueness of the decomposition of a state into extremal states, a central theme in the foundations of quantum theory~\cite[Section~2.4.2]{nielsen2000quantum}.
\end{remark}

\begin{proposition}[Classicalization of the distribution object]\label{prop:abel_dist}
	Assume that $\cD$ is a representable and classically representable picture. Then 
	\[\ev_A^{\sharp} \colon SA \to PA\] 
	induces a bijection
	\[
		\ev^{\sharp}_A \circ \_ \colon \cD_{\det} (B,SA) \to \cD_{\det}(B,PA)
	\]
	for every classical object $B$.
\end{proposition}
\begin{proof}
	Since $\samp_{A} \comp \ev_{A}^{\sharp}= \ev_A$, we can consider the commutative diagram 
	\[
	\begin{tikzcd}[column sep=huge]
		\cD_{\det}(B,SA)\ar[r,"\ev_A\circ \_"]\ar[d,"\ev^{\sharp}_A\circ \_"'] & \cD(B,A)  \\
		\cD_{\det}(B,PA)\ar[ru,"\samp_A\circ \_"']
	\end{tikzcd}
	\]
	which immediately proves the statement because the two arrows going to the right are bijections.
\end{proof}

We recall that the abelianization of a C*-algebra is defined as the quotient by the closed two-sided ideal generated by the commutators~\cite[Definition~2.8]{blackadar1985shape}.
This definition naturally extends to pre-C*-algebras in an analogous manner, leading to the following result.

\begin{corollary}\label{cor:abel_state}
	For every object $A$ in $\cpu_{\min/\max}$, the state space object ${SA}$ is the abelianization of the distribution object ${PA}$.
\end{corollary}
\begin{proof}
	\cref{prop:abel_dist} is exactly the universal property of the abelianization of ${PA}$ restricted to $*$-homomorphisms, so the statement follows.
\end{proof}

\begin{corollary}\label{cor:evsharp_monic}
	For any objects $A$ and $B$ in $\cpu_{\min/\max}$, the morphism $\ev_A^{\sharp}\tensor \id_B$ is monic in $\gen{\cpu_{\min/\max}}$. 
\end{corollary}
\begin{proof}
	By \cref{cor:abel_state}, we know that $\opm{\ev}{\sharp}_A$ is surjective. 
	In particular, it is epic in the category of vector spaces.
	We obtain that $\opm{\ev}{\sharp}_A \odot \id_{{B}}$ is also epic in the category of vector spaces, because tensoring by a vector space is a left adjoint endofunctor (this actually holds for modules over a ring, see \cite[Application 2.6.2]{weibel1994homological}), and left adjoints preserve colimits, in particular epimorphisms.
	The statement follows since epimorphisms are monomorphisms in the opposite category, and monicity is preserved when restricting to smaller categories.
\end{proof}

\begin{proposition}
	Let $\cD$ be a representable and classically representable picture. 
	Then $(\ev_A^{\sharp})_{A \in \cD}$ is a natural transformation $S \to P$, where $S$ and $P$ are both considered as functors $\cD \to \cD_{\det}$.
\end{proposition}
\begin{proof}
	Let $\phi\colon A \to B$ be any morphism in $\cD$. We need to show that 
	\[
	\begin{tikzcd}
		S A \ar[r,"S(\phi)"]\ar[d,"\ev^{\sharp}_A"'] & S B\ar[d,"\ev^{\sharp}_B"] \\
		P A \ar[r,"P(\phi)"] & P B
	\end{tikzcd} 
	\]
	commutes.
	Since all these morphisms are deterministic, it is sufficient to show this after post-composing with $\samp_{B}$. This is a direct computation:
	\[
		\samp_B\comp \ev^{\sharp}_B\comp S(\phi) = \ev_B\comp S (\phi) = \phi \comp \ev_A = \phi \comp \samp_A \comp \ev_A^{\sharp} = \samp_B \comp P(\phi)\comp \ev_A^{\sharp}.
		\qedhere
	\]
\end{proof}
\begin{definition}
	A  picture $\cD$ is \newterm{observationally classically representable} if it is classically representable and for all objects $A$ and $B$, the morphisms $(\ev_A^{(n)}\tensor \id_B)_{n\ge 1}$ are jointly monic in $\gen{\cD}$: For all suitably composable generalized morphisms $\phi$ and $\psi$,
	\[
	(\ev_A^{(n)} \tensor \id_B)\comp\phi = (\ev_A^{(n)} \tensor \id_B)\comp \psi \quad \forall \, n \qquad \implies \qquad \phi= \psi.
	\]
\end{definition}

This definition is motivated by the fact that Items~\ref{it:obs_as_eq} and \ref{it:obs_sym_as_eq} of \cref{prop:obs} now also hold with $\ev^{(n)}$ in place of $\samp^{(n)}$. 
It is not clear to us whether leaving out the extra tensor factor $B$ from the definition would still result in an analogue of \Cref{prop:obs}.

\begin{lemma}\label{lem:obs_obsclass}
	Let $\cD$ be an observationally representable picture, which is also classically representable.
	The following assertions are equivalent:
	\begin{itemize}
		\item The picture $\cD$ is observationally classically representable;
		\item For all objects $A$ and $B$, the morphism $\ev^{\sharp}_A \tensor \id_B$ is monic in $\gen{\cD}$.
	\end{itemize}
\end{lemma}

\begin{proof}
	Since $\ev^{(n)}_A\tensor \id_B = (\samp^{(n)}_A \comp \ev_A^{\sharp}) \tensor \id_B$, by \cref{prop:obs}\ref{it:obs_monic} we obtain that
	\[
		(\ev^{(n)}_A \tensor \id_B)\comp\phi = (\ev^{(n)}_A \tensor \id_B)\comp \psi
	\]
	for all $n$ if and only if $(\ev^{\sharp}_A \tensor \id_B )\comp \phi = (\ev^{\sharp}_A \tensor \id_B )\comp \psi$. The statement follows.
\end{proof}

\begin{proposition}\label{prop:pu_obsclass}
	The  pictures $\cpu_{\min/\max}$ are observationally classically representable.
\end{proposition}
\begin{proof}
	By \cref{cor:evsharp_monic}, $\ev_A^{\sharp}\tensor \id_B$ is monic, so we can apply \cref{lem:obs_obsclass} to conclude the proof for $\cpu_{\min/\max}$. 
\end{proof}

Concerning representability in general, we know that every $\samp_A$ is epic because it has a right inverse given by $\delta_A$. It is therefore natural to ask whether $\ev_A$ is also epic. This is indeed the case, provided we make the additional assumption of local state-separability.
We first generalize the definition from~\cite[Definition~2.2.3(i)]{fritz2023representable}.

\begin{definition}
	A picture $\cD$ is \newterm{locally state-separable} if the following implication holds for all generalized morphisms $\phi,\psi\colon W \tensor A \to B$:
	\begin{equation}\label{eq:locally_state_separable}
		\begin{tikzpicture}
	\begin{pgfonlayer}{nodelayer}
		\node [style=generic morphism] (0) at (-9.5, 0.5) {$\;\; \phi \;\;$};
		\node [style=none] (1) at (-9.5, 1.75) {};
		\node [style=none] (2) at (-9.5, 2.25) {$B$};
		\node [style=none] (3) at (-10, 0.5) {};
		\node [style=none] (4) at (-9, 0.5) {};
		\node [style=none] (5) at (-9, -1.75) {};
		\node [style=none] (6) at (-7, 0) {$=$};
		\node [style=state] (7) at (-10, -1) {$\omega$};
		\node [style=none] (8) at (-9, -2.25) {$A$};
		\node [style=generic morphism] (9) at (-4.5, 0.5) {$\;\; \psi \;\;$};
		\node [style=none] (10) at (-4.5, 1.75) {};
		\node [style=none] (11) at (-4.5, 2.25) {$B$};
		\node [style=none] (12) at (-5, 0.5) {};
		\node [style=none] (13) at (-4, 0.5) {};
		\node [style=none] (14) at (-4, -1.75) {};
		\node [style=state] (15) at (-5, -1) {$\omega$};
		\node [style=none] (16) at (-4, -2.25) {$A$};
		\node [style=none] (17) at (8.5, 0) {$\phi = \psi.$};
		\node [style=none] (18) at (0, 0) {};
		\node [style=none] (19) at (0, 0) {$\forall \; \omega \in \cat{D}(I, W)$};
		\node [style=none] (20) at (4.5, 0) {$\implies$};
	\end{pgfonlayer}
	\begin{pgfonlayer}{edgelayer}
		\draw (5.center) to (4.center);
		\draw (3.center) to (7);
		\draw (1.center) to (0);
		\draw (14.center) to (13.center);
		\draw (12.center) to (15);
		\draw (10.center) to (9);
	\end{pgfonlayer}
\end{tikzpicture}	
	\end{equation}	
\end{definition}
	
\begin{proposition}\label{prop:spu_locstat}
	The pictures $\cpu_{\min/\max}$ are locally state-separable.
\end{proposition}

The proof is immediate from the following result, with notation matching the one of \eqref{eq:locally_state_separable}. 

\begin{lemma}
	Let $A$ and $W$ be pre-C*-algebras. If $x \in {W} \tensor {A}$ satisfies $(\opm{\omega}{} \tensor \id_A) (x)=0$ for every state $\opm{\omega}{}\colon W \rightsquigarrow \mathbb{C}$, then $x=0$.
\end{lemma}
\begin{proof}
	We consider $x \in W \tensor A$ and choose a decomposition $x = \sum_i w_i \tensor a_i$, where the $a_i$ are linearly independent.
	Then
	\[
		0 = (\opm{\omega}{} \tensor \id_A)(x) = \sum_i \opm{\omega}{}(w_i) a_i
	\]
	implies that $\opm{\omega}{}(w_i)=0$ for all $i$.
	It therefore remains to be shown that whenever $\opm{\omega}{}(w_i)=0$ for all states $\opm{\omega}{}$, then $w_i=0$. This is \cref{cor:zero_states}.
\end{proof}

\begin{remark}
	In the classical case, local state-separability was ensured because it is implied by point-separability~\cite[Lemma~2.2.4]{fritz2023supports}.
	Although the proof of this implication is still valid in our setting, 
	this is not of much interest as point-separability is not satisfied for $\cpu_{\min/\max}$,
	because there are not sufficiently many deterministic states. 
	For instance, for $n \ge 2$ there is no multiplicative linear map $M_n \rightsquigarrow \mathbb{C}$, because $M_n$ is a simple noncommutative algebra.
	Thus choosing two distinct parallel unital maps into $M_n$ gives a counterexample to point-separability.
\end{remark}

\begin{lemma}
	Let $\cD$ be a classically representable locally state-separable picture. Then $\ev_A\tensor \id_B$ is epic in $\gen{\cD}$ for all objects $A$ and $B$. 
	Moreover, $S\colon \cD \to \classical{\cD}_{\det}$ is faithful.
\end{lemma}
\begin{proof}
	Let $\phi,\psi\colon A\tensor B \to C$ be two morphisms such that $\phi \comp (\ev_A \tensor \id_B) = \psi \comp (\ev_A\tensor \id_B)$. 
	By local state-separability, $\phi=\psi$ iff $\phi \comp (\omega \tensor \id_B) = \psi \comp (\omega \tensor \id_B)$ for every state $\omega\colon I \to A$. 
	Since $I$ is classical, classical representability shows that $\omega$ factors through $\ev_A$.
	Hence we have 
	\[
	\begin{tikzpicture}
	\begin{pgfonlayer}{nodelayer}
		\node [style=generic morphism] (0) at (-3, 0.5) {$\;\; \phi \;\;$};
		\node [style=none] (1) at (-3, 1.75) {};
		\node [style=none] (2) at (-3, 2.25) {$C$};
		\node [style=none] (3) at (-3.5, 0.5) {};
		\node [style=none] (4) at (-2.5, 0.5) {};
		\node [style=none] (5) at (-2, -3.25) {};
		\node [style=none] (6) at (5.5, 0) {$=$};
		\node [style=state] (7) at (-3.5, -2.25) {$\omega^{\sharpc}$};
		\node [style=none] (8) at (-2, -3.75) {$B$};
		\node [style=generic morphism] (9) at (8, 0.5) {$\;\; \psi \;\;$};
		\node [style=none] (10) at (8, 1.75) {};
		\node [style=none] (11) at (8, 2.25) {$C$};
		\node [style=none] (12) at (7.5, 0.5) {};
		\node [style=none] (13) at (8.5, 0.5) {};
		\node [style=none] (14) at (8.5, -1.75) {};
		\node [style=state] (15) at (7.5, -1) {$\omega$};
		\node [style=none] (16) at (8.5, -2.25) {$B$};
		\node [style=generic morphism] (17) at (-8, 0.5) {$\;\; \phi \;\;$};
		\node [style=none] (18) at (-8, 1.75) {};
		\node [style=none] (19) at (-8, 2.25) {$C$};
		\node [style=none] (20) at (-8.5, 0.5) {};
		\node [style=none] (21) at (-7.5, 0.5) {};
		\node [style=none] (22) at (-7.5, -1.75) {};
		\node [style=none] (23) at (-5.5, 0) {$=$};
		\node [style=state] (24) at (-8.5, -1) {$\omega$};
		\node [style=none] (25) at (-7.5, -2.25) {$B$};
		\node [style=morphism] (26) at (-3.5, -1) {$\ev_A$};
		\node [style=generic morphism] (27) at (2.5, 0.5) {$\;\; \psi \;\;$};
		\node [style=none] (28) at (2.5, 1.75) {};
		\node [style=none] (29) at (2.5, 2.25) {$C$};
		\node [style=none] (30) at (2, 0.5) {};
		\node [style=none] (31) at (3, 0.5) {};
		\node [style=none] (32) at (3.5, -3.25) {};
		\node [style=state] (33) at (2, -2.25) {$\omega^{\sharpc}$};
		\node [style=none] (34) at (3.5, -3.75) {$B$};
		\node [style=none] (35) at (0, 0) {$=$};
		\node [style=morphism] (36) at (2, -1) {$\ev_A$};
	\end{pgfonlayer}
	\begin{pgfonlayer}{edgelayer}
		\draw [in=-90, out=90, looseness=1.50] (5.center) to (4.center);
		\draw (3.center) to (7);
		\draw (1.center) to (0);
		\draw (14.center) to (13.center);
		\draw (12.center) to (15);
		\draw (10.center) to (9);
		\draw (22.center) to (21.center);
		\draw (20.center) to (24);
		\draw (18.center) to (17);
		\draw [in=-90, out=90, looseness=1.50] (32.center) to (31.center);
		\draw (30.center) to (33);
		\draw (28.center) to (27);
	\end{pgfonlayer}
\end{tikzpicture}
	\]
	for all $\omega$, so that $\phi = \psi$. 

	The last sentence is a general fact about adjoint functors: Faithfulness of the right adjoint is equivalent to the counit being epic. 
\end{proof}

\subsection{De Finetti representability}\label{sec:qdf}

In classical probability theory, the de Finetti theorem states that every permutation-invariant distribution of an infinite sequence of random variables is given by a mixture of i.i.d.\ distributions.
The fact that this representation is unique amounts to a universal property characterizing distribution objects.
In this way, the de Finetti theorem is closely related to representability.

In this subsection, we consider these questions in the involutive setting.
The relevant flavor of representability is classical representability, where the role of distribution objects is now played by state space objects.
One theoretical motivation for this is that we would like to factor through an iterated sampling map, and such an operation should be physically implementable, which is not satisfied for general representability (see \cref{rem:it_samp_not_selfadjoint}).
In \cref{sec:qdf_cpu}, we will prove that these properties apply in our pictures of interest $\cpu_{\min/\max}$ by the quantum de Finetti theorem for states proven as \Cref{thm:qdf_states}.

\begin{definition}\label{def:exchangeable}
	Let $\cD$ be a picture with countable Kolmogorov products. A morphism $\phi \colon C \to A_{\N}\tensor B$ is \newterm{exchangeable in the first factor} if, for any finite permutation $\sigma \colon \N \to \N$, we have 
	\begin{equation*}
		\begin{tikzpicture}
	\begin{pgfonlayer}{nodelayer}
		\node [style=morphism] (9) at (-3, -0.75) {$\;\;\phi\;\;$};
		\node [style=none] (16) at (-3.5, 2.5) {$A_{\N}$};
		\node [style=none] (18) at (0, 0) {$=$};
		\node [style=none] (19) at (-3, -2.5) {$C$};
		\node [style=none] (21) at (-3, -2) {};
		\node [style=none] (29) at (-3.5, -0.75) {};
		\node [style=none] (30) at (-2.5, -0.75) {};
		\node [style=none] (31) at (-2.5, 2.5) {$B$};
		\node [style=none] (35) at (-3.5, 2) {};
		\node [style=none] (36) at (-2.5, 2) {};
		\node [style=morphism] (37) at (3, -0.75) {$\;\;\phi\;\;$};
		\node [style=none] (38) at (2.5, 2.5) {$A_{\N}$};
		\node [style=none] (39) at (3, -2.5) {$C$};
		\node [style=none] (40) at (3, -2) {};
		\node [style=none] (41) at (2.5, -0.75) {};
		\node [style=none] (42) at (3.5, -0.75) {};
		\node [style=none] (43) at (3.5, 2.5) {$B$};
		\node [style=none] (44) at (2.5, 2) {};
		\node [style=none] (45) at (3.5, 2) {};
		\node [style=morphism] (46) at (-3.5, 0.75) {$A_{\sigma}$};
	\end{pgfonlayer}
	\begin{pgfonlayer}{edgelayer}
		\draw (9) to (21.center);
		\draw [style=double wire] (35.center) to (29.center);
		\draw (36.center) to (30.center);
		\draw (37) to (40.center);
		\draw [style=double wire] (44.center) to (41.center);
		\draw (45.center) to (42.center);
	\end{pgfonlayer}
\end{tikzpicture}	
	\end{equation*}
	If $B$ is omitted, or equivalently $B=I$, such $\phi$ is called \newterm{exchangeable}.
\end{definition}

As in earlier works, we use the double wire to denote an infinite Kolmogorov product~\cite{fritz2021definetti}, while $A_\sigma$ is as in \cref{nota:sigma}.

\begin{definition}\label{def:dF}
	Let $\cD$ be a picture with countable Kolmogorov products. 
	For an object $A$, a \newterm{de Finetti object} is an object $QA$ together with a morphism $\ell_A \colon QA \to A$ such that:
	\begin{enumerate}
		\item\label{it:df_autocomp} $\ell_A$ is autocompatible;\footnote{In particular, $\ell_A^{(\N)}$ is a morphism in $\cD$ by \cref{rem:display_autocomp_N}.}
		\item\label{it:df_limit} $\ell_A^{(\N)} \colon QA \to A_{\N}$ is the universal morphism making $Q A$ into the equalizer of all finite permutations
			\[
				A_\sigma \colon A_\N \to A_\N
			\]
			in $\cD$, and this limit is preserved by every $\_ \tensor \id_B$ for every object $B$;
		\item \label{it:ell_N_monic} $\ell_A^{(\N)}\tensor \id_B$ is monic in $\gen{\cD}$.
	\end{enumerate}
\end{definition}

In more detail, Property~\ref{it:df_limit} says the following: every morphism $\phi \colon C \to A_{\N}\tensor B$ exchangeable in the first factor admits a decomposition of the form
\begin{equation*}
	\begin{tikzpicture}
	\begin{pgfonlayer}{nodelayer}
		\node [style=none] (0) at (3.25, -1.25) {};
		\node [style=bn] (2) at (3.25, 0) {};
		\node [style=morphism] (4) at (4.5, 1.5) {$\ell_A$};
		\node [style=none] (6) at (4.5, 1.25) {};
		\node [style=none] (7) at (2, 1.25) {};
		\node [style=morphism] (9) at (-2.5, 0) {$\,\,\phi\,\,$};
		\node [style=none] (10) at (4.5, 2.5) {};
		\node [style=none] (14) at (4.5, 3) {$A$};
		\node [style=none] (16) at (-3, 3) {$A_{\N}$};
		\node [style=none] (18) at (0, 0) {$=$};
		\node [style=none] (19) at (-2.5, -2.75) {$C$};
		\node [style=none] (21) at (-2.5, -2.25) {};
		\node [style=morphism] (22) at (2, 1.5) {$\ell_A$};
		\node [style=none] (23) at (2, 1.25) {};
		\node [style=none] (24) at (2, 2.5) {};
		\node [style=none] (25) at (2, 3) {$A$};
		\node [style=none] (26) at (3.25, 1.5) {$\cdots$};
		\node [style=none] (27) at (4, -2.25) {};
		\node [style=morphism] (28) at (4, -1.25) {$\,\,\,\,\mu\,\,\,\,$};
		\node [style=none] (29) at (-3, 0) {};
		\node [style=none] (30) at (-2, 0) {};
		\node [style=none] (31) at (-2, 3) {$B$};
		\node [style=none] (32) at (5.75, 3) {$B$};
		\node [style=none] (33) at (4.75, -1.25) {};
		\node [style=none] (34) at (5.75, 2.5) {};
		\node [style=none] (35) at (-3, 2.5) {};
		\node [style=none] (36) at (-2, 2.5) {};
		\node [style=none] (37) at (4, -2.75) {$C$};
	\end{pgfonlayer}
	\begin{pgfonlayer}{edgelayer}
		\draw [in=-90, out=90, looseness=1.50] (0.center) to (2);
		\draw (6.center) to (4);
		\draw [in=0, out=-90] (6.center) to (2);
		\draw [in=-90, out=-180] (2) to (7.center);
		\draw (10.center) to (4);
		\draw (9) to (21.center);
		\draw (23.center) to (22);
		\draw (24.center) to (22);
		\draw [in=90, out=-90, looseness=1.25] (34.center) to (33.center);
		\draw [style=double wire] (35.center) to (29.center);
		\draw (36.center) to (30.center);
		\draw (28) to (27.center);
	\end{pgfonlayer}
\end{tikzpicture}
\end{equation*}
for a unique morphism $\mu \colon C \to QA \tensor B$.

Following Staton and Summers~\cite[Lemma~4.2 and Theorem~4.3]{staton2023quantum}, we offer an alternative characterization of the second item of \cref{def:dF}.

\begin{proposition}[Alternative characterization of de Finetti objects]\label{prop:alt_char_df}
	Let $\cD$ be a picture with countable Kolmogorov products. For objects $A$ and $B$, 
	the following are in bijective correspondence natural in $C$:
	\begin{enumerate}
		\item\label{it:prop_alt_df} 
		Morphisms $C\to A_{\mathbb{N}}\tensor B$ exchangeable in the first factor (or, in other words, cones $C \to A_{\mathbb{N}}\tensor B$ for the diagram of all finite permutations $A_{\sigma}\tensor \id_B$);
		\item\label{it:prop_alt_ss} 
		Cones $C\to A_F \tensor B$ for the diagram given by
		\[
		A_F \tensor B \xlongrightarrow{A_{\iota}\tensor \id_B} A_{F'}\tensor B
		\]
		for all injections $\iota\colon F' \hookrightarrow F$ between finite sets $F,F'\subseteq \mathbb{N}$.
	\end{enumerate}
	In particular, the equalizer for all finite permutations $A_{\sigma}\tensor \id_B$ coincides with the limit of the diagram in Item~\ref{it:prop_alt_ss}.
\end{proposition}
In this way, one can make sense of the universal property of a de Finetti object without requiring the existence of countable Kolmogorov products.
However, Property~\ref{it:df_limit} of \cref{def:dF} seems more natural from an intuitive point of view, as it fits better with the narrative of exchangeable morphisms.
\begin{proof}
	Concerning \ref{it:prop_alt_ss}$\implies$\ref{it:prop_alt_df}, take a cone $(C\to A_F\tensor B)$ for the diagram of Item~\ref{it:prop_alt_ss}. 
	Since this family commutes with the maps $A_F\tensor B \to A_{F'}\tensor B$ associated to the injections $F' \hookrightarrow F$, this family factors through $C\to A_{\mathbb{N}}\tensor B$ (since Kolmogorov products are preserved under tensoring). 
	Moreover, every finite permutation ${\sigma}$ can be described finitely via injections $F'\hookrightarrow F$, so that $C\to A_{\mathbb{N}}\tensor B$ is in fact an exchangeable morphism in the first factor. 
	
	Conversely, for \ref{it:prop_alt_df}$\implies$\ref{it:prop_alt_ss}, take a morphism $C \to A_{\mathbb{N}}\tensor B$ exchangeable in the first factor and consider $C\to A_{{F}}\tensor B$ given by composing the exchangeable morphism with the marginalization $A_{\mathbb{N}}\tensor B \to A_{F}\tensor B$.
	Every injection $\iota\colon F' \hookrightarrow F$ can be seen as a restriction of a finite permutation $\sigma$ on $\N$.\footnote{To see this, set $\sigma(n)\coloneqq\iota(n)$ if $n\in F'$, extend by any bijection between $\im(\iota)\setminus F'$ and $F'\setminus \im(\iota)$, and fix every element outside $F'\cup \im(\iota)$.}
	Then the resulting family of morphisms $(C\to A_{F}\tensor B)$ is indeed a cone thanks to the following commuting diagram:
	\[
	\begin{tikzcd}
		C\ar[r]\ar[rd] & A_{\mathbb{N}}\tensor B \ar[r]\ar[d,"A_{\sigma}\tensor \id_B"] & A_{F}\tensor B\ar[d,"A_{\iota}\tensor \id_B"] \\
		& A_{\mathbb{N}}\tensor B \ar[r]& A_{F'}\tensor B
		\end{tikzcd}
	\]

	Naturality of the bijection in $C$ is immediate, since the construction from \ref{it:prop_alt_df} to \ref{it:prop_alt_ss} is given by post-composition.
\end{proof}

Here is now our main categorical result.

\begin{theorem}\label{thm:QA_SA}
	Let $\cD$ be a picture with countable Kolmogorov products.
	If $A$ has a de Finetti object $QA$, then this is a state space object with evaluation morphism $\ell_A$.
\end{theorem}

This is somewhat surprising, since we did not explicitly ask anything about the behavior of $QA$ with respect to deterministic morphisms.

\begin{proof} 
	First, let us show that $QA$ is classical.
	Indeed, by \cref{cor:autocomp_powers}, the autocompatibility of $\ell_A$ itself and the definition of Kolmogorov products, $\ell_A^{(\mathbb{N})}$ is autocompatible.
	Since $\ell_A^{(\mathbb{N})} \tensor \id_A$ is monic, the same is true for $\ell_A^{(\mathbb{N})} \tensor \ell_A^{(\mathbb{N})}$. 
	The classicality of $QA$ now follows by an application of \cref{lem:noninv_mono}.
	
	Second, we aim to show that $\ell_A \colon QA \to A$ gives a bijection as expressed in \eqref{eq:evaluation}. 
	So let us fix any classical object $B$.
	Then the injectivity of 
	\[
		\ell_A \circ \_ \colon \cD_{\det} (B, QA)\to \cD (B,A)
	\]
	is quite clear:
	If $\ell_A\comp \phi = \ell_A \comp \psi$, then also $\ell_A^{(\mathbb{N})} \comp \phi = \ell_A^{(\mathbb{N})}\comp \psi$ because $\phi$ and $\psi$ are deterministic, and so $\phi=\psi$ by monicness of $\ell_A^{(\mathbb{N})}$. 
	To prove surjectivity, consider a morphism $\phi \colon B \to A$. 
	The classicality of $B$ implies that $\phi$ is autocompatible (\cref{rem:compatibility_prop}\ref{it:class_comp}), and therefore $\phi^{(\mathbb{N})} \colon B \to A_{\mathbb{N}}$ belongs to $\cD$ as mentioned in \cref{rem:display_autocomp_N}.  
	Moreover, $\phi^{(\mathbb{N})}$ is exchangeable by direct check: for any finite permutation $\sigma$, we have 
	$\pi_F \comp A_{\sigma}\comp \phi^{(\mathbb{N})}=\phi^{(n)}$
	for any finite set $F\subseteq \mathbb{N}$ of size $n$, where $\pi_F\colon A_{\mathbb{N}}\to A_{F}$ denotes a finite marginalization (or, in other words, a cone morphism of the Kolmogorov product).
	Since $F$ was arbitrary, we have $A_{\sigma}\comp \phi^{(\mathbb{N})}=\phi^{(\mathbb{N})}$ for all $\sigma$.
	
	For the sake of brevity, we now omit the subscript $A$ from $\ell_A$.
	By the universal property, we have $\phi^{(\mathbb{N})} = \ell^{(\mathbb{N})} \comp \mu$ for some morphism $\mu \colon B \to QA$ (in $\cD$). 
	By marginalization, we immediately get $\phi = \ell \comp \mu$, and hence
	\begin{equation*}
		\begin{tikzpicture}
	\begin{pgfonlayer}{nodelayer}
		\node [style=none] (0) at (3.5, -2.25) {};
		\node [style=bn] (2) at (3.5, -0.25) {};
		\node [style=morphism] (4) at (5, 1) {$\ell$};
		\node [style=none] (6) at (5, 0.75) {};
		\node [style=none] (7) at (2, 0.75) {};
		\node [style=none] (10) at (5, 2.25) {};
		\node [style=none] (14) at (5, 2.75) {$A$};
		\node [style=none] (18) at (0, 0) {$=$};
		\node [style=morphism] (22) at (2, 1) {$\ell$};
		\node [style=none] (23) at (2, 0.75) {};
		\node [style=none] (24) at (2, 2.25) {};
		\node [style=none] (25) at (2, 2.75) {$A$};
		\node [style=none] (26) at (3.5, 1) {$\cdots$};
		\node [style=none] (27) at (3.5, -2.75) {$B$};
		\node [style=morphism] (28) at (3.5, -1.25) {$\mu$};
		\node [style=none] (29) at (-3.5, -2.25) {};
		\node [style=bn] (30) at (-3.5, -1.25) {};
		\node [style=morphism] (31) at (-2, 1.25) {$\ell$};
		\node [style=none] (32) at (-2, -0.25) {};
		\node [style=none] (33) at (-5, -0.25) {};
		\node [style=none] (34) at (-2, 2.25) {};
		\node [style=none] (35) at (-2, 2.75) {$A$};
		\node [style=morphism] (36) at (-5, 1.25) {$\ell$};
		\node [style=none] (37) at (-5, -0.25) {};
		\node [style=none] (38) at (-5, 2.25) {};
		\node [style=none] (39) at (-5, 2.75) {$A$};
		\node [style=none] (40) at (-3.5, 0) {$\cdots$};
		\node [style=none] (41) at (-3.5, -2.75) {$B$};
		\node [style=morphism] (42) at (-5, 0) {$\mu$};
		\node [style=morphism] (43) at (-2, 0) {$\mu$};
	\end{pgfonlayer}
	\begin{pgfonlayer}{edgelayer}
		\draw (0.center) to (2);
		\draw (6.center) to (4);
		\draw [in=0, out=-90, looseness=0.75] (6.center) to (2);
		\draw [in=-90, out=180, looseness=0.75] (2) to (7.center);
		\draw (10.center) to (4);
		\draw (23.center) to (22);
		\draw (24.center) to (22);
		\draw (29.center) to (30);
		\draw (32.center) to (31);
		\draw [in=0, out=-90, looseness=0.75] (32.center) to (30);
		\draw [in=-90, out=180, looseness=0.75] (30) to (33.center);
		\draw (34.center) to (31);
		\draw (37.center) to (36);
		\draw (38.center) to (36);
	\end{pgfonlayer}
\end{tikzpicture}
	\end{equation*}
	Using this equality, together with the definition of Kolmogorov products and associativity of $\cop_B$, we obtain
	\begin{equation*}
		\begin{tikzpicture}
	\begin{pgfonlayer}{nodelayer}
		\node [style=none] (29) at (8, -3.25) {};
		\node [style=bn] (30) at (8, -1.5) {};
		\node [style=morphism] (31) at (9.5, 2.25) {$\ell$};
		\node [style=none] (32) at (9.5, 0.25) {};
		\node [style=none] (33) at (6.5, 0.25) {};
		\node [style=none] (34) at (9.5, 3.25) {};
		\node [style=none] (35) at (9.5, 3.75) {$A$};
		\node [style=morphism] (36) at (6.5, 2.25) {$\ell$};
		\node [style=none] (37) at (6.5, 0.25) {};
		\node [style=none] (38) at (6.5, 3.25) {};
		\node [style=none] (39) at (6.5, 3.75) {$A$};
		\node [style=none] (40) at (8, 0.5) {$\cdots$};
		\node [style=none] (41) at (8, -3.75) {$B$};
		\node [style=morphism] (42) at (6.5, 0.5) {$\mu$};
		\node [style=morphism] (43) at (9.5, 0.5) {$\mu$};
		\node [style=none] (44) at (1.5, -1) {};
		\node [style=bn] (45) at (1.5, -0.25) {};
		\node [style=morphism] (46) at (2.5, 2.25) {$\ell$};
		\node [style=none] (47) at (2.5, 0.75) {};
		\node [style=none] (48) at (0.5, 0.75) {};
		\node [style=none] (49) at (2.5, 3.25) {};
		\node [style=none] (50) at (2.5, 3.75) {$A$};
		\node [style=morphism] (51) at (0.5, 2.25) {$\ell$};
		\node [style=none] (52) at (0.5, 0.75) {};
		\node [style=none] (53) at (0.5, 3.25) {};
		\node [style=none] (54) at (0.5, 3.75) {$A$};
		\node [style=none] (55) at (1.5, 1) {$\cdots$};
		\node [style=morphism] (57) at (0.5, 1) {$\mu$};
		\node [style=morphism] (58) at (2.5, 1) {$\mu$};
		\node [style=none] (59) at (-0.25, -3.25) {};
		\node [style=bn] (60) at (-0.25, -2.25) {};
		\node [style=none] (61) at (1.5, -1) {};
		\node [style=none] (62) at (-2, -1) {};
		\node [style=none] (63) at (-2, -1) {};
		\node [style=none] (64) at (-0.25, -3.75) {$B$};
		\node [style=none] (65) at (-2, -1) {};
		\node [style=bn] (66) at (-2, -0.25) {};
		\node [style=morphism] (67) at (-1, 2.25) {$\ell$};
		\node [style=none] (68) at (-1, 0.75) {};
		\node [style=none] (69) at (-3, 0.75) {};
		\node [style=none] (70) at (-1, 3.25) {};
		\node [style=none] (71) at (-1, 3.75) {$A$};
		\node [style=morphism] (72) at (-3, 2.25) {$\ell$};
		\node [style=none] (73) at (-3, 0.75) {};
		\node [style=none] (74) at (-3, 3.25) {};
		\node [style=none] (75) at (-3, 3.75) {$A$};
		\node [style=none] (76) at (-2, 1) {$\cdots$};
		\node [style=morphism] (77) at (-3, 1) {$\mu$};
		\node [style=morphism] (78) at (-1, 1) {$\mu$};
		\node [style=none] (79) at (4.5, 0) {$=$};
		\node [style=none] (80) at (-8, -1) {};
		\node [style=bn] (81) at (-8, 1) {};
		\node [style=morphism] (82) at (-7, 2.25) {$\ell$};
		\node [style=none] (83) at (-7, 2) {};
		\node [style=none] (84) at (-9, 2) {};
		\node [style=none] (85) at (-7, 3.25) {};
		\node [style=none] (86) at (-7, 3.75) {$A$};
		\node [style=morphism] (87) at (-9, 2.25) {$\ell$};
		\node [style=none] (88) at (-9, 2) {};
		\node [style=none] (89) at (-9, 3.25) {};
		\node [style=none] (90) at (-9, 3.75) {$A$};
		\node [style=none] (91) at (-8, 2.25) {$\cdots$};
		\node [style=morphism] (92) at (-8, -0.25) {$\mu$};
		\node [style=none] (94) at (-9.75, -3.25) {};
		\node [style=bn] (95) at (-9.75, -2.25) {};
		\node [style=none] (96) at (-8, -1) {};
		\node [style=none] (97) at (-11.5, -1) {};
		\node [style=none] (98) at (-11.5, -1) {};
		\node [style=none] (99) at (-9.75, -3.75) {$B$};
		\node [style=none] (100) at (-11.5, -1) {};
		\node [style=bn] (101) at (-11.5, 1) {};
		\node [style=morphism] (102) at (-10.5, 2.25) {$\ell$};
		\node [style=none] (103) at (-10.5, 2) {};
		\node [style=none] (104) at (-12.5, 2) {};
		\node [style=none] (105) at (-10.5, 3.25) {};
		\node [style=none] (106) at (-10.5, 3.75) {$A$};
		\node [style=morphism] (107) at (-12.5, 2.25) {$\ell$};
		\node [style=none] (108) at (-12.5, 2) {};
		\node [style=none] (109) at (-12.5, 3.25) {};
		\node [style=none] (110) at (-12.5, 3.75) {$A$};
		\node [style=none] (111) at (-11.5, 2.25) {$\cdots$};
		\node [style=morphism] (113) at (-11.5, -0.25) {$\mu$};
		\node [style=none] (114) at (-5, 0) {$=$};
	\end{pgfonlayer}
	\begin{pgfonlayer}{edgelayer}
		\draw (29.center) to (30);
		\draw (32.center) to (31);
		\draw [in=0, out=-90] (32.center) to (30);
		\draw [in=-90, out=180] (30) to (33.center);
		\draw (34.center) to (31);
		\draw (37.center) to (36);
		\draw (38.center) to (36);
		\draw (44.center) to (45);
		\draw (47.center) to (46);
		\draw [in=0, out=-90] (47.center) to (45);
		\draw [in=-90, out=-180] (45) to (48.center);
		\draw (49.center) to (46);
		\draw (52.center) to (51);
		\draw (53.center) to (51);
		\draw (59.center) to (60);
		\draw [in=0, out=-90] (61.center) to (60);
		\draw [in=-90, out=180] (60) to (62.center);
		\draw (65.center) to (66);
		\draw (68.center) to (67);
		\draw [in=0, out=-90] (68.center) to (66);
		\draw [in=-90, out=-180] (66) to (69.center);
		\draw (70.center) to (67);
		\draw (73.center) to (72);
		\draw (74.center) to (72);
		\draw (80.center) to (81);
		\draw (83.center) to (82);
		\draw [in=0, out=-90] (83.center) to (81);
		\draw [in=-90, out=-180] (81) to (84.center);
		\draw (85.center) to (82);
		\draw (88.center) to (87);
		\draw (89.center) to (87);
		\draw (94.center) to (95);
		\draw [in=0, out=-90] (96.center) to (95);
		\draw [in=-90, out=180] (95) to (97.center);
		\draw (100.center) to (101);
		\draw (103.center) to (102);
		\draw [in=0, out=-90] (103.center) to (101);
		\draw [in=-90, out=-180] (101) to (104.center);
		\draw (105.center) to (102);
		\draw (108.center) to (107);
		\draw (109.center) to (107);
	\end{pgfonlayer}
\end{tikzpicture}
	\end{equation*}
	\begin{equation*}
		\begin{tikzpicture}
	\begin{pgfonlayer}{nodelayer}
		\node [style=none] (0) at (-3.5, -3.25) {};
		\node [style=bn] (2) at (-3.5, -0.5) {};
		\node [style=morphism] (4) at (-2, 2) {$\ell$};
		\node [style=none] (6) at (-2, 1.25) {};
		\node [style=none] (7) at (-5, 1.25) {};
		\node [style=none] (10) at (-2, 3.25) {};
		\node [style=none] (14) at (-2, 3.75) {$A$};
		\node [style=none] (18) at (-7, 0) {$=$};
		\node [style=morphism] (22) at (-5, 2) {$\ell$};
		\node [style=none] (23) at (-5, 1.25) {};
		\node [style=none] (24) at (-5, 3.25) {};
		\node [style=none] (25) at (-5, 3.75) {$A$};
		\node [style=none] (26) at (-3.5, 2) {$\cdots$};
		\node [style=none] (27) at (-3.5, -3.75) {$B$};
		\node [style=morphism] (28) at (-3.5, -2) {$\mu$};
		\node [style=none] (115) at (6.5, 0.25) {};
		\node [style=bn] (116) at (6.5, 1) {};
		\node [style=morphism] (117) at (7.5, 2.25) {$\ell$};
		\node [style=none] (118) at (7.5, 2) {};
		\node [style=none] (119) at (5.5, 2) {};
		\node [style=none] (120) at (7.5, 3.25) {};
		\node [style=none] (121) at (7.5, 3.75) {$A$};
		\node [style=morphism] (122) at (5.5, 2.25) {$\ell$};
		\node [style=none] (123) at (5.5, 2) {};
		\node [style=none] (124) at (5.5, 3.25) {};
		\node [style=none] (125) at (5.5, 3.75) {$A$};
		\node [style=none] (126) at (6.5, 2.25) {$\cdots$};
		\node [style=morphism] (127) at (4.75, -2) {$\mu$};
		\node [style=none] (128) at (4.75, -3.25) {};
		\node [style=bn] (129) at (4.75, -1) {};
		\node [style=none] (130) at (6.5, 0.25) {};
		\node [style=none] (131) at (3, 0.25) {};
		\node [style=none] (132) at (3, 0.25) {};
		\node [style=none] (133) at (4.75, -3.75) {$B$};
		\node [style=none] (134) at (3, 0.25) {};
		\node [style=bn] (135) at (3, 1) {};
		\node [style=morphism] (136) at (4, 2.25) {$\ell$};
		\node [style=none] (137) at (4, 2) {};
		\node [style=none] (138) at (2, 2) {};
		\node [style=none] (139) at (4, 3.25) {};
		\node [style=none] (140) at (4, 3.75) {$A$};
		\node [style=morphism] (141) at (2, 2.25) {$\ell$};
		\node [style=none] (142) at (2, 2) {};
		\node [style=none] (143) at (2, 3.25) {};
		\node [style=none] (144) at (2, 3.75) {$A$};
		\node [style=none] (145) at (3, 2.25) {$\cdots$};
		\node [style=none] (147) at (0, 0) {$=$};
	\end{pgfonlayer}
	\begin{pgfonlayer}{edgelayer}
		\draw (0.center) to (2);
		\draw (6.center) to (4);
		\draw [in=0, out=-90] (6.center) to (2);
		\draw [in=-90, out=180] (2) to (7.center);
		\draw (10.center) to (4);
		\draw (23.center) to (22);
		\draw (24.center) to (22);
		\draw (115.center) to (116);
		\draw (118.center) to (117);
		\draw [in=0, out=-90] (118.center) to (116);
		\draw [in=-90, out=-180] (116) to (119.center);
		\draw (120.center) to (117);
		\draw (123.center) to (122);
		\draw (124.center) to (122);
		\draw (128.center) to (129);
		\draw [in=0, out=-90] (130.center) to (129);
		\draw [in=-90, out=180] (129) to (131.center);
		\draw (134.center) to (135);
		\draw (137.center) to (136);
		\draw [in=0, out=-90] (137.center) to (135);
		\draw [in=-90, out=-180] (135) to (138.center);
		\draw (139.center) to (136);
		\draw (142.center) to (141);
		\draw (143.center) to (141);
	\end{pgfonlayer}
\end{tikzpicture}
	\end{equation*}
	Since $\ell^{(\mathbb{N})} \tensor \ell^{(\mathbb{N})}$ is monic, $\mu$ is deterministic ($\mu$ is self-adjoint because all morphisms in a picture are), and surjectivity is ensured.
\end{proof}

In view of this result, de Finetti objects will be denoted by $(SA, \ev)$. 
We also obtain a new notion of representability.

\begin{definition}
Let $\cD$ be a picture. Then $\cD$ is \newterm{de Finetti representable} if it has countable Kolmogorov products and a de Finetti object for every object.
\end{definition}

Due to Property~\ref{it:ell_N_monic} of \cref{def:dF}, an immediate consequence of \cref{thm:QA_SA} is the following.

\begin{corollary}\label{cor:deF_obsclass}
	Any de Finetti representable picture is observationally classically representable.
\end{corollary}

It is worth noting that there is a strong connection between this result and the recent proof of the de Finetti theorem for Markov categories~\cite[Theorem~4.4]{fritz2021definetti}. However, we refrain from discussing this here, as it is beyond the scope of the present paper.

\begin{lemma}\label{lem:autocomp}
	Let $\cD$ be a de Finetti representable picture. For every pair of objects $A$ and $B$, we have a natural bijection
	\[
		\cD_{\det} (B, SA) \cong \Set{\textrm{Autocompatible morphisms }B\to A}
	\]
	given by composition with $\ev_A$.
\end{lemma}

For $\cpu_{\min/\max}$ this result is expected: autocompatible morphisms translate to maps with commutative image (see~\cref{ex:commutingranges}), and therefore they factor through a deterministic epimorphism with classical codomain.
However, de Finetti representability provides an elegant highway independent of such an explicit factorization.

\begin{proof}
	The reasoning is exactly the same as employed in the proof of \cref{thm:QA_SA}.
	Indeed, the fact that $B$ was classical ensured that, given $\phi \colon B\to A$, we could construct an exchangeable $\phi^{(\mathbb{N})}$, and this is exactly the requirement of $\phi$ to be autocompatible.
	The only thing that may be unclear is why $\ev \comp \phi$, with $\phi\colon B \to SA$ deterministic, is autocompatible.
	This follows from the fact that $\phi$ is deterministic and $SA$ is classical, so that $(\ev \comp \phi)^{(\mathbb{N})} = \ev^{(\mathbb{N})} \comp \phi$, and the latter is a composition of two morphisms in $\cD$.
\end{proof}

\subsection{The quantum de Finetti theorem}\label{sec:qdf_cpu}

The arguments of Hulanicki and Phelps on quantum de Finetti theorems~\cite{hulanicki68tensor} are very insightful. 
In fact, they were able to show that the de Finetti theorem holds in a much more general context than that of C*-algebras~\cite[Theorem~4.1]{hulanicki68tensor}.
Their method is completely independent of St{\o}rmer's essentially simultaneous paper~\cite{stormer69exchangeable}, in which the quantum de Finetti theorem for C*-algebras with respect to the minimal tensor product was proven.

In this section, we show how the ideas of Hulanicki and Phelps that enter the proof of their Theorem 4.1 lead to a quantum de Finetti theorem for both the minimal and the maximal tensor products (\cref{thm:qdf_states}). 
In addition, this result allows for one extra tensor factor as in \cref{def:exchangeable}, and matches the style of a known criterion for the separability of states in quantum information theory~\cite[Theorem~1]{doherty04complete}.
This result will then be used to show that the pictures $\cpu_{\min/\max}$ are de Finetti representable (\cref{thm:qdf}).

\begin{lemma}\label{lem:qdf_constantmulti}
	In $\cpu_{\min/\max}$, let $\opm{\phi}{}\colon \pre{A}_{\mathbb{N}} \odot \pre{B} \rightsquigarrow \mathbb{C}$ be a state which is exchangeable in the first factor, and let $a \in \pre{A}$ be positive.
	Then there exist $\lambda \in \R_+$ and a state $\opm{\psi}{} \colon \pre{A}_{\mathbb{N}}\odot \pre{B} \rightsquigarrow \mathbb{C}$ exchangeable in the first factor such that the functional 
	\[
		\opm{\phi}{}_a(y)\coloneqq \opm{\phi}{}(a \odot y), \qquad \forall \, y \in \pre{A}_{\N} \odot \pre{B}
	\]
	is simply given by $\lambda \opm{\psi}{}$.
\end{lemma}

\begin{proof}
	If $\opm{\phi}{}(a \odot 1)=0$, then for every positive $y \in \pre{A}_{\N} \odot \pre{B}$, we have
	\[
		0 \le \opm{\phi}{}(a \odot y) \le \norm{y} \opm{\phi}{}(a \odot 1)= 0
	\]
	by \cref{cor:bound_with_norm}.
	This in particular implies that the statement is trivially true with $\lambda = 0$,
	because $\opm{\phi}{}_a=0$ by the fact that positive $y$ span the algebra.

	If $\opm{\phi}{}(a \odot 1)\ne 0$, then this quantity is positive because $a$ is. Let us define 
	\[
	\opm{\psi}{} (y) \coloneqq \frac{1}{\opm{\phi}{}(a \odot 1)}\opm{\phi}{}_a(y).
	\]
	By construction, $\opm{\psi}{}$ is an exchangeable state satisfying the desired equation, therefore concluding the proof.
\end{proof}

\begin{lemma}\label{lem:ks_odot_inequality}
	In $\cpu_{\min/\max}$, let $\opm{\phi}{}\colon \pre{A}_{\mathbb{N}}\rightsquigarrow \mathbb{C}$ be an exchangeable state, and let $a \in \pre{A}_{\mathbb{N}}$ be self-adjoint. 
	Then 
	\[ 
	\opm{\phi}{}(a\odot a)\ge \opm{\phi}{}(a)^2.
	\]
\end{lemma}
\begin{proof}
	By the description offered in the proof of \cref{prop:kolmogorovproducts}, $a$ has a representative $a_k$ in $\pre{A}_1 \odot \dots \odot \pre{A}_k$ for some $k\in \N$. 
	Let $1_k = 1\odot \dots \odot 1$ the unit tensored with itself $k$-times.
	For $n \ge 2$, we take 
	\[ 
		S_n = a + 1_k\odot a + 1_k\odot 1_k \odot a + \dots + \underbrace{1_k\odot \dots \odot 1_k}_{(n-1)\text{ times}} {}\odot a,
	\] 
	which is self-adjoint because $a$ is.
	By the Kadison--Schwarz inequality from \cref{prop:ks_inequality}, we have $\opm{\phi}{} (S_n^2) \ge \opm{\phi}{} (S_n)^2$. 
	Moreover, the exchangeability of $\opm{\phi}{}$ allows one to write both arguments more concisely,
	\[ 
	\opm{\phi}{}(S_n^2) = n \opm{\phi}{}(a^2)+ n(n-1)\opm{\phi}{} (a\odot a)
	\qquad \text{and}\qquad \opm{\phi}{} (S_n)^2 = n^2 \opm{\phi}{}(a)^2
	\] 
	We now plug these expressions into the inequality $\opm{\phi}{} (S_n^2) \ge \opm{\phi}{} (S_n)^2$ to obtain 
	\[ 
	\opm{\phi}{}(a\odot a)\ge \frac{n}{n-1} \opm{\phi}{}(a)^2 - \frac{1}{n-1}\opm{\phi}{}(a^2)
	\]
	The claim then follows by taking the limit $n\to \infty$.
\end{proof}

\begin{theorem}[Quantum de Finetti, analytic version]
	\label{thm:qdf_states}
	In $\cpu_{\min/\max}$, consider states $\pre{A}_{\mathbb{N}} \odot \pre{B} \rightsquigarrow \mathbb{C}$ that are exchangeable in the first factor.
	Then:
	\begin{enumerate}
		\item\label{it:exc_compact_convex}
			These states form a compact convex set 
			in the weak-* topology.
		\item\label{it:exc_extremal}
			Such a state $\opm{\phi}{} \colon \pre{A}_{\mathbb{N}} \odot \pre{B} \rightsquigarrow \mathbb{C}$ is extremal in this set 
			if and only if it is of the form
			\begin{equation}
				\label{eq:qdf_states}
				\opm{\phi}{} = \opm{\psi}{(\N)} \odot \opm{\omega}{}
			\end{equation}
			for $\opm{\psi}{}\colon \pre{A} \rightsquigarrow \mathbb{C}$ an arbitrary state and $\opm{\omega}{}\colon \pre{B} \rightsquigarrow \mathbb{C}$ a pure state.
		\item\label{it:exc_decomp}
			For every such state $\opm{\phi}{} \colon \pre{A}_{\mathbb{N}} \odot \pre{B} \rightsquigarrow \mathbb{C}$, there is a Radon probability measure $\mu$ on the product of state spaces $\Sigma A \times \Sigma B$ such that
			\[
				\opm{\phi}{} = \int \left[ \opm{\psi}{(\N)} \odot \opm{\omega}{} \right] \,\,  \mu(\mathrm{d} \opm{\psi}{}, \mathrm{d} \opm{\omega}{}).
			\]
	\end{enumerate}
\end{theorem}
\begin{proof}
	Claim~\ref{it:exc_compact_convex} is a direct check based on the fact that permuting the tensor factors of $\pre{A}$ is a weak-* continuous operation on the state space $\Sigma({A_{\mathbb{N}}} \odot {B})$.

	Concerning part~\ref{it:exc_extremal}, suppose that $\opm{\phi}{}$ is extremal.
	Then for any positive $a \in \pre{A}$, we can write $\opm{\phi}{}_a = \lambda \opm{\phi}{}$ by \cref{lem:qdf_constantmulti} and extremality.
	Evaluating at $1$ shows that $\lambda = \opm{\phi}{}(a \odot 1)$. 
	In other words, we have
	\[
		\opm{\phi}{}(a \odot y) = \opm{\phi}{}(a\odot 1) \, \opm{\phi}{}(y)
	\]
	for all positive $a \in \pre{A}$ and $y \in \pre{A_{\N}}\odot \pre{B}$. 
	Applying this equation inductively gives
	\begin{equation*}
	\begin{split}
		\opm{\phi}{} (a_1 \odot \dots \odot a_n \odot b) &= \opm{\phi}{}(a_1 \odot 1) \, \opm{\phi}{} (a_2 \odot \dots \odot a_n \odot b) 	\\
		&\hspace{0.52em} \vdots \\
		& = \opm{\phi}{}(a_1 \odot 1) \cdots \opm{\phi}{}(a_n \odot 1) \, \opm{\phi}{}(1 \odot b)
	\end{split}
	\end{equation*}
	for all positive $a_1, \ldots, a_n \in \pre{A}$ and $b \in \pre{B}$.
	We now set $\opm{\psi}{} (a) \coloneqq \opm{\phi}{}(a \odot 1)$ for all $a \in \pre{A}$ and $\opm{\omega}{}(b) \coloneqq \opm{\phi}{} (1 \odot b)$ for all $b \in \pre{B}$.
	Then we already have~\eqref{eq:qdf_states} for simple tensors consisting of positive elements.
	These span the entire algebraic tensor product (using \cref{cor:bound_with_norm}), and therefore the equation holds for all elements by linearity.
	The fact that $\opm{\omega}{}$ must be pure is straightforward to see.

	Conversely, to show that every state of the form~\eqref{eq:qdf_states} is extremal among states exchangeable in the first factor, we first show that $\opm{\psi}{(\N)}$ is extremal among exchangeable states on $\pre{A}_{\N}$.
	Given a convex decomposition
	\[
		\opm{\psi}{(\N)} = \lambda \opm{\rho}{}_1 + (1-\lambda) \opm{\rho}{}_2
	\]
	with $\lambda \in (0,1)$, we consider arbitrary self-adjoint $a \in \pre{A}_{\N}$ and evaluate
	\begin{align*}
		\opm{\psi}{(\N)}(a)^2 & = \opm{\psi}{(\N)}(a \odot a) \\
			& = \lambda \opm{\rho}{}_1(a \odot a) + (1-\lambda) \opm{\rho}{}_2(a \odot a) \\
			& \ge \lambda \opm{\rho}{}_1(a)^2 + (1-\lambda) \opm{\rho}{}_2(a)^2,
	\end{align*}
	where the inequality step follows from \cref{lem:ks_odot_inequality}.
	By combining this with the expression obtained by expanding the left-hand side directly and using the binomial formula, we obtain the inequality
	\[
		2 \opm{\rho}{}_1(a)\opm{\rho}{}_2(a) \ge \opm{\rho}{}_1(a)^2 + \opm{\rho}{}_2(a)^2,
	\]
	where a factor of $\lambda(1-\lambda)$ has been cancelled by $\lambda \in (0,1)$.
	This inequality can be rearranged to $(\opm{\rho}{}_1(a)-\opm{\rho}{}_2(a))^2\le 0$.
	This implies $\opm{\rho}{}_1(a) = \opm{\rho}{}_2(a)$.
	Since $a$ was arbitrary self-adjoint, we conclude $\opm{\rho}{}_1 = \opm{\rho}{}_2$, so that $\opm{\psi}{(\N)}$ is indeed extremal.

	We now show that $\opm{\phi}{}$ is also extremal. 
	Consider
	\[ 
	\opm{\phi}{} = \lambda \opm{\phi_1}{} + (1-\lambda) \opm{\phi_2}{}
	\]
	for some states $\opm{\phi_i}{}$ exchangeable in the first factor and $\lambda\in [0,1]$. Extremality of $\opm{\psi}{(\N)}$ implies that $\opm{\phi_i}{}(a\odot 1) = \opm{\psi}{(\N)}(a)$ for all $a\in \pre{A}_{\N}$. 
	If $a$ is positive and such that $\opm{\psi}{(\N)}(a)= 0$, then 
	\[
	0\le \opm{\phi_i}{}(a\odot b) \le \opm{\phi_i}{}( a\odot \norm{b} 1) = \norm{b} \opm{\phi_i}{}( a\odot 1) = \norm{b} \opm{\psi}{(\N)}(a)=0
	\] 
	for all positive elements $b$ by \cref{cor:bound_with_norm}, hence by linearity (and \cref{cor:bound_with_norm} again) we get
	\[
		\opm{\phi_i}{}(a\odot b) =0=\opm{\psi}{(\N)}(a) \opm{\omega}{}(b)
	\]
	for all $b\in \pre{B}$.
	Instead, if $a$ is positive such that $\opm{\psi}{(\N)}(a)\neq 0$, 
	from $\opm{\phi}{}(a\odot b) = \opm{\psi}{(\N)}(a) \opm{\omega}{}(b)$ we obtain
	\[
	\opm{\omega}{}(b) =\lambda \frac{\opm{\phi_1}{}(a\odot b)}{\opm{\psi}{(\N)}(a)}+(1-\lambda) \frac{\opm{\phi_2}{}(a\odot b)}{\opm{\psi}{(\N)}(a)}
	\]
	Since the functionals $\frac{\opm{\phi}{}_i(a\odot b)}{\opm{\psi}{(\N)}(a)}$ are clearly positive, and unital by $\opm{\phi}{}_i(a\odot 1) = \opm{\psi}{(\N)}(a)$, extremality of $\opm{\omega}{}$ is sufficient to conclude that these functionals coincide with $\opm{\omega}{}$.
	But then we can conclude $\opm{\phi_i}{}(a\odot b) = \opm{\psi}{(\N)}(a)\opm{\omega}{}(b)$ for all positive $a$ and all $b\in \pre{B}$, and therefore $\opm{\phi}{}$ is extremal since the algebraic tensor product is spanned by these elements (using \cref{cor:bound_with_norm}).
	
	We now turn to part~\ref{it:exc_decomp}.
	Since the inclusion map
	\begin{align*}
		\Sigma A \times \Sigma B &\to \Sigma(A_{\mathbb{N}} \odot B) \\
		(\opm{\psi}{}, \opm{\omega}{}) &\mapsto \opm{\psi}{(\N)} \odot \opm{\omega}{}
	\end{align*}
	is weak-* continuous, its image is compact, and by \ref{it:exc_extremal} it contains all extremal states.
	By applying the Choquet--Bishop--de Leeuw theorem \cite[Section~4]{phelps01choquet}, for every state $\opm{\phi}{}$ exchangeable in the first factor we obtain a Baire probability measure $\nu$ on the set of all states exchangeable in the first factor such that
	\begin{equation}\label{eq:choquet}
	\opm{\phi}{} = \int \opm{\xi}{}\,\, \nu(\mathrm{d} \opm{\xi}{}),
	\end{equation}
	and moreover $\nu(E)=1$ for all Baire sets $E$ containing the extremal exchangeable states.
	It remains to explain how this integral restricts to the Borel set $\Sigma A \times \Sigma B$, which may or may not be Baire.\footnote{Every Baire set is Borel, but the converse need not hold \cite[Section 7.1]{dudley02real}).}
	Since the underlying space is compact Hausdorff, the Baire probability measure $\nu$ can be uniquely extended to a Radon probability measure $\mu$ on the Borel $\sigma$-algebra,
	which on compact sets $K$ satisfies~\cite[Theorem 7.3.1]{dudley02real}
	\[ 
		\mu(K)=\inf \, \Set {\nu(E)\given E \text{ is Baire and } E\supseteq K}.
	\]
	By taking $K$ to be the image of $\Sigma A\times \Sigma B$,
	we find that the infimum on the right-hand side is 1 because any Baire set of the family contains the extremal states. 
	Therefore $\mu(K)=1$, and the integral in \eqref{eq:choquet}, now interpreted with respect to $\mu$ restricted to $\Sigma A \times \Sigma B$, achieves the desired representation.
\end{proof}

\begin{remark}
	In particular, every exchangeable state on an infinite maximal tensor power of a C*-algebra $\closed{A}$ can be written as an integral over product states.
	With an additional tensor factor $B$, this also generalizes the separability criterion of~\cite[Theorem~1]{doherty04complete} to the infinite-dimensional case, even with the maximal tensor product.
\end{remark}

To use the previous result to prove the categorical version, let us start with a technical lemma.

\begin{lemma}
	\label{lem:pu_factor}
	Let $\opm{\phi}{} \colon \pre{A} \rightsquigarrow \pre{A}'$ be any surjective positive unital map.
	Then for any pre-C*-algebra $B$, composing with $\phi$ establishes a bijection between 
	\begin{enumerate}
		\item the set of positive unital maps $\pre{A}' \rightsquigarrow \pre{B}$, and
		\item the set of positive unital maps $\opm{\psi}{} \colon \pre{A} \rightsquigarrow \pre{B}$ such that $\norm{\opm{\psi}{}(x)} \le \norm{\opm{\phi}{}(x)}$ for every self-adjoint element $x$.
	\end{enumerate}
\end{lemma}

\begin{proof}
	Given any positive unital map $\opm{\psi'}{} \colon \pre{A'} \rightsquigarrow \pre{B}$, we clearly have
	\[ 
		\norm{\opm{\psi'}{}\comp \opm{\phi}{}(x)} \le \norm{\opm{\phi}{}(x)}
	\]
	for every $x \in \pre{A}$, so that the composition map under consideration is well-defined.
	The injectivity is obvious because $\opm{\phi}{}$ is assumed surjective.

	For surjectivity, suppose that $\opm{\psi}{} \colon \pre{A} \rightsquigarrow \pre{B}$ satisfies the given inequality.
	Then for any self-adjoint $x \in \pre{A'}$, we can choose any self-adjoint\footnote{Since $x$ has some preimage, it also has a self-adjoint preimage, namely the self-adjoint part of any preimage.} $y \in \pre{A}$ with $\opm{\phi}{}(y) = x$ and define $\opm{\psi'}{}(x) \coloneqq \opm{\psi}{}(y)$.
	This is well-defined since, for any self-adjoint $y\in A$, $\opm{\phi}{}(y)=0$ implies $\opm{\psi}{}(y)=0$ by the assumed inequality.
	This also implies the linearity of $\opm{\psi'}{}$ on the self-adjoint part of $\im(\opm{\phi}{})$.
	Furthermore, the assumed inequality ensures that $\norm{\opm{\psi'}{}(x)} = \norm{\opm{\psi}{}(y)} \le \norm{\opm{\phi}{}(y)} = \norm{x}$,
	where $y$ is any self-adjoint element with $\opm{\phi}{}(y)=x$.
	The complex linear extension $\opm{\psi'}{}$ is a bounded unital map on $\pre{A'}$ by surjectivity and unitality of $\opm{\phi}{}$.
	This map has norm $\le 1$ on self-adjoints, and is therefore positive by \cref{cor:self_contr_positive}.
\end{proof}

\begin{theorem}[Quantum de Finetti, categorical version]\label{thm:qdf}
	The pictures $\cpu_{\min/\max}$ are de Finetti representable. In particular, for any morphism $\phi\colon C \to A_{\N} \tensor B$ exchangeable in the first factor, there exists a unique morphism $\mu \colon C \to SA \tensor B$ such that 
	\begin{equation*}
		\begin{tikzpicture}
	\begin{pgfonlayer}{nodelayer}
		\node [style=none] (0) at (3.5, -1.25) {};
		\node [style=bn] (2) at (3.5, 0) {};
		\node [style=morphism] (4) at (5, 1.25) {$\ev_A$};
		\node [style=none] (6) at (5, 1) {};
		\node [style=none] (7) at (2, 1) {};
		\node [style=morphism] (9) at (-2.5, -0.25) {$\,\,\phi\,\,$};
		\node [style=none] (10) at (5, 2.25) {};
		\node [style=none] (14) at (5, 2.75) {$A$};
		\node [style=none] (16) at (-3, 2.75) {$A_{\N}$};
		\node [style=none] (18) at (0, 0) {$=$};
		\node [style=none] (19) at (-2.5, -2.75) {$C$};
		\node [style=none] (21) at (-2.5, -2.25) {};
		\node [style=morphism] (22) at (2, 1.25) {$\ev_A$};
		\node [style=none] (23) at (2, 1) {};
		\node [style=none] (24) at (2, 2.25) {};
		\node [style=none] (25) at (2, 2.75) {$A$};
		\node [style=none] (26) at (3.5, 1.25) {$\cdots$};
		\node [style=none] (27) at (4.25, -2.25) {};
		\node [style=morphism] (28) at (4.25, -1.25) {$\,\,\,\,\mu\,\,\,\,$};
		\node [style=none] (29) at (-3, -0.25) {};
		\node [style=none] (30) at (-2, -0.25) {};
		\node [style=none] (31) at (-2, 2.75) {$B$};
		\node [style=none] (32) at (6.5, 2.75) {$B$};
		\node [style=none] (33) at (5, -1.25) {};
		\node [style=none] (34) at (6.5, 2.25) {};
		\node [style=none] (35) at (-3, 2.25) {};
		\node [style=none] (36) at (-2, 2.25) {};
		\node [style=none] (37) at (4.25, -2.75) {$C$};
		\node [style=none] (38) at (0, -3.25) {};
	\end{pgfonlayer}
	\begin{pgfonlayer}{edgelayer}
		\draw [in=-90, out=90, looseness=1.50] (0.center) to (2);
		\draw (6.center) to (4);
		\draw [in=0, out=-90] (6.center) to (2);
		\draw [in=-90, out=-180] (2) to (7.center);
		\draw (10.center) to (4);
		\draw (9) to (21.center);
		\draw (23.center) to (22);
		\draw (24.center) to (22);
		\draw [in=90, out=-90, looseness=1.25] (34.center) to (33.center);
		\draw [style=double wire] (35.center) to (29.center);
		\draw (36.center) to (30.center);
		\draw (28) to (27.center);
	\end{pgfonlayer}
\end{tikzpicture}
	\end{equation*} 
\end{theorem}
\begin{proof}
	We first construct a candidate de Finetti object $QA$ for these pictures and then prove that it indeed has the required properties.
	The fact that it is naturally isomorphic to the state space object $SA$, as indicated in the second part of the claim, is then clear by \cref{thm:QA_SA}.

	Let us equip the Kolmogorov power $A_\N$ with the seminorm
	\begin{equation}
		\label{eq:QA_seminorm}
		\norm{x}_{QA} \coloneqq \sup \, \Set{\abs{\opm{\phi}{}(x)} \given \phi \colon \mathbb{C} \to A_\N \textrm{ is an exchangeable state} }.
	\end{equation}
	This is finite since exchangeable states are positive unital, and hence contractive by \cref{prop:pu_bounded}.
	As a supremum of seminorms, it is therefore a seminorm.
	By exchangeability, it is clear that for every finite permutation $\sigma$, the elements of the form $x - \opm{A_{\sigma}}{}(x)$ have seminorm zero.
	We now define $QA$ as the normed vector space given by the quotient of $A_\N$ by the nullspace $\mathcal{N}_{QA}$ of this seminorm.
	In particular, this choice makes the projection map $A_\N \rightsquigarrow QA$ invariant under finite permutations.

	Next, we define a multiplication on $QA$ by ``disjoint tensoring'' as follows.
	For $x, y \in A_\N$, we have $x \tensor y \in A_{\N + \N}$, and we can consider this as an element of $A_\N$ via any bijection $\N + \N \cong \N$, for which we also write $x \tensor y$ by abuse of notation.
	We then define the multiplication on $QA$ as
	\[
		[x] \cdot [y] \coloneqq [x \tensor y],
	\]
	and the finite permutation invariance just noted implies that the result is independent of the choice of bijection, since $x$ and $y$ are supported on finitely many tensor factors.
	For well-definedness of multiplication, we first note that by Bauer's maximum principle~\cite[Theorem~1.4.1]{blanchard92variational} and the analytic quantum de Finetti theorem for states (\cref{thm:qdf_states}), we can equivalently write
	\begin{equation}
		\label{eq:QA_seminorm2}
		\norm{x}_{QA} = \sup \Set {\abs{\opm{\psi}{(\N)}(x)} \given  \psi \colon \mathbb{C} \to A \textrm{ is a state}}.
	\end{equation}
	Since every power state evaluates a disjoint tensor as
	\[
		\opm{\psi}{(\N)}(x\tensor y)=\opm{\psi}{(\N)}(x)\opm{\psi}{(\N)}(y),
	\]
	it follows that if one factor has $QA$-seminorm zero, then so does its disjoint tensor with any other factor.
	Thus multiplication is well-defined on equivalence classes.
	Using this identity together with~\eqref{eq:QA_seminorm2} gives submultiplicativity and the C*-identity.
	The star operation descends because exchangeable states are self-adjoint, and the remaining algebraic identities follow from the corresponding tensor identifications, again modulo finite permutation invariance.
	Therefore $QA$ is a commutative pre-C*-algebra.

	We now take $\opm{\ell_A}{} \colon \pre{A} \rightsquigarrow \pre{A_\N} \rightsquigarrow \pre{QA}$ to be the composition of the inclusion $\pre{A} \rightsquigarrow \pre{A_\N}$ in the first factor and the quotient map $\pre{A_\N} \rightsquigarrow \pre{QA}$ itself.
	Then $\opm{\ell_A}{}$ is unital and contractive, and therefore positive by \cref{prop:pu_bounded}.
	Moreover since $QA$ is commutative, complete positivity follows by \cref{prop:com_dom_codom}, and also $\ell_A$ is trivially autocompatible.
	Thus Property~\ref{it:df_autocomp} of \Cref{def:dF} is satisfied.
	
	The monicness of $\ell_A^{(\N)} \tensor \id_B$ for every $B$ required by Property~\ref{it:ell_N_monic} is clear, since $\opm{\ell_A}{(\N)}$ is precisely the quotient map $\pre{A_\N}\rightsquigarrow \pre{QA}$, and hence also $\opm{(\ell_A^{(\N)} \tensor \id_B)}{}$ maps $\pre{A_\N} \odot \pre{B}$ surjectively onto $\pre{QA} \odot \pre{B}$ (recall that epicness is preserved by algebraic tensoring, as commented in the proof of \cref{cor:evsharp_monic}, and in the category of vector spaces epimorphisms are precisely the surjective linear maps).

	Finally for Property~\ref{it:df_limit}, let us consider the case of positive unital maps first.
	Writing maps in the algebra direction, $\opm{(\ell_A^{(\N)}\tensor \id_B)}{}$ is surjective by the preceding paragraph.
	Thus \cref{lem:pu_factor} reduces the unique positive-unital factorization of a positive unital $\rho \colon C \to A_\N \tensor B$ exchangeable in the first factor to showing
	that for all self-adjoint $x \in A_\N \tensor B$,
	\[
		\norm{\opm{\rho}{}(x)} \le \norm{\opm{(\ell_A^{(\N)} \tensor \id_B)}{}(x)}.
	\]
	Since $\opm{\rho}{}(x)$ is also self-adjoint, we have
	\[
		\norm{\opm{\rho}{}(x)} = \sup \Set{\abs{\opm{\eta}{}(\opm{\rho}{}(x))} \given {\eta \colon \mathbb{C} \to C \textrm{ 
			 is a state}}} 
	\]
	by \cref{cor:norm_state}.
	But since $\rho \comp \eta \colon \mathbb{C} \to A_\N \tensor B$ is a state exchangeable in the first factor, it only remains to be shown that
	\begin{equation}\label{eq:claim}
		\sup_{\phi} \, \abs{\opm{\phi}{}(x)} \le \norm{\opm{(\ell_A^{(\N)} \tensor \id_B)}{}(x)}	
	\end{equation}
	where $\phi \colon \mathbb{C} \to A_\N \tensor B$ ranges over states exchangeable in the first factor. 
	For $B=\mathbb{C}$, the right-hand side is $\norm{[x]}_{QA}=\norm{x}_{QA}$ because $\opm{\ell_A}{(\N)}$ is the quotient map.
	Hence in this case, \eqref{eq:claim} holds with equality by the definition of $\norm{\_}_{QA}$.
	In general, by another application of Bauer's maximum principle and the quantum de Finetti theorem for states (\cref{thm:qdf_states}), 
	we get
	\[
		\sup_{\phi} \, \abs{\opm{\phi}{}(x)} = \sup \, \Set{  \abs{\opm{\pairing{\psi^{(\N)}}{\omega}}{} (x)} \given \psi \colon \mathbb{C} \to A \textrm{ and } \omega \colon \mathbb{C} \to B \textrm{ are states}}.
	\]
	The case $B=\mathbb{C}$ treated above implies that any power state $\psi^{(\N)}$ factors through $\ell^{(\N)}$,
	so that we get a state $\mu$ on $QA$ with $\psi^{(\N)} = \ell^{(\N)}\mu$. 
	In particular, $\pairing{\psi^{(\N)}}{\omega}= (\ell^{(\N)} \tensor \id_B)\comp \pairing{\mu}{\omega}$. An application of 
	\cref{cor:norm_state} is then enough to conclude that~\eqref{eq:claim}
	holds.
	Hence the required unique factorization of the given $\rho \colon C \to A_\N \tensor B$ across with positive unital $\mu \colon C \to A \tensor B$ indeed follows,
	and we also note that this holds by the same argument for merely self-adjoint unital $\rho$ with self-adjoint unital $\mu$.

	This proves existence and uniqueness for positive unital maps.
	If $\rho$ is completely positive, then applying the positive unital case to $\rho\tensor\id_D$ for an additional tensor factor $D$ shows,
	by uniqueness of the factorization $\mu$ among self-adjoint unital maps, that the factor obtained above tensors with $\id_D$ to a positive map.
	Taking $D=M_n$ for all $n$ shows complete positivity, and uniqueness in the completely positive setting follows from uniqueness in the positive-unital setting.
	Thus Property~\ref{it:df_limit} follows.
\end{proof}

The argument used in the proof of \cref{thm:qdf} shows that the factorization is in fact more general, i.e.~it holds for any $n$-positive\footnote{See e.g.~\cite[p.~26]{paulsen2002} for a definition.} unital map which is exchangeable in the first factor, for any choice of positive integer $n$. 
In particular, by \cref{prop:com_dom_codom} we obtain the following immediate consequence.

\begin{corollary}
	An exchangeable positive unital map $\opm{\phi}{}\colon \pre{A}_{\N} \rightsquigarrow \pre{B}$, where the tensor is with respect to either the maximal or the minimal tensor norm, is completely positive. 
\end{corollary}

Combining \cref{lem:autocomp,thm:qdf}, we also get the following. 

\begin{corollary}
	Given any two pre-C*-algebras $A$ and $B$, we have natural bijections
	\[
		\cpu_{\min/\max, \det}(B, SA) \cong \Set{\textrm{Autocompatible morphisms }B\to A}.
	\]
\end{corollary}

\newpage
\appendix
\section{Infinite tensor powers and permutation-invariant elements}
\label{sec:infinite_tensor_trivial}

In this appendix, we give a proof of the final sentence of \cref{rem:01laws}. 
To avoid overloading the notation, here we simply write $A$ instead of $\closed{A}$ to indicate a C*-algebra, in contrast to \cref{nota:completion}.

Let us consider $A_{J}$, the infinite tensor power of a C*-algebra $A$~\cite[Section~6.1]{murphy1990} by some infinite set $J$.
In our notation, this is precisely the completion of the Kolmogorov power, which in the main text we denote by $A_J$ (the use of the same notation should not cause confusion since here we are only considering C*-algebras).

In the following, it will not matter whether we take this infinite tensor product with respect to the minimal or the maximal C*-norm.
Following \cref{nota:sigma}, every finite permutation $\sigma$ induces a bounded $*$-homomorphism $\opm{A_\sigma}{} \colon A_{J} \rightsquigarrow A_J$ given by permuting the tensor factors.

\begin{proposition}\label{prop:no_01HSlaw}
	Let us consider any C*-algebra $A$, an infinite set $J$ and $x \in A_J$. If $\opm{A_\sigma}{}(x)=x$ for every finite permutation $\sigma \colon J \to J$, then $x= \lambda 1$ for some $\lambda \in \mathbb{C}$.
\end{proposition}

We thank MathOverflow users Diego Martinez and Caleb Eckhardt, whose comments were crucial for arriving at the following argument.\footnote{See \href{https://mathoverflow.net/questions/452405/is-there-an-element-in-an-infinite-tensor-power-of-a-c-algebra-that-is-invarian}{mathoverflow.net/questions/452405}.}

\begin{proof}
	The statement is trivial if one factor is zero, since then also $A_J \cong 0$, so assume otherwise.
	By definition of $A_J$, there is a sequence $x_n \in A_{F_n}\subset A_J$, where $F_n$ is a finite set, such that $x = \lim_n x_n$. 
	For clarity, let us consider $\epsilon_n>0$ such that $\norm{ x-x_n } \le \epsilon_n$ and $\epsilon_n \to 0$. 
	We pick a finite permutation $\sigma_n$ such that $\sigma_n(F_n)\cap F_n=\emptyset$. Then, since $\norm{\opm{A_{\sigma_n}}{}} \le 1$ by \cref{prop:pu_bounded},
	\[
		\norm{ x_n - {\opm{A_{\sigma_n}}{}}(x_n)} = \norm{x_n - x + {\opm{A_{\sigma_n}}{}} (x) - {\opm{A_{\sigma_n}}{}}(x_n)} \le \norm{x_n - x} + \norm{\opm{A_{\sigma_n}}{}} \norm{x-x_n} \le 2 \epsilon_n.
	\]
	Let us fix any state $\opm{\phi}{}\colon A_J \rightsquigarrow \mathbb{C}$, which is possible by $A_J \not\cong 0$.
	Then our assumptions also imply
	\[
		\norm{ \opm{\phi}{}({\opm{A_{\sigma_n}}{}}(x_n)) - \opm{\phi}{}(x)} \le \epsilon_n.
	\]
	Using the completely positive map
	\[
		\opm{\psi}{} \coloneqq \opm{\phi}{}_{\mid A_{\sigma_n(F_n)}} \otimes \id_{A_{J\setminus \sigma_n(F_n)}} \colon A_J \rightsquigarrow A_{J\setminus \sigma_n(F_n)},
	\]
	we obtain
	\[
		\norm{  x_n - \opm{\phi}{}({\opm{A_{\sigma_n}}{}}(x_n))1 }=\norm{ \opm{\psi}{} (x_n -{\opm{A_{\sigma_n}}{}}(x_n)) }\le\norm{ x_n -{\opm{A_{\sigma_n}}{}}(x_n) }\le 2\epsilon_n,
	\]
	where the first equation holds because both terms can be written with tensor factors of $1$ in all tensor factors that belong to $J \setminus F_n$.
	Putting everything together, we get
	\[
		\norm{ x- \opm{\phi}{}(x)1 } \le \norm{ x-x_n} + \norm{ x_n - \opm{\phi}{}({\opm{A_{\sigma_n}}{}}(x_n))1 } + \norm{ \opm{\phi}{}({\opm{A_{\sigma_n}}{}}(x_n))1 - \opm{\phi}{}(x)1 } \le 4 \epsilon_n \to 0,
	\]
	from which we conclude that $x=\opm{\phi}{}(x)1$.
\end{proof}

\section*{Declarations}
\paragraph{Funding} This research was funded in whole by the Austrian Science Fund FWF P 35992-N.
\paragraph{Competing Interests} none
\paragraph{Data Availability} NA
\paragraph{Author Contribution} Both the authors wrote and reviewed the manuscript.
\newpage
{\small \bibliographystyle{plain}
\bibliography{markov}}

\end{document}